\documentclass[10pt,regno]{amsart}
\usepackage{amsmath,amsthm,amsfonts,amssymb,amscd,overpic,mathrsfs}
\usepackage{euscript,graphicx,color}
\usepackage{enumerate}

\usepackage{indentfirst}
\usepackage{amsmath}
\usepackage[backref=page]{hyperref}  
\usepackage{graphicx}

\makeatletter

\def\myMRbibitem{\@ifnextchar[\my@lbibitem\my@bibitem}

\def\mybiblabel#1#2{\@biblabel{{\hyperref{http://www.ams.org/mathscinet-getitem?mr=#1}{}{}{#2}}}}

\def\myhyperanchor#1{\Hy@raisedlink{\hyper@anchorstart{cite.#1}\hyper@anchorend}}

\def\my@lbibitem[#1]#2#3#4\par{%
    \item[\mybiblabel{#2}{#1}\myhyperanchor{#3}\hfill]#4%
    \@ifundefined{ifbackrefparscan}{}{\BR@backref{#3}}%
    \if@filesw{\let\protect\noexpand\immediate
       \write\@auxout{\string\bibcite{#3}{#1}}}\fi\ignorespaces%
}

\def\my@bibitem#1#2#3\par{%
    \refstepcounter\@listctr
    \item[\mybiblabel{#1}{\the\value\@listctr}\myhyperanchor{#2}\hfill]#3%
    \@ifundefined{ifbackrefparscan}{}{\BR@backref{#2}}%
    \if@filesw\immediate\write\@auxout
        {\string\bibcite{#2}{\the\value\@listctr}}\fi\ignorespaces%
}

\makeatother

\renewcommand*{\backref}[1]{}
\renewcommand*{\backrefalt}[4]{\quad \tiny 
    \ifcase #1 (Not cited.)%
    \or        (Cited on page~#2.)%
    \else      (Cited on pages~#2.)%
    \fi}

\theoremstyle{plain}
	\newtheorem{theo}{Theorem}
	\newtheorem{mcor}[theo]{Corollary}
	\newtheorem{mprop}[theo]{Proposition}

	\newtheorem{lemm}{Lemma}[section]
	\newtheorem{theor}[lemm]{Theorem}
	
	\newtheorem{prop}[lemm]{Proposition}
	\newtheorem{defi}[lemm]{Definition}
	
	\newtheorem{adde}[lemm]{Addendum}
\theoremstyle{definition}
	\newtheorem{clai}[lemm]{Claim}
	\newtheorem{fact}[lemm]{Fact}

         \newtheorem{mrem}[theo]{Remark}
         \newtheorem{rema}[lemm]{Remark}

\newtheorem*{Remark*}{Remark}

\def\MM{{\mathbb M}} \def\NN{{\mathbb N}}  
 \def\RR{{\mathbb R}}  \def\TT{{\mathbb T}}
   
 \def\ZZ{{\mathbb Z}}

\def\Si{\Sigma}
\def\La{\Lambda}

\def\De{\Delta}

\def\Ga{\Gamma}

\newcommand{\cC}{\mathcal{C}}\newcommand{\cD}{\mathcal{D}}
\newcommand{\cF}{\mathcal{F}}
\newcommand{\cI}{\mathcal{I}}\newcommand{\cJ}{\mathcal{J}}
\newcommand{\cM}{\mathcal{M}}\newcommand{\cO}{\mathcal{O}}\newcommand{\cP}{\mathcal{P}}
\newcommand{\cR}{\mathcal{R}}\newcommand{\cT}{\mathcal{T}}
\newcommand{\cU}{\mathcal{U}}\newcommand{\cV}{\mathcal{V}}\newcommand{\cW}{\mathcal{W}}
\newcommand{\cZ}{\mathcal{Z}}

\newcommand{\fD}{\mathfrak{D}}
\newcommand{\fF}{\mathfrak{F}}
\newcommand{\fP}{\mathfrak{P}}
\newcommand{\fd}{\mathfrak{d}}

\def\diff{\operatorname{Diff}}
\def\Diff{\operatorname{Diff}}

\def\dim{\operatorname{dim}}

\def\La{\Lambda}

\numberwithin{equation}{section}         

\newcommand{\mru}{{\mathrm{u}}}
\newcommand{\mrs}{{\mathrm{s}}}
\newcommand{\mruu}{{\mathrm{uu}}}
\newcommand{\mrss}{{\mathrm{ss}}}
\newcommand{\mrc}{{\mathrm{c}}}

\renewcommand{\setminus}{\smallsetminus}
\renewcommand{\emptyset}{\varnothing}

\newcommand{\eqdef}{\stackrel{\scriptscriptstyle\rm def}{=}}

\begin{document}

\title[A criterion for zero averages and support of ergodic measures]{A criterion for zero averages and full support of ergodic measures}
\date{Jan 26, 2012}

\author[Ch.~Bonatti]{Christian Bonatti}
\address{Institut de Math\'ematiques de Bourgogne}
\email{bonatti@u-bourgogne.fr}

\author[L.~J.~D\'\i az]{Lorenzo J.~D\'\i az}
\address{Departamento de Matem\'atica, Pontif\'{\i}cia Universidade Cat\'olica do Rio de Janeiro} 
\email{lodiaz@mat.puc-rio.br}

\author[J.~Bochi]{Jairo Bochi}
\address{Facultad de Matem\'aticas, Pontificia Universidad Cat\'olica de Chile}
\email{jairo.bochi@mat.puc.cl}

\begin{abstract}
 Consider  a homeomorphism $f$ defined on a compact metric space $X$
 and a continuous map 
 $\phi\colon X \to \RR$. We  provide an abstract criterion, called \emph{control at any scale with
 a long sparse tail} for a point $x\in X$ and the map $\phi$,  that guarantees that any
 weak$\ast$ limit measure $\mu$ of the 
 Birkhoff average of Dirac measures $\frac1n\sum_0^{n-1}\delta(f^i(x))$ is such that
 $\mu$-almost every point $y$ has a dense orbit in $X$ and the
 Birkhoff average of $\phi$ along the orbit of $y$ is zero. 
 
As an illustration of the strength of this criterion, we prove 
that the diffeomorphisms with nonhyperbolic ergodic measures form a 
$C^1$-open and dense subset
of the set of robustly transitive partially hyperbolic diffeomorphisms with one dimensional
nonhyperbolic central direction.
We also obtain applications for nonhyperbolic homoclinic classes.
\end{abstract}

\begin{thanks}{This research has been supported  [in part] by CAPES - Ci\^encia sem fronteiras,
CNE-Faperj, and CNPq-grants (Brazil),
EU Marie-Curie IRSES ``Brazilian-European partnership in Dynamical
Systems" (FP7-PEOPLE-2012-IRSES 318999 BREUDS), and 
Fondecyt project 1140202 (Chile).
The authors  acknowledge the hospitality of PUC-Rio and IMB. 
LJD thanks the hospitality and support of
ICERM - Brown University during the thematic semester 	
``Fractal Geometry, Hyperbolic Dynamics, and Thermodynamical Formalism".}
\end{thanks}

\maketitle



\section{Introduction}

\subsection{Motivation and general setting}

This work is a part of a long-term project to attack the following general  
problem which  rephrases the opening question in \cite{GIKN} from a different perspective:
\emph{To what extent does  ergodic theory detect the nonhyperbolicity of a dynamical system?}

More precisely, we say that a  diffeomorphism $f$ is \emph{nonhyperbolic} if its non-wandering set is nonhyperbolic. 
We aim to know if such $f$ possesses \emph{nonhyperbolic ergodic measures} 
(i.e.  with some zero  \emph{Lyapunov exponent}) and if some of them fully reflect the nonhyperbolic
behaviour of $f$. For instance, we would like to know
\begin{itemize}
 \item what is their support,
 \item what is their entropy, and 
  \item how many Lyapunov exponents of the measures are zero.
\end{itemize}

In this generality, the answer to this question is negative. There are simple examples of (even analytic)  nonhyperbolic dynamical systems whose invariant measures 
are all hyperbolic and
even with Lyapunov exponents uniformly far from zero, see
for instance the logistic map $t\mapsto 4t(1-t)$ or the
surgery examples in \cite{BBS} where  a saddle of a uniformly hyperbolic set is replaced 
by non-uniformly hyperbolic sets, among others (more examples of different nature can be found in \cite{CLR,LOR}).
Nevertheless, these examples are very specific and fragile. Thus,  
one hopes that the ``great majority" of nonhyperbolic systems have nonhyperbolic ergodic measures 
which detect and truly 
reflect the nonhyperbolic behaviour of the dynamics. 
 
 Concerning this sort of questions,
a first wave of results, initiated with \cite{GIKN}, continued in  \cite{DG,BDG}, and culminated in 
\cite{CCGWY}, 
  show that the existence of nonhyperbolic
ergodic  measures for nonhyperbolic dynamical systems is quite general in the $C^1$-setting: 
for $C^1$-generic diffeomorphisms, every nonhyperbolic  homoclinic class 
supports a 
nonhyperbolic ergodic measure, furthermore under quite natural hypotheses  the support of the measure is the whole homoclinic class\footnote{See \cite[Main Theorem]{CCGWY} and also \cite[Theorem B and Proposition 1.1]{CCGWY}.
This last result states that the support of the nonhyperbolic measure 
in a nonhyperbolic homoclinic class of a saddle
is the whole homoclinic class. This result requires 
neither that the stable/unstable splitting of the saddle extends to 
a dominated splitting on the class (compare with \cite{BDG})  nor 
that the homoclinic  class contains saddles of different type of hyperbolicity (compare with \cite{DG}).}.

Given a periodic point $p$  of a diffeomorphism $f$ denote by $\mu_{\cO(p)}$ the 
unique $f$-invariant measure supported on the orbit of $p$. We say that such a measure
is {\emph{periodic}}. 
The previous works follow 
the strategy of periodic approximations in \cite{GIKN} for constructing a nonhyperbolic  ergodic measure
as weak$\ast$ limits of periodic measures $\mu_{\cO(p_n)}$ supported on orbits $\cO(p_n)$ of hyperbolic periodic points $p_n$ 
 with decreasing ``amount of hyperbolicity".
 The main difficulty  is to obtain the  ergodicity of the limit measure.
\cite{GIKN}  provides a criterion 
for ergodicity summarised in rough terms as follows. 
Each periodic orbit $\mathcal{O}(p_n)$ consists of two parts: a ``shadowing part"
where $\mathcal{O}(p_n)$ closely shadows the previous orbit $\mathcal{O}(p_{n-1})$ and a
``tail" where the orbit is far from the previous one. 
To get an ergodic limit measure one needs some balance between the ``shadowing" and the ``tail"
parts of the orbits.
The ``tail part" is used to decrease the amount of hyperbolicity of a given Lyapunov exponent 
(see \cite{GIKN})
and also  to spread the support of the limit measure, (see \cite{BDG}).

Nonhyperbolic measures seem very fragile as small perturbations may turn the zero Lyapunov exponent into a nonzero one. However, in  \cite{KN} 
there are obtained (using the method  in \cite{GIKN}) certain 
$C^1$-open sets of diffeomorphisms 
having nonhyperbolic ergodic measures.
Bearing this result  in mind, it is 
natural to ask if the existence of nonhyperbolic measures is a $C^1$-open and dense property in the 
space  of nonhyperbolic diffeomorphisms. 
In this direction, \cite[Theorem 4]{BoBoDi2} formulates an abstract criterion called
\emph{control at any scale}\footnote{This construction also involves the so-called {\emph{flip-flop families,}}
we will review these notions below as they play an important role in our constructions.}
that leads to the following result (see \cite[Theorems 1 and 3]{BoBoDi2}):
{\emph{The $C^1$-interior of the set of diffeomorphisms having a nonhyperbolic ergodic measure 
contains an open and  dense subset of  the set of  $C^1$-diffeomorphisms having a pair of hyperbolic  periodic points 
of different indices robustly in the same chain recurrence class.}}

The method in \cite{BoBoDi2} 
 provides  a partially hyperbolic 
invariant set with positive topological entropy whose central Lyapunov exponent vanishes uniformly. 
This set only supports nonhyperbolic measures and the existence of a measure with positive entropy is a consequence of  the variational principle for entropy~\cite{walters}. A con of this method is that 
the ``completely" nonhyperbolic nature of the (obtained) set where a Lyapunov exponent  vanishes uniformly prevents the measures to have full support in nonhyperbolic chain recurrence classes.
This shows that, in some sense, the criterion in \cite{BoBoDi2} may be ``too demanding" and ``rigid". 

The aim of  this paper is to
introduce a new criterion that
 relaxes the ``control at any scale criterion" and allows to get nonhyperbolic measures
 with ``full support" (in the appropriate
ambient space: homoclinic class, chain recurrence class, the whole manifold, according to the case).
To be a bit more precise, given a point $x$ and a diffeomorphism $f$
consider the  {\emph{empirical measures}}
$\mu_n(x)$, $n\in \NN$, associated to $x$ defined as  the averages of the Dirac measures $\delta(f^i(x))$ in the orbit segment $\{x,\dots,f^{n-1}(x)\}$,
\begin{equation}
\label{e.empiricalmeasure}
\mu_n(x)\eqdef \frac{1}{n}\, \sum_{i=0}^{n-1} \delta (f^i(x)).
\end{equation}
The  criterion in this paper,
called {\emph{control at any scale with a long sparse tail}}
with respect to a continuous map $\varphi$ of a point $x$,
allows 
to construct ergodic measures with full support (in the appropriate ambient space) and a prescribed 
average with respect to $\varphi$, see Theorem~\ref{t.accumulation}.
This construction
involves two main aspects of different nature: 
density of the orbits of  $\mu$-generic points and
control of averages. The existence of ergodic measures satisfying both properties is a consequence of 
the construction.

A specially interesting case
occurs  when the map $\varphi$
 is the derivative of a diffeomorphism
with respect to a continuous one-dimensional center direction (taking positive and negative values).
In such a case we get that every
measure $\mu$ that is a weak$\ast$ limit of a sequence 
of empirical measures of $x$ is such that
 $\mu$-almost every point
has a zero Lyapunov exponent and 
a dense orbit (in the corresponding ambient space), see Theorems~\ref{t.cycle} and \ref{t.ctail}.

To state more precisely the dynamical  consequences of  the criterion 
let us introduce some notation
(the precise definitions can be found below). 
In what follows we consider a boundaryless Riemannian compact manifold $M$ 
and the following two $C^1$-open
subsets of diffeomorphisms:
\begin{itemize}
\item 
The set  $\cR\cT(M)$  of all {\emph{robustly transitive diffeomorphisms}}\footnote{A diffeomorphism
is called transitive if it has a dense orbit. The diffeomorphism is $C^1$-robustly transitive if
$C^1$-nearby diffeomorphisms are also transitive.}
 with  a 
partially hyperbolic splitting  with one-dimensional (nonhyperbolic) center,
\item 
The set $\cZ(M)$ defined as the $C^1$-interior of the set of $C^1$-diffeomorphisms 
having a nonhyperbolic ergodic measure with full support in  $M$. 
\end{itemize}

As an application of our criterion we get that set $\cZ(M)\cap \cR\cT(M)$ is $C^1$-open and 
$C^1$-dense
in $\cR\cT(M)$, see Theorem~\ref{t.c.openanddense}.
We also get semi-local versions of this result formulated in terms of
nonhyperbolic homoclinic classes or/and chain recurrence classes, see Theorems~\ref{t.cycle} and~\ref{t.ctail}.
These results turn the $C^1$-generic statements in \cite{BDG} into  $C^1$-open and $C^1$-dense ones. 
We observe that a similar result involving  different methods was announced
in \cite{BJ}\footnote{\label{fn}The construction in \cite{BJ} combines  the criteria of periodic approximations
 in \cite{GIKN}  and of the control 
at any scale in \cite{BoBoDi2} and  
a shadowing lemma by Gan-Liao, \cite{G}.}.
Applications of the criterion in hyperbolic-like contexts,
as for instance full shifts and horseshoes, are discussed in Section~\ref{ss.medias}.


In this paper
we restrict ourselves to the control of the support and the averages of the measures, omitting  
questions related to the entropy of these measures. Nevertheless it 
seems that our method is well suited to construct nonhyperbolic ergodic measures with positive entropy and full support. This is the next step  of an ongoing project whose ingredients  involve tools of a very different nature beyond the scope of this paper.

In the dynamical applications we
focus on partially hyperbolic diffeomorphisms with a one-dimensional center bundle
and therefore the measures may  have at most one zero Lyapunov exponent. Here we do not consider the case of
higher dimensional central bundles and the possible occurrence of multiple zero exponents.
Up to now, there are quite few results on multiple zero  Lyapunov exponents. The simultaneous
control of several exponents is much more difficult, essentially
due   to the non-commutativity  of $\mathrm{GL}(n,\RR)$ for $n>1$. We refer to 
\cite{BoBoDi} for examples of ($C^1$ and $C^2$) robust existence  of ergodic measures with multiple zero exponents in the context of 
iterated function systems. Recently, \cite{WZ} announces the locally $C^1$-generic vanishing of several Lyapunov exponents
in homoclinic classes of diffeomorphisms.

We now  describe our methods and results in  a more detailed way.

\subsection{A criterion for controlling averages of continuous maps}
\label{ss.abstractcriterion}
Consider a compact metric space $(X,d)$, a homeomorphism $f$ defined on $X$, and a continuous map
$\varphi\colon X \to \RR$.
Given a point $x\in X$ 
consider the set of {\emph{empirical measures}}
$\mu_n(x)$ associated to $x$ defined as  in \eqref{e.empiricalmeasure}.
Consider the following notation for finite Birkhoff averages of $\varphi$,
\begin{equation}\label{e.birkhoffaverages}
\varphi_n(x) \eqdef
 \frac{1}{n}\, \sum_{i=0}^{n-1} \varphi (f^i(x)),
\end{equation}
and limit averages of $\varphi$
\begin{equation}
\label{e.notationBirkhoff}
\varphi_\infty (x) \eqdef
\lim_{n\to+\infty}\
\varphi_n (x)=
\lim_{n\to+\infty}\frac 1n\sum_{i=0}^{n-1}\varphi(f^i(x)),
\end{equation}
if such a limit exists.

Consider a measure $\mu$
that is a weak$\ast$ limit of empirical measures of $x$
 and  a subsequence  $\mu_{n_k}(x)$
with  $\mu_{n_k}(x)\to \mu$ in the weak$\ast$ topology.
The convergence of the sequence of Birkhoff averages $\int \varphi \, d\mu_{n_k}(x)$  to some limit 
$\alpha$ implies that $\int \varphi\, d\mu=\alpha$. But since 
$\mu$ may be non-ergodic this does not provide 
any information about the 
Birkhoff averages $\varphi_n(y)$
 of  $\mu$-generic points $y$. 
We aim for a criterion guaranteeing that  $\mu$-generic points have the same limit average as $x$.
Naively, in \cite{BoBoDi2} the way to get this property 
is to require that ``all large orbit intervals of the forward orbit of $x$ have average close to the 
limit average (say) $\alpha$".  
This was formalised in the
criterion  \emph{control of Birkhoff averages at any scale} of a point $x$ 
with respect to a map $\varphi$ in \cite{BoBoDi2}. This criterion implies 
that there are sequences of times $t_n\to\infty$ and of ``errors" $\varepsilon_n\to0$ such
that every  orbit interval with length $t\ge t_n$ of the forward orbit of $x$ 
has  $\varphi$-Birkhoff average in $[\alpha-\varepsilon_n, \alpha+\varepsilon_n]$.
When $\varphi$-Birkhoff averages are controlled at any scale
  then the
  $\varphi$-Birkhoff averages of any
    $\omega$-limit point  of $x$ 
  converge uniformly to $\alpha$ 
(see \cite[Lemma 2.2]{BoBoDi2}).

%
%

To get  a limit measure  whose support is the whole
ambient space the requirement 
``all long orbit intervals
satisfy the limit average property" is extremely restrictive. Roughly, in the criterion in this paper we 
only require that  ``most of large orbit intervals of the forward orbit of $x$ have average close to the 
limit average $\alpha$".  
Let us explain a little more precisely this rough idea.

If the limit measure has full support then the orbit of the point
$x$ must
necessarily visit ``all regions" of the ambient space and 
these visits require an arbitrary
large time. Moreover, 
 to get  limit measures 
whose generic points 
have dense orbits in the ambient space
 these ``long visits" must occur  with some frequency. 
 During these long visits the control of the averages can be lost.

To control simultaneously Birkhoff averages and support of the limit measure, one needs some
``balance" between the part 
of the orbit where there is a ``good control of 
the averages" and the part of the orbit used for spreading the support  of the measure to get its density
(roughly, these parts play the roles of the ``shadowing" and ``tail parts" of the method in \cite{GIKN}).
The criterion in this paper formalizes an abstract notion for this balance  that we call 
\emph{control at any scale with a long sparse tail with respect to $\varphi$ and $X$} 
(see Definitions~\ref{d.alphacontrol} and \ref{d.controledtail}). 
Our main technical result is that this criterion provides ergodic measures  having simultaneously a prescribed average and a prescribed support.
 
\begin{theo}\label{t.accumulation}
Let $(X,d)$ be a compact metric space, $f\colon X\to X$ a homeomorphism, and 
 $\varphi \colon X\to \RR$
a continuous map.
Consider
\begin{itemize}
\item
 a point $x_0\in X$
that  is controlled at any scale with a long sparse tail with respect $\varphi$ and $X$
 and 
 \item
 a measure $\mu$
that is a weak$\ast$ limit of the sequence of  empirical measures $(\mu_n(x_0))_n$ of $x_0$.
\end{itemize}
Then for $\mu$-almost every point 
$x$ 
the following holds:
\begin{enumerate}
\item 
the forward orbit of $x$ for $f$ is dense in $X$ and
\item
$\lim_{n\to\infty}  \frac{1}{n} \sum_{i=0}^{n-1} \varphi (f^i(x))=
\int \varphi \, d\mu$.
\end{enumerate}
In particular, these two assertions hold for almost every ergodic measure of the ergodic decomposition of
$\mu$.
\end{theo}

We now exhibit some dynamical configurations where the criterion  holds. 
Indeed,
we see that such configurations are quite ``frequent".

\subsection{Flip-flop families with sojourns: control at any scale with a long sparse tail}
\label{ss.flipflps}
To present a mechanism providing orbits controlled at any scale
we borrow
the following definition from \cite{BoBoDi2}:

\begin{defi}[Flip-flop family]\label{d.flipflop} 
Let $(X,d)$ be a compact metric space, $f\colon X \to X$ a homeomorphism, and $\varphi\colon X \to \mathbb{R}$
a continuous function.

A \emph{flip-flop family} associated to $\varphi$ and $f$ is a family 
$\fF=\fF^+\bigsqcup\fF^-$ of compact subsets of $X$ such that there are $\alpha>0$ and a 
sequence of numbers  $(\zeta_n)_n$, $\zeta_n>0$ and $\zeta_n\to 0$ as $n\to \infty$, 
such that:
\begin{enumerate}
 \item\label{i.flipflop1} for every $D\in\fF^+$ (resp. $D\in\fF^-$) and every $x\in D$ it holds $\varphi(x)\geq\alpha$ (resp. $\varphi(x)\leq -\alpha$);
 \item\label{i.flipflop2}  
 for every $D\in \fF$, there are sets $D^+\in \fF^+$ and $D^-\in\fF^-$ contained in $f(D)$;
 \item\label{i.flipflop33}   
 for every  $n>0 $ and  every family of sets $D_i\in \fF$, $i\in\{0,\dots ,n\}$  with $D_{i+1}\subset f(D_i)$
 it holds
  $$
 d(f^{n-i}(x),f^{n-i}(y))\leq\zeta_i\cdot d(f^n(x) f^n(y))
 $$
 for every $i\in\{0,\dots,n\}$ and every pair of points
 $x,y\in f^{-n}(D_n)$.
\end{enumerate}

We call {\emph{plaques}}\footnote{We pay special attention to the case when the sets of the flip-flop family
are discs tangent to a strong unstable cone field. This justifies this name.}
 the sets of the flip-flop family $\cF$.
\end{defi}

With the notation in Definition~\ref{d.flipflop},
\cite[Theorem 2.1]{BoBoDi2} claims  that for every number
$t\in (-\alpha,\alpha)$ and  every set  $D\in \fF$  
there is a point $x_t\in D$ 
whose orbit is controlled at any scale for the function 
$\varphi_t=\varphi-t$. Hence the Birkhoff average of $\varphi$ along the orbit of any point $y\in \omega(x_t)$ is $t$.
Furthermore, the $\omega$-limit set of $x_t$ has positive topological entropy. 

Since we  aim to obtain measures with full support
we need to  relax the control of the averages. For that 
we introduce a ``sojourn condition" 
for the returns of the sets of the flip-flop family
(item~\eqref{i.defff0} in the definition below).
These``sojourns" will   be used to get
dense orbits and to spread the support of the measures
and play a role similar to the ``tails" in \cite{GIKN}. 


\begin{defi}[Flip-flop family with sojourns]\label{d.flipfloptail} 
Let $(X,d)$ be a compact metric space,
$Y$ a compact subset of $X$, $f\colon X \to X$ a homeomorphism, and $\varphi\colon X \to \mathbb{R}$
a continuous function.

Consider a  flip-flop  family
 $\fF=\fF^+\bigsqcup\fF^-$
  associated to $\varphi$ and $f$.  
We say that the \emph{flip-flop family $\fF$ has sojourns along $Y$} 
(or that  \emph{$\fF$ sojourns along $Y$}) if
for every $\delta>0$ there is an integer $N=N_\delta$ such that 
every plaque $D\in\fF$ contains subsets $\widehat D^+, \widehat D^-$ such that: 
\begin{enumerate}
 \item\label{i.defff0}  
 for every $x\in \widehat D^+\cup \widehat D^-$ the orbit segment $\{x,\dots, f^N(x)\}$ is $\delta$-dense in $Y$ 
 (i.e., the $\delta$-neighbourhood of the orbit segment contains $Y$); 
 \item\label{i.defff1} 
 $f^N(\widehat D^+)\in \fF^+$ and $f^N(\widehat D^-)\in\fF^-$;
 \item\label{i.defff2}   
 for every $i\in\{0,\dots, N\}$ and  every pair of points $x,y\in \widehat D^+$  or $x, y\in \widehat D^-$  
 it holds 
 $$ 
 d(f^{N-i}(x),f^{N-i}(y))\leq\zeta_i\cdot d(f^N(x) f^N(y)), 
 $$
 where $(\zeta_i)_i$ is a sequence  as in Definition~\ref{d.flipflop}.
\end{enumerate}
\end{defi}

The conditions in Definition~\ref{d.flipfloptail} are depicted in Figure~\ref{f.ffs}.
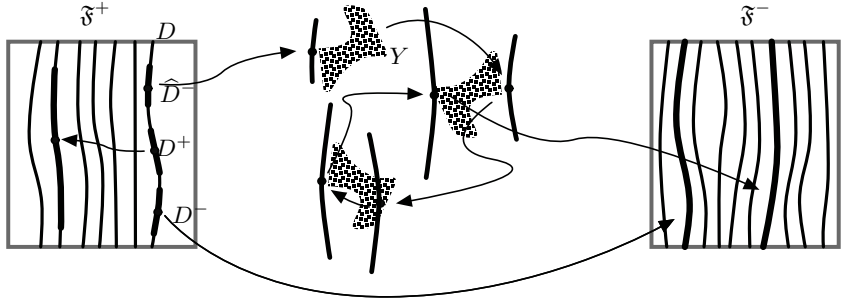
\begin{figure}
\begin{minipage}[h]{\linewidth}
\centering
 \begin{overpic}[scale=.60, 
 ]{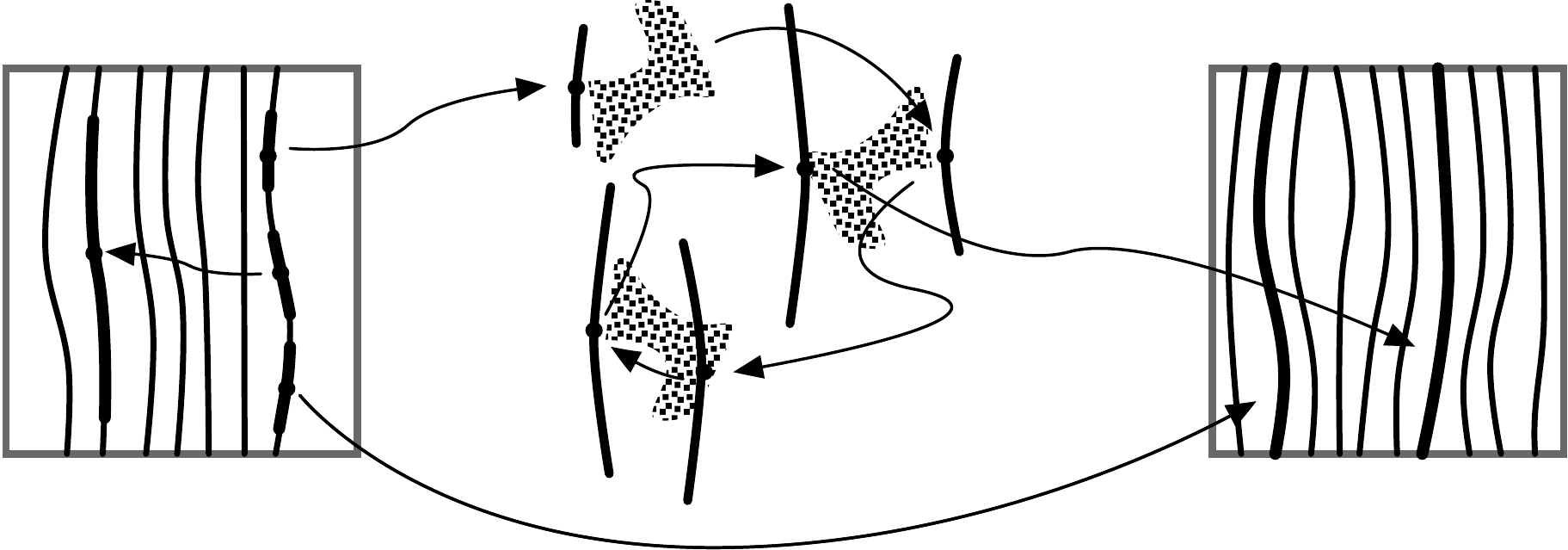}
			\put(9,33){\small$\fF^+$}
					\put(88,33){\small$\fF^-$}
			\put(18,31){\small$D$}
\put(18,17){\small$D^+$}
\put(20,9){\small$D^-$}
\put(18.5,23.5){\small$\widehat D^-$}
\put(46,28){\small$Y$}
  \end{overpic}
\caption{A flip-flop family with sojourns: the plaques $D^+$, $D^-$, and $\widehat D^-$
(the plaque 
$\widehat D^+$ is omitted for visual simplicity).
}
\label{f.ffs}
\end{minipage}
\end{figure}	

Next theorem corresponds to \cite[Theorem 2.1]{BoBoDi2} in our setting:

\begin{theo}\label{t.flipfloptail} 
Let $(X,d)$ be a compact metric space,
$Y$ a compact subset of $X$, $f\colon X \to X$ a homeomorphism, and $\varphi\colon X \to \mathbb{R}$
a continuous function.
Consider  a flip-flop family $\fF$  associated to $\varphi$ and $f$  having
sojourns along $Y$. 

Then every plaque $D\in \fF$ contains a point $x_D\in D$ that is 
controlled at any scale 
with a long sparse tail 
with respect to $\varphi$ and $Y$.
\end{theo}

As a corollary of Theorems~\ref{t.accumulation} and \ref{t.flipfloptail}  we get
(recall the notation for Birkhoff limits in \eqref{e.notationBirkhoff}):

\begin{mcor}\label{c.flipfloptail} 
Under the hypotheses of  Theorem~\ref{t.flipfloptail} and with the same notation,  
any measure $\mu$ that is 
 a weak$\ast$ limit of the empirical measures
$(\mu_n(x_D))_n$
satisfies the following properties:
\begin{itemize}
 \item 
 the orbit of $\mu$-almost every point is dense in $Y$ and
 \item 
 for $\mu$-almost every point $x$ it holds
 $\varphi_\infty (x)=0$.  
\end{itemize}
As a consequence, almost every  measure $\nu$ in the ergodic decomposition of $\mu$ has full support 
in $Y$ and satisfies $\int\varphi\, d\nu=0$.
\end{mcor}

We  now explore some consequences of the results above.

\subsection{Birkhoff averages  in homoclinic classes}\label{ss.medias}
An important property  of our methods is that they 
can be used  in nonhyperbolic and non-Markovian settings.  
We now present two
applications of our criteria 
in the ``hyperbolic" setting of a  mixing sub-shift of finite type that are, as far as we are aware, unknown.
The key point of Proposition~\ref{p.r.hyperboliclike} is that it only requires continuity of the
potential $\varphi$.
When the potential is H\"older continuous this sort of result is well-known\footnote{For instance, techniques from multifractal analysis provide the following:
Given  a H\"older continuous function $\varphi$, there is a parametrised family of Gibbs states $\mu_t$, $t\in(\alpha,\beta)$, where $\alpha,\beta$ are as above, such that  $\int \varphi \,d\mu_t=t$. Each
 $\mu_t$ has full support  and positive entropy. The conclusion in this statement 
 is stronger than the than the one in b) as it guarantees also positive entropy.
For a survey of this topic see for instance \cite{PW}.}.

\begin{mprop}\label{p.r.hyperboliclike}
Let $\sigma\colon \Si\to\Si$ be a mixing sub-shift of finite type and  $\varphi\colon\Si\to\RR$
 a continuous function.
Let $\alpha$ and $\beta$ be the infimum and maximum, respectively, of $\int \varphi\, d\mu$ over the set of
$\sigma$-invariant probability measures $\mu$ (or equivalently 
of the Birkhoff averages along periodic orbits).
Then   for every $t\in(\alpha,\beta)$ the following holds:
\begin{enumerate}
\item {\em (Application of the criterion in \cite{BoBoDi2})}
There is a $\sigma$-invariant compact set $K_t$ 
with positive topological entropy
such that   
 the Birkhoff average of $\varphi$ along the orbit of any point in $K_t$ is $t$. 
\item {\em (Application of  the new criterion)}
There is an ergodic measure $\mu_t$ with full  support in  $\Si$ such that  $\int \varphi\, d\mu_t=t$. 
\end{enumerate}
\end{mprop}

This  proposition deals with systems satisfying
 \emph{specification properties}. An important property of our two criteria is 
that they do not involve and do not depend on specification-like properties.
Indeed,
they are introduced to control  averages of functions in
partially hyperbolic settings where specification fails. 
We now present an application of our criterion in settings  without specification properties. 

In what follows let  $M$ be a boundaryless compact Riemannian manifold and
$\Diff^1(M)$ the space of $C^1$-diffeomorphisms endowed with the standard uniform topology.
The {\emph{homoclinic class}} of  a hyperbolic periodic point $q$ of a diffeomorphism $f\in\Diff^1(M)$, denoted by $H(q,f)$, 
is the closure of the set of  transverse intersection points of the stable and unstable 
manifolds of the  orbit of $q$. 
Two hyperbolic periodic  points $p$ and $q$ of $f$ are \emph{homoclinically related} if the stable and unstable manifolds
of their orbits  intersect cyclically and transversely.
The homoclinic class of $q$ can  also be defined as the closure of the periodic points of $f$ that 
are homoclinically related to $q$. 
A homoclinic class is a \emph{transitive set} (existence of a dense orbit) 
whose periodic points  form a dense subset of it. 
Homoclinic classes are in many cases  the ``elementary pieces of the dynamics" of a diffeomorphism and are  used to structure  its dynamics, playing a similar role of the basic sets of the hyperbolic theory (indeed each basic set is a homoclinic class), for a discussion see the survey in \cite{B}.

The {\emph{$\mru$-index}} of a hyperbolic periodic point is the dimension of its unstable bundle. 
We analogously define  {\emph{$\mrs$-index}}. Two saddles which are homoclinically related have necessarily the
same $\mru$- and $\mrs$-indices. However two saddles with \emph{different indices}
(it is not necessary to specify the index type) may be in the same homoclinic class. In such a case the class is necessarily nonhyperbolic. Indeed,  the property of a homoclinic class containing   saddles of different indices is a typical feature in the nonhyperbolic dynamics studied in this paper (see also \cite{Sh,Mda,BD-robtran}).

The next result is a generalisation of the second part of Proposition~\ref{p.r.hyperboliclike} 
to a non-necessarily hyperbolic context, observe that we do not require hyperbolicity of the
homoclinic class.
Recall that if $p$ is a periodic point of $f$ we denote by
$\mu_{\cO(p)}$ the $f$-invariant probability supported on the orbit of $p$.

\begin{theo}\label{t.homoclinic} 
Let $f\colon M\to M$ be a $C^1$-diffeomorphism defined on a boun\-dary\-less compact manifold 
and $\varphi\colon M\to \RR$ 
a continuous function. Consider a pair of hyperbolic periodic points $p$ and $q$ of $f$ 
that are homoclinically related and satisfy
$$
a_p\eqdef\int \varphi\, d\mu_{\cO(p)}<\int \varphi \,d\mu_{\cO(q)} \eqdef a_q.
$$
Then for every $t\in(a_p,a_q)$ there is an ergodic measure $\mu_t$ whose support is the whole homoclinic class $H(p,f)=H(q,f)$  and satisfies $\int\varphi\, d\mu_t=t.$
\end{theo}

Note that the hypotheses in the theorem are $C^1$-open.
Observe 
that the difficulty in  the theorem is to get simultaneously the three properties 
\emph{ergodicity}, \emph{prescribed average}, and \emph{full support}. 
It is easier (and also known) to build 
measures satisfying simultaneously only two of these properties.

We also aim to apply the criterion in Theorem~\ref{t.accumulation} to saddles $p$ and $q$ 
that have different indices
and are in the same homoclinic class 
(or, more generally,  chain recurrence class) and thus the saddles are not homoclinically related. 

Before stating the next corollary let us recall the definition of a chain recurrence class.
Given $\epsilon>0$, a finite sequence of points $(x_i)_{i=0}^n$ is an
{\emph{$\epsilon$-pseudo-orbit}} of a diffeomorphism 
$f$ if
$d(f(x_i),x_{i+1})<\epsilon$ for every $i=0,\dots,n-1$ (here $d$ denotes the distance in $M$). A
point $x$ is {\emph{chain recurrent}} for $f$ if
 for every $\epsilon>0$ there is an $\epsilon$-pseudo-orbit
 $(x_i)_{i=0}^n$  with $x_0=x=x_n$. 
The {\emph{chain recurrent set}}  of $f$, denoted by $\cR(f)$, is the union of the chain recurrent points of $f$. 
The \emph{chain recurrence class} 
$C(x,f)$ of a point  $x\in \cR(f)$ is the set of points $y$ such that
for every $\epsilon>0$ there are $\epsilon$-pseudo-orbits joining
$x$ to $y$ and $y$ to $x$. 
Two chain recurrence classes are either disjoint or equal. Thus the 
set $\cR(f)$ is the union of pairwise disjoint chain recurrence classes. 
Let us observe that two points in the same homoclinic class 
are also in the same chain recurrence class 
(the converse is false in general, although $C^1$-generically homoclinic classes and chain recurrence classes of periodic points coincide, see \cite{BC}). Thus if $p$ is a hyperbolic periodic point then $H(p,f)\subseteq C(p,f)$.

\begin{mcor}\label{c.function}
Let $M$ be a boundaryless compact manifold and
$\cU$ be a $C^1$-open set in $\diff^1(M)$
 such that 
 every $f\in \cU$ has a pair of  hyperbolic periodic orbits $p_f$ and $q_f$ of different indices depending continuously on $f$ 
 whose chain recurrence classes are equal.  
 Let $\varphi\colon M\to\RR$ be a continuous function 
 such that
$$
\int \varphi\, d\mu_{\cO(p_f)}<0<\int \varphi \,d\mu_{\cO(q_f)}, \quad \mbox{for every $f\in \cU$}.
$$
Then there are two $C^1$-open sets $\cV_p$ and $\cV_q$  whose union is $C^1$-dense in $ \cU$ such that every $f\in\cV_p$ (resp. $f\in\cV_q$) has an ergodic measure
$\mu_f$ whose support is the homoclinic class  $H(p_f,f)$ (resp. $H(q_f,f)$) and satisfies
$\int \varphi\, d\mu_f=0$. 
\end{mcor}

Note that the saddles in the corollary cannot be  homoclinically related and hence Theorem
~\ref{t.homoclinic} cannot be applied.
We   bypass this difficulty by transferring the desired 
averages to pairs of homoclinically related periodic points (then the proof follows from 
Theorem~\ref{t.homoclinic}), 
see Section~\ref{ss.proofofcorollarycfunction} for the proof of the corollary.

\begin{rema}
\label{r.bdpr}
By \cite[Theorem E]{BDPR}, if in Corollary~\ref{c.function} we assume that the chain recurrence
 class 
is partially hyperbolic  with one-dimensional center 
(see definition below) then there is a $C^1$-open and dense subset $\cV$ of $\cU$  
such that $H(p_g,g)=H(q_g,g)$  for all $g\in \cV$. Without this extra hypothesis the equality of the homoclinic classes is only guaranteed for a residual subset of $\cU$, see \cite{BC}. 
\end{rema}

\subsection{Nonhyperbolic ergodic measures with full support}\label{ss.nonhyperbolicfull}

In what follows we focus on \emph{partially hyperbolic diffeomorphisms with one-dimensional center}. 
Our aim is to get results as above
when $\varphi$ is the ``logarithm of the center derivative". 
This will allow us to obtain nonhyperbolic ergodic measures  with large support in quite general nonhyperbolic settings.
Before going to the details we need some
definitions.

Given a diffeomorphism $f$ we say that a compact $f$-invariant set $\Lambda$ 
is \emph{partially hyperbolic with one-dimensional center}
if there is a $Df$-invariant dominated\footnote{A $Df$-invariant splitting $T_\La M=F\oplus E$ is \emph{dominated} if
there are constants $C>0$ and $\lambda<1$ such that $|| Df^{-n} F_{f^n(x)}||\, || Df^n E_x|| < C \lambda^n$
for all $x\in\Lambda$ and $n\in \NN$. In our case domination means that the bundles
$E^{\mathrm{uc} } \oplus E^{\mrss}$
and
$E^{\mruu} \oplus E^{\mathrm{cs}}$ are both dominated,
where $E^\mathrm{uc}= E^{\mruu} \oplus E^{\mrc}$ and
$E^\mathrm{cs}= E^{\mrc} \oplus E^{\mrss}$.}
splitting with three non-trivial bundles 
\begin{equation}\label{e.ph}
T_\Lambda M = E^{\mruu} \oplus E^{\mrc} \oplus E^{\mrss}
\end{equation}
such that 
$E^\mathrm{uu}$ is uniformly expanding,
	$E^{\mathrm{c}}$ has dimension~$1$, and
	$E^\mathrm{ss}$ is uniformly contracting. We say that 
	$E^{\mruu}$ and $E^{\mrss}$ are the \emph{strong unstable} and \emph{strong stable bundles}, 
	respectively,  and that
	$E^\mrc$ is the \emph{central bundle}. We denote by
$d^\mruu$ and $d^\mrss$  the dimensions of $E^{\mruu}$ and $E^{\mrss}$, respectively.

Given an ergodic measure $\mu$ of a diffeomorphism $f$
the  Oseledets' Theorem gives numbers 
$\chi_1(\mu) \ge \chi_2(\mu) \ge \cdots \ge \chi_d(\mu)$, the {\emph{Lyapunov exponents}},
and a $Df$-invariant splitting 
$E_1\oplus E_2\oplus \cdots \oplus E_d$,
the {\em{Oseledets' splitting,}}
where $d =\dim (M)$, with the following property:
for $\mu$-almost every point
$$
\lim_{n \to \pm \infty} \frac{\log \|Df^n_x (v_i)\|}{n} = \chi_i(\mu), \quad 
\mbox{for every $i$ and $v\in E_i\setminus \{\bar 0\}$}.
$$
If the measure is supported  on a partially hyperbolic set with one-dimensional center as above
then 
$$
E^{\mruu}=E_1\oplus \cdots \oplus E_{d^\mruu},
\quad
E^{\mrc}=E_{d^\mruu+1},
\quad
E^{\mrss}=E_{d^\mruu+2}\oplus \cdots  \oplus E_{d},
$$
and
$\chi_{d^\mruu}(\mu)>0> \chi_{d^\mruu+2}(\mu)$.
Let  $\chi_{d^\mruu+1}(\mu)\eqdef \chi_{\mrc}(\mu)$,  
we say that $\chi_{\mrc}(\mu)$ is the \emph{central exponent} of $\mu$.
In this partially hyperbolic setting the 
 \emph{logarithm of the center derivative}
map
\begin{equation}\label{e.logmap}
 \mathrm{J}_f^{\mrc}(x) \eqdef \log | Df_x |_{E^\mrc (x)}|
\end{equation}
is well defined and continuous, 
 therefore
the central Lyapunov exponent of the measure is given  by the integral
$$
 \chi_{\mrc}(\mu)= \int  \mathrm{J}_f^\mrc \, d\mu.
 $$
This equality allows to use the methods in the previous sections to construct and control nonhyperbolic ergodic measures.

%
%
%
%

Let us explain some relevant points of our study. 
 A  (new) difficulty, compared with Theorem~\ref{t.homoclinic},
 is that the logarithm of the center derivative $ \mathrm{J}_f^\mrc$
 cannot take values with different signs at homoclinically related periodic points
 (by definition, such points have the same indices and thus the sign of $ \mathrm{J}_f^\mrc$ is the same). 
 To recover this signal property we 
 consider  chain recurrence classes containing  saddles of different indices.

%

\begin{theo}\label{t.cycle} 
Let $M$ be a  boundaryless compact manifold
and $\cU$ a $C^1$-open set of $\diff^1(M)$
 such that  every $f\in \cU$ has hyperbolic periodic orbits $p_f$ and $q_f$ such that:
 \begin{itemize}
 \item
 they have different indices and
 depend continuously on $f\in \cU$,
 \item 
 their chain recurrence classes $C(p_f,f)$ and $C(q_f,f)$ are equal and have a partially hyperbolic splitting with 
 one-dimensional center. 
 \end{itemize}
Then there is a $C^1$-open and dense subset $\cV\subset \cU$ such  that every 
diffeomorphism $f\in\cV$ has a nonhyperbolic ergodic measure
$\mu_f$ 
 whose support
is the homoclinic class $H(p_f,f)=H(q_f,f)$.  
\end{theo}

Let us first observe that Theorem~\ref{t.cycle} can be rephrased in terms of robust cycles
instead of periodic points in the same chain recurrence class.
For that we need to review  the definition of a \emph{robust cycle}.
 Recall that a hyperbolic set $\Lambda_f$  of $f\in \Diff^1(M)$
 has a well  defined hyperbolic continuation $\La_g$ for every $g$ close to 
 $f$.
 Two transitive hyperbolic basic sets 
$\Lambda_f$ and $\Gamma_f$ of a diffeomorphism $f$ have 
a {\emph{$C^1$-robust (heterodimensional) cycle}} if these sets have
  different indices  and if there is a $C^1$-neighbourhood $\cU_f$ of $f$ 
  such that for every $g\in\cU$ the invariant sets of $\La_g$ and $\Ga_g$  intersect cyclically.
 As discussed
  in \cite{BoBoDi2}, the dynamical scenarios 
  of ``dynamics with $C^1$-robust cycles" and ``dynamics with chain recurrence classes containing $C^1$-robustly saddles of different indices"
  are essentially equivalent (they coincide in a $C^1$-open and dense subset of  $\diff^1(M)$).

We now describe explicitly 
the open and dense subset $\cV$ of  $\cU$ in Theorem~\ref{t.cycle}
using 
  \emph{dynamical blenders} and 
 \emph{flip-flop configurations} introduced in   \cite{BoBoDi2}, see Remark~\ref{r.described}.
Naively, a
 \emph{dynamical blender} is a hyperbolic and partially hyperbolic set together with a
 {\emph{strictly invariant
  family of
 discs}} (i.e., the image of any disc of the family contains another
   disc of the family)
almost tangent to its strong unstable direction,  see Definition~\ref{d.dynamicalblender}.
In very rough terms, a \emph{flip-flop configuration} 
 of a diffeomorphism $f$ and a continuous function $\varphi$ 
  is a $C^1$-robust  cycle associated to
a hyperbolic periodic point $q$  and
 a dynamical blender  $\La$ 
such that 
$\varphi$ is bigger than $\alpha>0$ in the blender $\Lambda$  and smaller than $-\alpha<0$ on the orbit of $q$.
Important properties of flip-flop configurations are their $C^1$-robustness, 
that they occur
$C^1$-open and densely in the set $\cU$ in Theorem~\ref{t.cycle}, and  
that they yield
flip-flop families. The latter allows to apply our criterion for zero averages.
The set $\cV$ in Theorem~\ref{t.cycle} is described in the remark below.

\begin{rema}[The set $\cV$ in Theorem~\ref{t.cycle}]
\label{r.described}
The set $\cV$ is the subset of $\cU$ of diffeomorphisms with flip-flop configurations ``containing" the saddle $q_g$.
\end{rema}

To state our next result recall that a \emph{filtrating region} of a diffeomorphism $f$ is the intersection of an attracting region and a repelling region of $f$. Let $U$ be  a filtrating region of $f$
 endowed with a strictly forward invariant unstable cone field of index $i$ and a 
 strictly backward invariant cone field of index $\dim(M)-i-1$, see  
 Section~\ref{ss.invariantconefields} for the precise definitions.
Then the maximal 
 $f$-invariant set  in $U$ has a 
partially hyperbolic  splitting 
$E^{\mruu}\oplus E^{\mrc} \oplus E^{\mrss}$, with $\dim (E^\mrc)=1$. 
As above this  allows us 
 to define the logarithm of the
center derivative $ \mathrm{J}_f^{\mrc}$ of $f$. We have the following  ``variation" of Theorem~\ref{t.cycle}.

\begin{theo}\label{t.ctail} 
Let $M$ be a  boundaryless compact manifold.
Consider $f\in \diff^1(M)$  with a
a filtrating region $U$ endowed 
with a strictly $Df$-invariant unstable cone field of index $i$ and a 
strictly $Df^{-1}$-invariant cone field of index $\dim(M)-i-1$.  

Assume that $f$ has a flip-flop configuration associated to
a dynamical blender 
and  a
hyperbolic periodic point $q$ both contained in $U$.

Then there is a $C^1$-neighbourhood $\cV_f$ of $f$ such that
every $g\in \cV_f$ has a nonhyperbolic ergodic measure whose support is the whole homoclinic class
$H(q_g,g)$ of the continuation $q_g$ of $q$.
\end{theo}

The hypothesis in this theorem imply that the blender and the saddle in the flip-flop configuration are in the
same chain recurrence class. With  the terminology of robust cycles,  they have a $C^1$-robust cycle. 

Note that
Theorem~\ref{t.ctail} is not a perturbation result: it
holds for every  diffeomorphism with such a  flip-flop configuration. Moreover, and more important,
the hypotheses in Theorem~\ref{t.ctail} are open (the set $U$ is also a filtrating set 
for every $g$ sufficiently close to $f$, hence 
the homoclinic class $ H(q_g,g)$ is contained in $U$ and 
partially hyperbolic, and   flip-flop configurations are robust). 
Thus
Theorem~\ref{t.ctail} holds for the homoclinic class of 
the continuation of $q$ for diffeomorphisms $g$ close to $f$.
%

\begin{rema}\label{r.newbdg}
Theorem~\ref{t.ctail} does not require  the
continuous variation of the homoclinic class $H(q_g,g)$ with respect to $g$. Note also that, in general, 
homoclinic classes  only depend lower semi-continuously on the diffeomorphism.
As a consequence, the partial hyperbolicity of a homoclinic class is not (in general) a robust property.
The relevant assumption is that the homoclinic  classes are contained in a partially hyperbolic filtrating neighbourhood which guaranteed the robust partial hyperbolicity of the homoclinic class.

We can change the hypotheses in the theorem, omitting that $U$ is a filtrating neighbourhood
and considering homoclinic classes depending continuously on the diffeomorphism 
(this occurs in a  residual subset of diffeomorphisms). Then, by continuity, the class is robustly contained in the partially hyperbolic region
and we can apply the previous arguments.
\end{rema}

\subsection{Applications to robustly nonhyperbolic transitive diffeomorphisms}
\label{ss.aplication}
A diffeomorphism $f\in \diff^1(M)$ is 
\emph{transitive} if it has a dense orbit. The diffeomorphism $f$ is 
\emph{$C^1$-robustly transitive} if any diffeomorphism $g$ that is $C^1$-close to $f$ is also transitive.  
In other words, a diffeomorphism is $C^1$-robustly transitive 
if it belongs to the $C^1$-interior of the set of transitive diffeomorphisms.

We denote by $\cR\cT(M)$ the ($C^1$-open) subset of $\diff^1(M)$ 
consisting of diffeomorphisms $f$ such that:
\begin{itemize}
\item
$f$ is robustly transitive, 
\item
$f$ has a pair  of hyperbolic periodic points of different indices,
\item 
$f$ has a partially hyperbolic splitting $TM=E^\mruu \oplus E^\mrc\oplus E^\mrss$, where $E^\mruu$ is uniformly expanding, 
$E^\mrss$ is uniformly contracting, and $E^\mrc$ is one-dimensional.
\end{itemize}
Note that the last condition implies that the hyperbolic periodic points of $f$
have either $\mru$-index $\dim (E^\mruu)$ or $\dim (E^\mruu)+1$. Note also that our assumptions imply that
$\dim(M)\ge 3$ (in lower dimensions $\cR\cT(M)=\emptyset$ , see \cite{PuSa}).

In dimension $\ge 3$ and depending on the type of manifold $M$, the set $\cR\cT(M)$ contains 
interesting examples. 
Chronologically, the first examples of such partially hyperbolic robustly transitive diffeomorphisms
were obtained in \cite{Sh} considering diffeomorphisms in $\TT^4$ obtained as  skew products of Anosov diffeomorphisms
on $\TT^2$ and derived from Anosov  on $\TT^2$
($\TT^i$ stands for the $i$-dimensional torus). Later, \cite{Mda} provides examples in $\TT^3$
considering  derived from Anosov diffeomorphisms. Finally, \cite{BD-robtran} gives examples that 
include  perturbations of  time-one maps of transitive Anosov flows
and perturbations of skew products of Anosov diffeomorphisms 
and isometries.

\begin{theo}
\label{t.c.openanddense}
There is a $C^1$-open and dense subset $\cZ(M)$ of $\cR\cT(M)$ such that
every $f\in \cZ(M)$ has an ergodic nonhyperbolic measure
whose support is the whole manifold $M$.
\end{theo}

Let us mention some related results. First, by
\cite{BoBoDi2}, there is  a $C^1$-open and dense subset
 of $\cR\cT(M)$ formed by diffeomorphisms
with  an ergodic nonhyperbolic measure
with positive entropy, but the support of these measures is not the whole ambient.
By \cite{BDG}, there is a residual subset of $\cR\cT(M)$  of diffeomorphism with an
ergodic nonhyperbolic measure with full support. 
Finally, a statement similar to our theorem
is stated in \cite{BJ}, see Footnote \ref{fn}.

Recall that given a periodic point $p$  of $f$ the measure   $\mu_{\cO(p)}$ is the 
unique $f$-invariant measure supported on the orbit of $p$. 

\begin{mcor}
\label{c.averages}
Consider a continuous map $\varphi \colon M \to \mathbb{R}$.
Suppose that $f\in \cR\cT(M)$ 
has two hyperbolic periodic orbits $p$ and 
$q$ such that
$$
\mu_{\cO(p)}  (\varphi)> 0> \mu_{\cO(q)}  (\varphi).
$$
Then there are a $C^1$-neighbourhood $\cV_f$ of $f$ and a
$C^1$-open and dense subset $\cO_f$ of $\cV_f$ such that every $g\in \cO_f$
has an ergodic measure $\mu_g$ with full support  on $M$ such that
$$
\int \varphi \, d\mu_g=0.
$$
\end{mcor}

\begin{mrem}
By \cite[Proposition 1.4]{C}, for diffeomorphisms in  $\cZ(M)$ 
every 
hyperbolic ergodic measure $\mu$ is the weak$\ast$ limit of 
periodic measures supported on points whose orbits 
tend (in the Hausdorff topology) to the support of the measure $\mu$. 
Thus, Corollary~\ref{c.averages} holds after replacing the hypothesis 
$\mu_{\cO(p)}  (\varphi)>0> 
\mu_{\cO(q)}  (\varphi)$
by the existence of
two hyperbolic ergodic measures $\nu^+$ and $\nu^-$ such that
$\int \varphi \, d\nu^+>0>\int \varphi\, d\nu^-$. 
\end{mrem}

\subsection{Organization of the paper}\label{ss.organization}
In Section~\ref{s.criterion} we introduce the concepts involved in the criterion of
control at any scale with a long sparse tail and prove
Theorem~\ref{t.accumulation}.
In Section~\ref{s.patterns}
we introduce the notion of a pattern and 
 see how they are induced by long tails  of  scales.
We study the {\emph{concatenations of plaques}} of  flip-flop families (associated to a map $\varphi$) and the
  control of the averages of $\varphi$ corresponding to these concatenations, see  Theorem~\ref{t.p.flipfloppattern}.
In Section~\ref{s.flipfloptail} we prove Theorem~\ref{t.flipfloptail}, 
Corollary~\ref{c.flipfloptail}, and Proposition~\ref{p.r.hyperboliclike}.
In Section~\ref{s.flip-flophomoclinic} we prove
Theorem~\ref{t.homoclinic} involving 
flip-flop families and homoclinic relations.
In Section~\ref{s.flipflopph} we review some key ingredients as 
dynamical blenders and
flip-flop configurations and
prove Theorems~\ref{t.cycle} and \ref{t.ctail}. Finally, in Section~\ref{s.applications} we apply our methods to construct nonhyperbolic ergodic measures with full support for some robustly transitive diffeomorphisms, proving Theorem~\ref{t.c.openanddense}
and  Corollary~\ref{c.averages}.

\section{A criterion for zero averages: control at any scale up 
to a long  sparse tail}
\label{s.criterion}

The construction that we present  for controlling averages is probably too rigid but it is enough to achieve our goals and certain constraints perhaps could be relaxed. However, at this state of the art, we do not aim for full generality but prefer to present the ingredients of the construction in a  simple as possible way. One may aim to extract a general conceptual principle behind the construction, but this is beyond the focus of this paper.
In Sections~\ref{ss.scalesandtails} and \ref{ss.controlatanyscale} we introduce the concepts involved in  the criterion for controlling averages and in Section~\ref{ss.proofoftaccumulation} we prove Theorem~\ref{t.accumulation}.


\subsection{Scales and long  sparse tails}
\label{ss.scalesandtails}
In what follows we introduce the definitions 
of scales and long sparse tails.

\begin{defi}[Scale]\label{d.scale}
A sequence $\cT=(T_n)_{n\in \NN}$ of strictly positive
natural numbers is called \emph{a scale} if
  there is a sequence $\bar \kappa=(\kappa_n)_{n\ge 1}$  
  (the \emph{sequence of factors} of the scale)
  of natural numbers with
  $\kappa_{n}\ge 3$ for every $n$ such  that
 \begin{itemize} 
  \item
   $T_{n}=\kappa_{n} \, T_{n-1}$ for every $n\ge 1$;
  \item
  $\kappa_{n+1}/\kappa_n \to \infty$.
\end{itemize}
We assume that
the number $T_0$, and hence every $T_n$, is a multiple of $3$.

\end{defi}

We now introduce some notation.
In what follows, given $a,b\in \RR$ we let
$$
[a,b]_{\NN}\eqdef[a,b]\cap \NN.
$$
Given a subset $\MM$ of $\NN$  
a \emph{component} of $\MM$ is an interval of integers 
$[a,b]_{\NN}\subset \MM$ such that $a,b\in \MM$ and $a-1, b+1\not\in \MM$.

\begin{defi}[Controling sequence]\label{d.controllingsequence}
Let $\bar\varepsilon=(\varepsilon_n)_{n\in\NN}$ be a sequence of positive numbers converging to $0$.  
We say that $\bar\varepsilon$ is a \emph{controlling sequence} if 
$$
\sum_n \varepsilon_n<+\infty
\quad 
\mbox{and}
\quad
\prod_n (1-\varepsilon_n) >0.$$
 \end{defi}
 
\begin{rema}\label{r.basicfact} 
For a sequence $(\varepsilon_n)_{n\in\NN}$ of numbers with $\epsilon_n\in(0,1)$  one has 
$$
\sum_n \varepsilon_n<+\infty\quad   \Longleftrightarrow \quad \prod_n (1-\varepsilon_n) >0.
$$
\end{rema}

\begin{rema}\label{r.controling}
Let
 $\cT=(T_n)_{n\in \NN}$  be a scale, $\kappa_{n+1}=\frac{T_{n+1}}{T_n}$, and 
 $\varepsilon_n =\frac{2}{\kappa_n}$.
 Then the sequence
 $(\varepsilon_n)_{n\in \NN}$ is a controlling one.
\end{rema}

\begin{defi}[Long sparse tail]
\label{d.tail} 
Consider a scale
$\cT=(T_n)_{n\in \NN}$ and a controlling sequence 
$\bar \varepsilon=(\varepsilon_n)_{n\in \NN}$. 
A  set $R_\infty\subset \NN$ is a \emph{$\cT$-long $\bar\varepsilon$-sparse tail} if
the following properties hold:

\begin{enumerate}
\item[a)]\label{i.component} Every component of $R_\infty$ is of the form $[k\,T_n, (k+1)\, T_n-1]_{\NN}$, for some $k$ and $n$ (we say that such a component has size $T_n$).
\end{enumerate}
Let $R_n$ be the union of the components of $R_\infty$ of size $T_n$ and let
$$
R_{n,\infty}\eqdef \bigcup_{i\ge n} R_i,
$$
the union of the components of $R_\infty$ of size larger than or equal to $T_n$. 
\begin{enumerate}
\item[b)]\label{i.0} 
$0\notin R_\infty$, in particular 
$[0,T_n-1]_{\NN} \not\subset R_n.$
\item[c)]\label{i.center}  
Consider  an  interval $I$ of natural numbers of the form
$$
I=[kT_n,(k+1) T_n -1]_{\NN},
\quad \mbox{for some $n\ge 1$ and $k\geq 0$,}
$$  
that  is not contained in any component of  $R_\infty$ then the following properties hold:
\begin{itemize}
\item{\emph{center position:}}
$$
\left[ kT_n, kT_n +\frac{T_n}{3}\right]\cap R_{n-1}=\emptyset=
\left[\big((k+1)T_n-1\big)-\frac{T_n}{3}, \big( (k+1)T_n-1\big)\right]\cap R_{n-1}.
$$
\item {\emph{$\bar\varepsilon$-sparseness:}} 
$$
0< \frac{ \#( R_{n-1}\cap I)}{T_n}<\varepsilon_n.
$$
\end{itemize}
\end{enumerate}
\end{defi}

The conditions in Definition~\ref{d.tail} are depicted in Figure~\ref{f.dlongsparsetail}.

\begin{figure}
\begin{minipage}[h]{\linewidth}
\centering
 \begin{overpic}[scale=.75, 
 ]{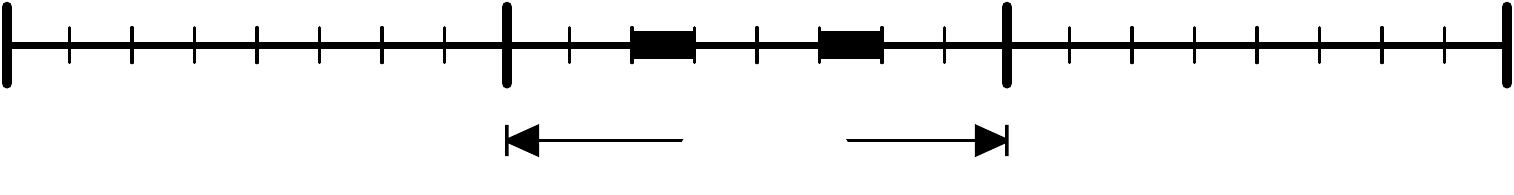}
 			\put(-2,4){\small$\kappa\, T_n$}
 			\put(92,4){\small$(\kappa+1)\, T_n$-1}
  			\put(41,6){\small$R_{n-1}$}
			\put(53,6){\small$R_{n-1}$}
 			\put(48.5,2){\small$\frac{T_n}3$}
  \end{overpic}
\caption{A long  sparse tail}
\label{f.dlongsparsetail}
\end{minipage}
\end{figure}
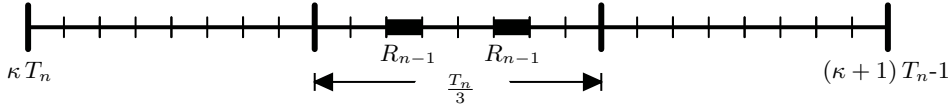

\begin{defi}[Good and bad intervals]\label{d.goodandbad}
With the notation of Definition~\ref{d.tail},
 an interval $I$ of the form $I=[kT_n,(k+1)T_n-1]_{\NN}$ 
 is called  \emph{$n$-bad} if $I\subset R_{n,\infty}\subset R_\infty$. 
 The interval is called \emph{$n$-good} if 
 $I\cap R_{n,\infty}=\emptyset$.
\end{defi}

\begin{rema}[On the definition of a $\cT$-long $\bar\epsilon$-sparse tail]
\label{r.commentson}
$\,$
\begin{enumerate}
\item It is assumed that $0\notin R_\infty$.  
This implies that for every $n\geq 0$ the initial interval $[0,T_n-1]_{\NN}$ is not a component of $R_\infty$ of size $T_n$. 
Therefore $[0,T_n-1]$ is disjoint from $R_n$ and thus from $R_m$ for every $m\geq n$.  
In other words, the interval  $[0,T_n-1]$ is $n$-good, that is,
$$
[0,T_n-1]\cap R_{n,\infty}=\emptyset.
$$
\item
Let  $I=[kT_n,(k+1) T_n -1]$ be an interval as in Item~(c) of Definition~\ref{d.tail}.  By  Item~(a) the interval
$I$  is either contained in a component of $R_\infty$
whose size is larger than or equal to  $T_n$ or is disjoint from $R_{n,\infty}$.  Thus, Item~(c)
considers the case where the interval is disjoint from $R_{n,\infty}$.

Now  any component of $R_\infty$ of size less than $T_n$ is either disjoint from $I$ or contained in $I$: 
just note that such a component has  length $T_m$, $m<n$,  and starts 
at a multiple of $T_m$ and   $T_n$ is a multiple of $T_m$. 

Item~(c)     describes the position and quantity  of the components of size $T_{n-1}$ in 
the interval $I$.  For that, one splits the interval 
$I$ into tree parts of equal length $\frac{T_n}{3}$. The following properties are required:
\begin{itemize}
\item 
Every component of size $T_{n-1}$ contained in $I$ is contained in the middle third interval;
\item 
The middle third interval contains at least one component of size $T_{n-1}$. Thus the intersection 
$I\cap R_{n-1}$ is not empty, but the the set $R_{n-1}$ has  a
 small density in the interval $I$  that is upper bounded by $\epsilon_n$.
 \end{itemize}
\item  
 Item~(c)
 does not consider  the case  $n=0$.  
  For $n=0$  and an interval $I$ of the form $[kT_0,(k+1)T_0-1]$ there are two possibilities: 
 either $I$ is a component of $R_\infty$ 
 (i.e., contained in $R_0$) or $I$ is disjoint from $R_\infty$. 
 \item  
 Given  any interval $I$ of the form $I=[kT_n,(k+1)T_n-1]_{\NN}$ there are two possibilities:
 \begin{itemize}
 \item either $I\subset R_\infty$ and then $I \subset R_{n,\infty}$ and $I$ is $n$-bad;
 \item or $I\cap R_{n,\infty}=\emptyset$ and then $I$ is  $n$-good.
 \end{itemize}
 \end{enumerate}
\end{rema}

The definition of a long  sparse tail 
involves many properties and conditions, thus its existence it is not obvious.
We solve this difficulty in the next lemma.

\begin{lemm}[Existence of long sparse tails]\label{l.tailexistence}
Consider a scale
 $\cT=(T_n)_{n\in \NN}$
 and its sequence of factors  $\bar \kappa=(k_n)_{n\ge 1}$.
Write
   $\varepsilon_n =\frac{2}{\kappa_n}$
   and let $\bar\epsilon=(\epsilon_n)_{n\ge 1}$.
 Then there is a $\cT$-long 
$\bar\varepsilon$-sparse tail $R_\infty$.
\end{lemm}

\begin{proof} 
First  note that, by  
Remark~\ref{r.controling}, the sequence
$\bar \varepsilon$ is a controlling one. 

The construction of the set  $R_\infty$ is done inductively.
For each  $n\in\NN$ we define 
the intersection of the set  $R_{\infty}$ with the intervals $[0,T_n-1]$. 
We denote such an intersection by $R_{\infty}(T_n)$.

For $n=0$,  we let $R_{\infty}(T_0)=\emptyset$. 
Fix now $n>0$ and suppose that the sets $R_{\infty}(T_{n-1})$ has been constructed satisfying (in 
restriction to the interval
$[0,T_{n-1}-1]$) the properties in Definition~\ref{d.tail}. We now proceed to define the set $R_{\infty}(T_n)$.

For any $i\le n-1$ we denote by $R_{i,n-1}$ the union of the components of 
$R_{\infty} (T_{n-1})$ of 
length $T_i$. 
We next define the family of subsets  $\{R_{j,n}, \, j=0,\dots, n\}$ of $[0,T_n-1]$  by 
 decreasing induction on  $j$ as follows. We let
$$
R_{n,n}=\emptyset \qquad \mbox{and}  \qquad
R_{n-1,n}= \left[ \frac{T_n}{3}, \frac{T_n}{3}+T_{n-1}-1 \right]_{\NN}.
$$
Let $j<n$ and assume that the sets $R_{i,n}$ are defined for every $n\ge i>j$. The set
$R_{j,n}$ is defined as follows:
\begin{itemize}
\item if
 $[kT_{j+1}, (k+1)T_{j+1}-1]_{\NN} \subset \bigcup_{i>j} R_{i,n}$ then 
$$
R_{j,n}\cap [kT_{j+1}, (k+1)T_{j+1}-1]=\emptyset,
$$ 
 \item Otherwise we let
\begin{equation}\label{e.Rjn}
R_{j,n}\cap [kT_{j+1}, (k+1)T_{j+1}-1]_{\NN}  =\left[   \Big( k+\frac{1}{3} \Big) T_{j+1}, 
\Big(k+\frac{1}{3} \Big) T_{j+1} +T_j-1 
\right]_{\NN}.
\end{equation}
\end{itemize}
Note that by construction,
$$
R_{\infty}(T_n)=\bigcup_{i=0}^n R_{i,n}.
$$

\begin{clai} The set $R_{\infty}(T_n)$ satisfies (in restriction to 
the interval $[0,T_n-1]$) the conditions of Definition~\ref{d.tail}. 
\end{clai}
\begin{proof} 
Property (a) in the definition follows from the construction: 
the components of $R_{i,n}$ have size $T_i$ and 
have no adjacent points with the components of  $\bigcup_{j>i}^n R_{j,n}$. 

For Property (b) one checks inductively that $O\notin R_{i,n}$ for every $i$ and $n$. 

Property (c)  is a consequence of \eqref{e.Rjn}. If  the set
$R_{i,n}$  intersects  a segment $[k T_{i+1},(k+1)T_{i+1}-1]$
then it is  contained in its middle third interval, implying the center position condition.
For the sparseness note that by construction and the definition of $\varepsilon_i$, 
for each $i$ it holds 
$$
0< \frac{ \#( R_{i-1}\cap I)}{T_i}= \frac{T_{i-1}}{T_{i}}=\frac{1}{\kappa_i}
<\varepsilon_j.
$$
This completes the proof of the claim. 
%
\end{proof}

Our construction also provides immediately the following properties:
For every $i<n$ it holds:
\begin{itemize}
\item
if $m\ge n$ then
$R_{i,m}\cap [0,T_n-1]= R_{i,n}$,
\item if $m\geq n$ the $R_{\infty}(T_m)\cap [0,T_n-1]=R_{\infty}(T_n)$, and
\item
$R_{i,n}\subset R_{i,n+1}$.
\end{itemize}
The tail is now defined by
\begin{equation}
\label{e.RinftyRi}
R_\infty\eqdef \bigcup_{i=0}^\infty R_i, \quad \mbox{where} \quad
R_i= \bigcup _{n>i} R_{i,n}.
\end{equation}
By construction, the set $R_{\infty}$ is an $\varepsilon$-sparse tail of $\cT$.
\end{proof}

\subsection{Control at any scale up a long sparse tail}
\label{ss.controlatanyscale}
In this section we give the definition of controlled points.

\begin{defi}
\label{d.alphacontrol}
Let $X$ be a compact set, $f\colon X\to X$ a homeomorphism,
 and $\varphi \colon X\to \RR$
a continuous map. Consider 
\begin{itemize}
\item
a scale $\cT$, a controlling sequence $\bar\varepsilon$, and a 
$\cT$-long  $\bar\varepsilon$-sparse tail  $R_\infty$;
\item 
decreasing sequences of positive numbers $\bar{\delta}=(\delta_n)_{n\in\NN} $ and $\bar{\alpha}=(\alpha_n)_{n\in\NN}$, 
converging to $0$. 
\end{itemize} 

The $f$-orbit of a  point $x\in X$ is \emph{$\bar\delta$-dense along the tail $R_\infty$} if 
for every component $I$ of $R_\infty$ of length $T_n$ 
the segment of orbit $\{f^i(x), \, i\in I\}$ is $\delta_n$-dense in $X$.

The \emph{Birkhoff averages of $\varphi$ along the orbit of $x$ are 
$\bar\alpha$-controlled
for the scale $\cT$ 
with the 
tail $R_\infty$} if
for every interval 
$$
I=[k \, T_n, (k+1) \, T_n-1]_\NN
$$ 
such that
$I\not \subset R_{n+1,\infty}$  
(i.e., $I$ is either  $n$-good or is a component of $R_n$) it holds
$$
\varphi_I(x)\eqdef
\frac{1}{T_n} \sum_{i\in I} \varphi (f^i(x)) \in [-\alpha_n,\alpha_n].
$$
\end{defi}

\begin{defi}\label{d.controledtail}
Let $X$ be a compact set, $f\colon X\to X$ a homeomorphism,
 and $\varphi \colon X\to \RR$
a continuous map.
 

A point $x\in X$ is \emph{controlled  at any scale with a long sparse tail
with respect to $X$ and $\varphi$} if there are a scale  $\cT$, a controlling sequence $\bar\varepsilon$, a $\cT$-long
$\bar\varepsilon$-sparse tail $R_\infty$, and 
sequences of positive numbers   $\bar\delta$ and $\bar\alpha$ converging to $0$,  such  that
\begin{itemize}
\item
the $f$-orbit of  
$x$ is $\bar\delta$-dense along the tail $R_\infty$ and 
\item
the Birkhoff averages of $\varphi$ along the orbit of $x$ are 
$\bar\alpha$-controlled for the scale $\cT$ 
with the 
tail $R_\infty$.
\end{itemize}
In this definition we say that $\bar \delta$ is the \emph{density forcing sequence},
$\bar \alpha$ is the \emph{average  forcing sequence,} and the
point $x$ is \emph{$(\bar\delta,\bar \alpha, \bar\epsilon, \mathcal{T}, R_\infty)$-controlled.}
\end{defi}

%
%

\subsection{Proof of Theorem \ref{t.accumulation}}
\label{ss.proofoftaccumulation}
In this section we prove 
Theorem~\ref{t.accumulation}, thus we use the assumptions and the notations in its statement.
Consider a point  $x_0\in X$  that is controlled  at any scale with a long sparse tail
for  $X$ and $\varphi$. Let 
\begin{itemize}
\item
$\cT=(T_n)_{n\in \NN}$ be the scale;
\item
$R_\infty$ the $\cT$-long  $\bar \varepsilon$-sparse tail; and
\item 
$\bar \delta=(\delta_n)_{n\ge 1}$ the density forcing sequence and
 $\bar \alpha=(\alpha_n)_{n\ge 1}$ the average forcing sequence. 
 \end{itemize}

Let $\mu$ be a measure that  
is a weak$\ast$ limit of the empirical measures $(\mu_n(x_0))_{n\in \NN}$.
As $x_0$ remains fixed let us write $\mu_n\eqdef \mu_n(x_0)$.
We need to prove that for $\mu$-almost every point $x$ it holds:
 \begin{enumerate}
\item \label{i.ass1}
the forward orbit of $x$ is dense in $X$
 \item \label{i.ass2}
 the Birkhoff averages of $x$ satisfy
 $\lim_{n\to\infty}  \frac{1}{n} \sum_{i=0}^{n-1} \varphi (f^i(x))=0$.
 \end{enumerate}
Proposition~\ref{p.l.deltadensity}  below
immediately implies   item~\eqref{i.ass1} 
(item \eqref{i.ass2} follows from Proposition~\ref{p.l.alpha}).

\begin{prop}\label{p.l.deltadensity}
Under the assumptions above, for every $k$ 
the (forward) orbit of
 $\mu$-almost every point 
 is $2\,\delta_k$-dense in $X$.
\end{prop}

\begin{proof}
Fix $k$.
For any given $t>0$ and  $\delta>0$ consider the set 
$$
X(t,\delta)\eqdef 
\Big\{x\in X \colon \{x, \dots ,f^t(x)\} \, \mbox{is $\delta$-dense in $X$} \Big\}
$$
and let
 $$
 P_{\infty,t} \eqdef \liminf_{n\to\infty} P_{n,t}, \quad
 \mbox{where} \quad
P_{n,t}\eqdef \mu_n(X(t,\delta_k)).
$$

\begin{lemm}\label{l.c.pinfty}
$\lim_{t\to \infty} P_{\infty,t}=1.$
\end{lemm}

We postpone the proof of this lemma
and deduce the proposition from it. Just
note  that the interior of $X(t,2\delta_k)$ contains
the closure of $X(t,\delta_k)$ for every $t$. Thus
$\mu (X(t,2\delta_k))\ge
P_{\infty,t}$. Taking the limit when $t\to \infty$ we prove the proposition.
\end{proof}

\begin{proof}[Proof of Lemma~\ref{l.c.pinfty}]
Fixed $k$ take $t>T_{k+1}$.

\begin{clai} \label{cl.biblioteca}
 The set of times $i\in\NN$ such that  $f^i(x_0)\not\in X(t,\delta_k)$ is contained in the set 
$$
\bigcup_{j=0}^\infty \Big( R_{m_t+j}\cup (R_{m_t+j}-T_{k+1})\Big),
$$
where 
\begin{itemize}
\item
 $m_t\eqdef \inf \{ m \ge k+1 \colon T_m+2\,T_{k+1}>t\}$  
 and
\item 
 $R_{k+j}-T_{k+1}\eqdef \{ \ell =i-T_{k+1}, \,\, \mbox{where $i\in R_{k+j}$}\}$.
 \end{itemize}
\end{clai} 
\begin{proof}
Take $i$ such that $f^i(x_0)\notin X(t,\delta_k)$. Then the set   $\{f^i(x_0),\dots, f^{i+t}(x_0)\}$ is not $\delta_k$-dense. 
Let $I=[i,i+t]_{\NN}$.
Recalling Definition~\ref{d.alphacontrol}, 
we have that $I$ can not contain any component of $R_\infty$ of size 
$T_k$ or greater than $T_k$.
This implies that 
\begin{itemize}
\item
the interval $I$ does not contain any $\ell$-bad interval for $\ell\ge k$,
\item
as a consequence of the sparseness property in item (c) of Definition~\ref{d.tail}, 
the interval $I$ does not contain any $(\ell+1)$-good interval for $\ell\ge k$
(i.e., disjoint from  or $R_{\ell+1,\infty}$).
\end{itemize}
Thus necessarily the interval $I$ intersects some bad interval
 $J=[r_m^-, r_m^+]_{\NN}\subset R_m$, $m>k$,
such that
$$
I\subset [r_m^--T_{k+1}, r_m^++T_{k+1}].
$$ 
Otherwise $I$ must contain a $(k+1)$-good interval.
Observe  that this implies that 
$$
T_m+2\,T_{k+1}>t,
$$
otherwise the segment of orbit $\{f^{i+j}(x_0)\}_{j=0}^{t}$ would be $\delta_k$-dense, a contradiction.
Hence $m\geq m_t$. 

Recall that $T_{k+1}<t$,  hence $i\in  [r_m^--T_{k+1}, r_m^+]$.
Thus 
$$
i\in J\cup (J-T_{k+1})\subset R_m\cup (R_m-T_{k+1})
$$
 for some $m\geq m_t$. This ends  the proof of the claim. 
\end{proof}

In view of Claim~\ref{cl.biblioteca}, to prove the  lemma 
it is enough to  see the following:

\begin{clai}\label{c.semnome}
 $$
 \lim_{t\to +\infty} \,\lim_{n\to+\infty} \frac1n \,
 \#
 \left(
 [0,n]\cap \bigcup_{j=0}^\infty \Big(  
 R_{m_t+j}\cup (R_{m_t+j}-T_{k+1})  \Big)
 \right)=0.
 $$
\end{clai}

\begin{proof}
Note that the components of the set $R_{m_t+j}\cup (R_{m_t+j}-T_{k+1})$ are intervals of 
length $T_{m_t+j}+T_{k+1} < 2\, T_{m_t+j}$.  Thus  the claim is a direct consequence of 
next fact (recall the definition of $R_{m_t,\infty}$ in Definition~\ref{d.tail}).

\begin{fact}\label{f.c.proportion}
$$
\lim_{t\to \infty} \limsup_{n\to \infty} \dfrac{1}{n} \#\,  \big(R_{m_t,\infty} \cap [0,n] \big)\to 0.
$$
\end{fact}
\begin{proof}
We need to estimate the proportion
$$
\varrho(m,n)\eqdef \frac{\# \, \big( R_m\cap [0,n] \big) }{n}
$$
of the set $R_m$ in  $[0,n]_{\NN}$. We claim that 
$\varrho(m,n)< 3\,\varepsilon_{m+1}$. 
There are three
cases:
\begin{itemize}
\item 
$T_{m+1} \le n$:  Let $k\,T_{m+1}\le n< (k+1)\, T_{m+1}$, where $k\in \NN$ and $k\ge 1$.
By the sparseness condition we have
$$
\# \big( R_m \cap [0, (k+1)\, T_{m+1}] \big) < (k+1)\, \varepsilon_{m+1}.
$$
Therefore
$$
\frac{\#\,\big( R_m \cap [0, n] \big)}n \le 
\frac{\#\,\big( R_m \cap [0, n] \big) }{k\, T_{m+1}} \le 
\frac{(k+1) \varepsilon_{m+1}}{k}< 2\, 
\varepsilon_{m+1}.
$$
\item 
$T_m\leq n< T_{m+1}$: 
Since $[0,T_{m+1}-1]$ is an $(m+1)$-good interval  we have that  
$[0,T_{m+1}/3]$ and $R_m$ are disjoint. 
If the proportion is $0$ we are done.
Otherwise, by the center position condition,
 $n>\frac{T_{m+1}}3$. 
 Therefore
  $$
  \frac{\# \, \big( R_m\cap [0,n]\big) }{n}< 3\, \frac{\# \, \big( R_m\cap [0,T_{m+1}]\big)}{T_{m+1}}<3\, \varepsilon_{m+1},
  $$
  where the last inequality follows from the sparseness condition. 
   \item 
$n<T_m$:  In this case, by condition (b) in Definition~\ref{d.tail},
$R_m\cap[0,n]=\emptyset$. 
\end{itemize}

Since $\varrho (m,n)<3\,\varepsilon_{m+1}$ for every $n$ we get
$$ 
\frac{1}{n} \# \, \big( R_{m_t,\infty} \cap [0,n] \big)< 3\, \sum_{m=m_t}^\infty \varepsilon_m.
$$
Since, by definition,
$\sum_{m=0}^\infty \varepsilon_m<+\infty$ 
this implies 
$$
\lim_{t\to \infty}\, \limsup_{n\to \infty} \frac{1}{n}\, \# (R_{m_t,\infty} \cap [0,n])\leq  3\, \lim_{t\to\infty}\sum_{m=m_t}^\infty \varepsilon_m = 0,
$$
proving the fact.
\end{proof}
This ends the proof of Claim~\ref{c.semnome}
\end{proof}
The proof of Lemma~\ref{l.c.pinfty} is now complete.
\end{proof}

Proposition~\ref{p.l.deltadensity} gives the density of orbits in Theorem~\ref{t.accumulation}.  To  end the proof of
the theorem it remains to 
prove the part relative to the averages. 
This is an immediate  %
%
consequence of next proposition. 
Recall the notation of finite Birkhoff averages $\varphi_n(x)$ and of limit averages $\varphi_\infty(x)$
of a function $\varphi$
in  \eqref{e.birkhoffaverages} and \eqref{e.notationBirkhoff}. 
Recall also that $\mu$ is a weak$\ast$ limit of the 
empirical measures  $\mu_n=\mu_n(x_0)$.

\begin{prop}\label{p.l.alpha}
Fix $k\in \NN$. For $\mu$-almost every point $x$  
the limite average $\varphi_\infty(x)$ is well defined and belongs
to $[-3\alpha_k,3\alpha_k]$.
\end{prop}

\begin{proof}
Let $B$ be the set of points such that the limit average $\varphi_\infty (x)$ is well defined.
By  Birkhoff theorem  it  holds $\mu(B)=1$.
Therefore it is enough to prove that for every $x\in B$ there is a sequence $n_j=n_j(x) \to \infty$
such that
$\varphi_{n_j}(x)\in [-3\,\alpha_k,3\,\alpha_k]$ for every $j$.
For $t\in \NN$ define the number
$$
q_t\eqdef 
\liminf_{n\to+\infty}\, q_{t,n},
\quad
\mbox{where}
\quad
q_{t,n}\eqdef
\mu_n
\big(\left\{x \colon \varphi_t(x) \in [-2\,\alpha_k,2\,\alpha_k]
\right\}\big).
$$
\begin{lemm}\label{l.c.alpha}
 $\lim_{t\to\infty} q_t=1.$
\end{lemm}

Let us postpone the proof of this lemma and  conclude the proof of the proposition assuming it.
By definition of $\mu$
$$
\mu (\{x \colon \varphi_t(x) \in [-3\,\alpha_k,3\,\alpha_k]\})\ge q_t.
$$
By Lemma~\ref{l.c.alpha}, $q_t\to 1$, thus there is a subsequence $(q_{t_i})$ such that 
$$
\sum_{0}^\infty(1- q_{t_i}) <+\infty.
$$
Fix the sequence $(q_{t_i})_i$ and define the sets
$$
Y_N\eqdef  \bigcap_{j>N}\{x \colon \varphi_{t_j}(x) \in [-2\, \alpha_k,2\, \alpha_k]\}.
$$
By definition
$$
\mu(Y_N) \ge 1-\left(\sum_{j=N}^\infty (1-q_{t_j})\right)
$$
and hence
$$
\mu(Y)=1, \quad \mbox{where} \quad Y\eqdef \bigcup_N Y_N.
$$
We have that $\mu(B\cap Y)=1$ and that
 every point
$x\in Y\cap B$ has Birkhoff average $\varphi_\infty (x)\in [-3\, \alpha_k,3\, \alpha_k]$,
this ends the proof of the proposition (assuming Lemma~\ref{l.c.alpha}).

\begin{proof}[Proof of Lemma~\ref{l.c.alpha}]
Recall that $x_0\in X$  is controlled  at any scale with a long sparse tail
for  $X$ and $\varphi$. Recall also the choices of $\cT$, $R_\infty$, $\bar \varepsilon$, and
$\bar \alpha$. 

Fix $k$ and take a number $t>2\,T_k$. 
Pick 
 a time interval $I=[i,i+t-1]_{\NN}$.
 For each $j\ge k$ consider the intervals $H$ of the form $[rT_j,(r+1)T_j-1]_{\NN}$ contained in $I$ which are
  either $j$-good or components of the set $R_j$. 
By definition of the average forcing sequence $\bar \alpha$
(Definition~\ref{d.alphacontrol})
in each of these intervals $H$ the average satisfies
$$
\varphi_H (x_0)\in 
[-\alpha_j,\alpha_j] \subset  [-\alpha_k,\alpha_k].
$$
We call these subintervals $H$ of $I$
\emph{$\alpha_k$-controlled}. 

\begin{clai}\label{c.JI}
 For every $i$ and $t$ the union of the $\alpha_k$-controlled intervals contained in 
 $I=[i,i+t-1]_{\NN}$ is a (possibly empty) interval $J=J_I$
 such that $\varphi_J(x_0)\in [-\alpha_k,\alpha_k]$. 
\end{clai}
\begin{proof} 
Let us first define auxiliary  intervals $A=A_I$ and $B=B_I$. The interval $A$ is
defined as follows:
\begin{itemize}
 \item 
 If $i\in R_m$ for some $m\geq k$ then $A$ is the intersection of the component of $R_m$ containing $i$
 and $I$;
 \item
 otherwise we let  $A=[i,jT_k-1]$, where $j$ is the infimum of the numbers $r$ with  $i\leq r \, T_k$  
 (note that $A$ is empty if $i=jT_k$).
\end{itemize}
The interval $B$ is symmetrically defined as follows:
\begin{itemize}
 \item 
 if $i+t\in R_m$ for some $m\geq k$ then $B$ is the intersection 
 of the component of $R_m$ containing $i+t$ and $I$;
 \item 
 otherwise  we let $B=[\ell T_k, i+t]$, where $\ell$ is the maximum of the numbers $r$
  $ r T_k\leq i+t+1$
  (note that $B$ is empty if $i+t=jT_k$). 
\end{itemize}

\begin{fact}
\label{f.setJ}
 $J=[i,i+t]\setminus(A\cup B)$.  
\end{fact} 

\begin{proof}
Just note that by construction every component of $R_m$ intersecting $J$ is
contained in $J$. A similar inclusion holds for  every $m$-good interval intersecting 
$J$.
These two inclusions imply the fact. 
\end{proof}
 
It remains to see that $J$ is the union of pairwise disjoint $\alpha_k$-controlled intervals.  
By Fact~\ref{f.setJ} and construction, the components of 
$\bigcup_{m\ge k} R_m$
intersecting $J$ are 
contained in $J$. 
These components are pairwise disjoint and 
their complement is a union of $T_k$-intervals which are good.
This implies that $J$ is a disjoint union of intervals  $H$ where  the average satisfies
$\varphi_H (x_0)\in [-\alpha_k,\alpha_k]$.
This implies that the
 average $\varphi_J (x_0)$ of $\varphi$ in  $J$ belongs to $[-\alpha_k,\alpha_k]$, ending the proof of the claim.
\end{proof}

\begin{clai} 
\label{c.filldensity}
Fix $k$.
Given any $m\in \NN$ there is $t_m$  such 
that for every  $t\geq t_m$ and for every $i\in \NN$ 
such that
 $$
 \frac 1t \, \sum_{j=0}^{t-1}
 \varphi(f^{i+j}(x_0))\notin [-2\,\alpha_k,2\,\alpha_k], 
 $$
 then 
either $i$ or $i+t$ belongs to
$\bigcup_{\ell \geq m} R_\ell$.
\end{clai}
\begin{proof} 
Pick and interval $I=[i,i+t]_{\NN}$ and associate to it the interval $J=J_I$ in Claim~\ref{c.JI}
and the intervals $A=A_I$ and $B=B_I$ in its proof.
As $\varphi_J(x_0)\in [-\alpha_k,\alpha_k]$,  in order to have
$|\varphi_I (x_0)|> 2\, \alpha_k$ the set
$I\setminus J=A\cup B$ 
must fill a relatively large proportion (depending on $\alpha_k$ and $\sup|\varphi|$ but independent of $t$)  of the interval
$I$. In other words, there is   a constant $C>0$ such that
\begin{equation}\label{e.contradiction}
\frac{\#(A\cup B)}{t}>C.
\end{equation}

Fixed $m$, let
$$
t_m\eqdef \frac{2\,T_m}{C}+1.
$$
Take any $t\ge t_m$. The proof is by contradiction, if 
$i,i+t\notin \bigcup_{\ell \geq m} R_\ell$
then $|A|, |B|\le T_m$ and therefore 
 $$
 \frac{\#(A\cup B)}{t}\le \frac{2\, T_m}{t}<C,
 $$
 a contradicting \eqref{e.contradiction}. The proof of the claim is complete.
%
%
%
%
\end{proof}

We are now ready to conclude the proof of the lemma.
Claim~\ref{c.filldensity} implies that
for $t>t_m$ the number $(1-q_t)$ is less than twice the density of the set 
$\bigcup_{\ell \geq m} R_\ell$ in 
$[0,t]$.
Fact~\ref{f.c.proportion} implies that this density goes to $0$ as $t\to \infty$,
proving the lemma.
\end{proof}
The proof of Proposition~\ref{p.l.alpha} is now complete.
\end{proof}

\section{Patterns, concatenations, flip-flop families, and control of averages}\label{s.patterns}
In this section we introduce the notion of a pattern (Section~\ref{ss.patterns}) and explain its relations with the scales and tails  in
the previous section. In Section~\ref{ss.tailsandpatterns}
we see that a $\cT$-long tail  of a scale $\cT$
 induces patterns in its good intervals. Patterns will be used to codify certain orbits
 (the orbit follows some distribution pattern). This naive idea is formalised in the  notion of a concatenation of sets following a pattern, see Section~\ref{ss.concatenations}.
 We are interested in concatenations of plaques of a flip-flop family (associated to a map $\varphi$) and in the
  control of the averages of $\varphi$ corresponding to these concatenations, see Section~\ref{ss.distortionfobirkhoffaverages}. The main result in this section is Theorem~\ref{t.p.flipfloppattern} 
  that gives the control of averages for concatenations, 
  see Section~\ref{ss.flip-flopfamiliesofplaques}.
 In the sequel we will make more precise these  vague notions.

\subsection{Patterns}\label{ss.patterns}
A scale $\mathcal{T}=(T_n)_{n\ge 0}$
induces, for each $n$, a partition of $\NN$ consisting of intervals of the form  $[\ell \, T_n,(\ell+1)\, T_n-1]_{\NN}$. 
A pattern is a partition of these intervals respecting some compatibility 
rules given by the scale.

\begin{defi}[Pattern]
\label{d.pattern}
Let $\mathcal{T}=(T_n)_{n\ge 0}$ be a scale
and $I\subset \NN$ an interval of the form $I=[\ell \, T_n,(\ell+1)\, T_n-1]_{\NN}$ for some $\ell\in \NN$.

A $T_n$-pattern  $\mathfrak{P}=\mathfrak{P}(I)$ of the interval $I$
consists of a 
partition $\cP=\{I_i\}_{i=1}^r$ of 
$I$  into intervals $I_i=[k\, T_{\ell(i)}, (k+1)\, T_{\ell(i)-1}]_{\NN}$, where $\ell(i)\in \{0,\dots, n\}$,
and a map $\iota\colon\cP \to \{r,w\}$
 such that
\begin{itemize}
\item
either $\ell(i)\ne 0$ and  then 
$\iota (I_i)=w$,
\item
or $\ell(i)=0$ and then $\iota (I_i)\in\{r,w\}$.
\end{itemize}
We write $\fP=(\cP,\iota)$.

A subinterval of $I=[kT_i, (k+1)T_i-1]$ that is not strictly contained in an interval of the partition $\cP$ is called
\emph{$\mathfrak{P}$-admissible (of length $T_i$)}.
\end{defi}

In this definition, the script $w$ refers to ``walk'' and $r$ to ``rest''. 

\begin{rema}\label{r.subpattern}$\,$
\begin{enumerate}
\item If $[kT_i, (k+1)T_i-1]$  is a  $\fP$-admissible subinterval of $I$,
then the restriction of the partition $\cP$ and the restriction of the map $\iota$ induces
a $T_i$-pattern in
$[kT_i, (k+1)T_i-1]$, which is called a \emph{subpattern} of $\fP$. 
\item A $T_n$-pattern  consists either of a unique interval of $w$-type or 
is a ``concatenation" of $T_{n-1}$-patterns.
\end{enumerate}
\end{rema}


Consider an interval $I$ and a pattern $\fP$ of it as in 
Definition~\ref{d.pattern}.
A point $j=k\,T_i\in \NN$ 
is \emph{$i$-initial for the pattern $\mathfrak{P}$}   if
the interval $[j, j+T_i-1]_{\NN}$ is admissible.
A point $j\in \NN$ is  \emph{$\mathfrak{P}$-initial} if it is $i$-initial for some $i$. 
We denote the set of initial points of $\mathfrak{P}$ by $I(\mathfrak{P})$.
The set of \emph{$\mathfrak{P}$-marked points} of the $T_n$-pattern $\mathfrak{P}$, denoted by $M(\mathfrak{P})$, is the union of 
the point $\{(\ell+1)T_n\}$ and the set of all initial points of $\mathfrak{P}$.

\subsection{Tails and patterns}
\label{ss.tailsandpatterns}
We now see that given a scale $\cT$ and a $\cT$-long sparse tail $R_\infty$, 
the tail induces patterns in its good intervals $I$ 
(i.e., $I\cap R_{n,\infty}=\emptyset$, recall Definition~\ref{d.goodandbad}).
In this subsection the sparseness of the tail is not relevant.

\begin{lemm}[Pattern induced by a tail]
\label{l.tailpattern} 
Let  
$\cT=(T_n)_{n\in \NN}$ 
 be a scale and 
$R_{\infty}$ a $\cT$-long  
 sparse tail.
Let 
$$
I=[\ell \, T_n,(\ell+1)\,T_n-1]_{\NN}
$$ 
be an $n$-good interval of $R_{\infty}$ and
consider the partition $\cP$ of  $I$ and  the map $\iota\colon \cP\to \{r,w\}$ defined as follows:
\begin{itemize}
 \item 
 the intervals $J$ of $\cP$ with $\iota(J)=w$ are the components of $R_\infty$ contained in 
 $[\ell\, T_n,(\ell+1)\,T_n-1]$;
 \item 
 the complement of $R_\infty$ in $[\ell \, T_n,(\ell+1)\, T_n-1]_{\NN}$ can be written as the union of intervals $J$ of the type $[k\, T_0,(k+1)\, T_0-1]$, these intervals are  the
 elements of the partition $\cP$ with $\iota(J)=r$. 
\end{itemize}
Then $\fP=(\cP,\iota)$ defines a $T_n$-pattern in $I$.
\end{lemm}

\begin{proof}
To prove the lemma it is enough to recall that, by definition of a $n$-good interval,
there is no component of $R_\infty$ containing
the interval $[\ell \, T_n,(\ell+1)\, T_n-1]_{\NN}$ 
and that every $T_k$ is a multiple of $T_0$.
\end{proof}

The pattern $\fP$ in Lemma~\ref{l.tailpattern} is called  \emph{the pattern induced by the tail $R_\infty$}  
in the good interval $[\ell \,T_n,(\ell+1)\, T_n-1]$  and is denoted by $\fP_{n,\ell}$,
or by $\fP_{n,\ell}(R_\infty)$ (for emphasising the role of the tail).

The  next remark associates a sequence of patterns to the tail $R_\infty$.

\begin{rema}\label{r.initial}[Initial patterns for a long 
tail]
With the notation of Lemma~\ref{l.tailpattern},
by definition of the tail $R_\infty$,  the initial interval of length $T_n$,
 $[0,T_n-1]_{\NN}$, 
 is a good interval.
We let $\fP_n=\fP_{n,0}$ and call it
 \emph{the initial $T_n$-pattern of $R_\infty$}. 

The set of initial points of  $\mathfrak{P}_n$ consists of the following points:
$$
\Big( \{\mbox{origins of the components of $R_\infty$}\}
\cup
\{ kT_0\not\in R_\infty\}\Big) \cap [0, T_n-1].
$$

\end{rema}

\begin{rema}\label{r.initialcompatibility}
[Compatibility of induced patterns]
For every $i<n$  the restriction of 
the initial pattern 
$\mathfrak{P}_n$ to the interval $[0,T_{i}-1]_{\NN}$ is the initial pattern $ \mathfrak{P}_i$.
In other words,  $\mathfrak{P}_i$ is the \emph{initial $T_i$-subpattern} of $\mathfrak{P}_n$.
\end{rema}

\subsection{Concatenations and controlled plaque-segments}
\label{ss.concatenations}
Consider a compact metric space $(X,d)$, a homeomorphism $f\colon X\to X$, and
an open set $U$ of $X$. Consider a family $\cD$ of compact sets contained in  $U$. We call the
elements in $\cD$ plaques\footnote{In our applications, the elements of $\cD$ are 
sets in a flip-flop family, recall Definition~\ref{d.flipflop}.}.

Given a pair of plaques $D_0,D_1\in \cD$
we say $(D_0,D_1)$ is a \emph{plaque-segment of size $T$ relative to $U$ and $\cD$} if:
\begin{itemize}
 \item  $f^{-i}(D_1)\subset U$ for every $i\in\{0,\dots, T\}$ and
 \item $f^{-T}(D_1)\subset D_0$. 
\end{itemize}
We say that
$D_0$ is the  \emph{origin of the segment} and $D_1$ is the \emph{end of the segment.} 

\begin{figure}
\begin{minipage}[h]{\linewidth}
\centering
 \begin{overpic}[scale=.4, 
 ]{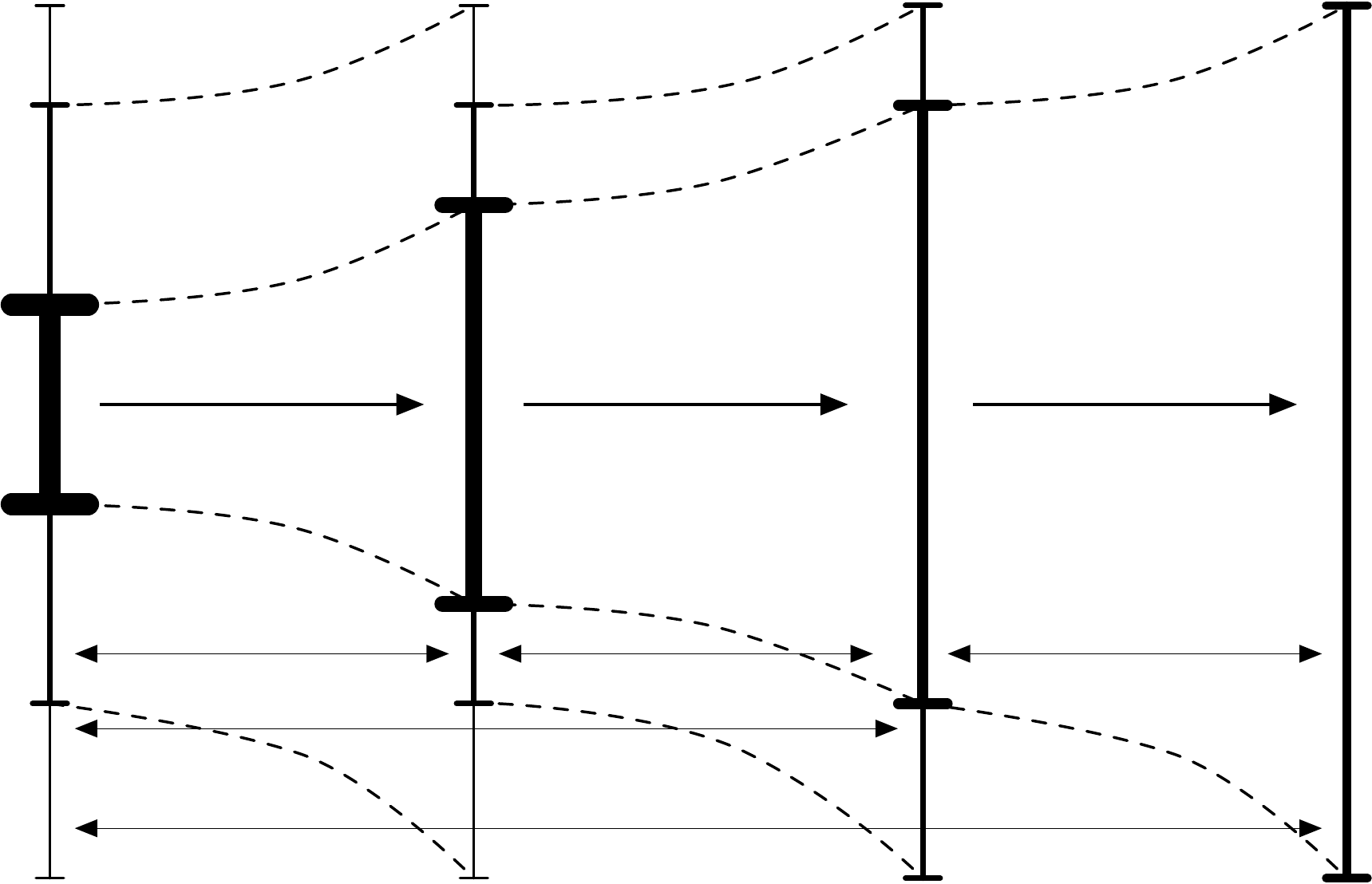}
			\put(5,63){\small$D_0$}
			\put(36,63){\small$D_1$}
			\put(69,63){\small$D_2$}
			\put(100,63){\small$D_3$}
			\put(18,36){\small$L_0$}
			\put(50,36){\small$L_1$}
			\put(82,36){\small$L_2$}
			\put(10,18){\tiny$(D_0,D_1)$}
			\put(42,18){\tiny$(D_1,D_2)$}
			\put(75,18){\tiny$(D_2,D_3)$}
			\put(18,12){\tiny$(D_0,D_2)$}
			\put(40,6){\tiny$(D_0,D_3)$}
  \end{overpic}
\caption{Concatenations}
\label{f.concatenations}
\end{minipage}
\end{figure}
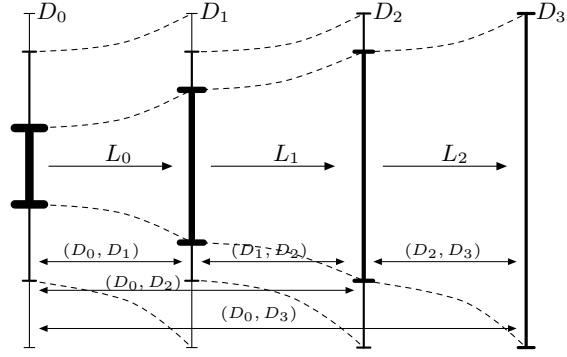

Let $(D_0,D_1)$ and $(D_1,D_2)$ be two plaque-segments of lengths $L_0$ and $L_1$, respectively, 
relative to $U$ and $\cD$.  Then  
$(D_0,D_2)$ is a plaque-segment of length $L_0+L_1$, 
 called \emph{the concatenation of $(D_0,D_1)$ and $(D_1,D_2)$}.
 We use the notation 
 $(D_0,D_2)\eqdef (D_0,D_1)\ast(D_1,D_2)$. See Figure~\ref{f.concatenations}.

\begin{defi}\label{d.plaquecontrolled}
Let $(X,d)$ be compact metric space, $f\colon X\to X$ a homeomorphism,
$U$ an open set of $X$, 
$\varphi\colon U\mapsto \RR$ a continuous map,
and $\cD$ a family of plaques contained in $U$.

Consider  $T\in \NN$ and a subset  $J\subset\RR$.
A plaque-segment
$(D_0,D_1)$ of length $T>0$  relative to $U$ and $\cD$ is called \emph{$(J,T)$-controlled} if 
  $$
 \varphi_T(x)=\frac 1T\sum_{i=0}^{T-1}\varphi(f^i(x))\in J,
 \quad
 \mbox{for every $x\in f^{-T}(D_1) \subset D_0$.}
 $$
 \end{defi}
 When there is no ambiguity on the pair $U$ and $\cD$ the dependence on these sets will be  omitted.

\begin{defi}
Let $(X,d)$ be compact metric space, $f\colon X\to X$ a homeomorphism,
$U$ an open set of $X$, 
$\varphi\colon U\mapsto \RR$ a continuous map,
and $\cD$ a family of plaques contained in $U$.

Consider a 
 a scale $\cT=(T_n)_{n\in \NN}$, a
$T_n$-pattern $\mathfrak{P}$
of  $[\ell\, T_n,(\ell+1)\, T_n-1]_{\NN}$, $n>0$, the set
$M(\mathfrak{P})$ of its marked points, and
a family of subsets
$\cJ=(J_i)_{i\in \NN}$ of $\RR$.

A family $\{D_i\}_{i\in M(\mathfrak{P})}$ of plaques  of $\cD$
is called \emph{$(\cJ,\mathfrak{P})$-controlled} 
(relatively to $U$ and $\cD$)
if:
\begin{itemize}
 \item 
 For every $\mathfrak{P}$-admissible interval 
 $[k\, T_i,(k+1)\, T_i-1]$ the pair   $(D_{kT_i},D_{(k+1)T_i})$ is a 
 plaque-segment of length $T_i$ that is
  $(J_i,T_i)$-controlled (relative to $U$ and $\cD$). 
 \item 
  For any $i,j\in  M(\mathfrak{P})$, $i<j$, the pair $(D_i,D_j)$ is a plaque-segment of length $j-i$
  (relative to $U$ and $\cD$).
\end{itemize}
\end{defi}

\subsection{Distortion of Birkhoff averages and concatenations in flip-flop families}
\label{ss.distortionfobirkhoffaverages}

The following result is a translation of \cite[Lemma 2.4]{BoBoDi2}
to the context of flip-flop families with sojourns.  Recall the notation for Birkhoff averages in 
\eqref{e.birkhoffaverages}. 

\begin{lemm}[Small distortion of Birkhoff averages over long concatenations]
\label{l.distortion}
Let \newline
$f \colon X \to X$ be a homeomorphism, $\varphi\colon X\to \RR$ a continuous function,
and $\fF$ 
a flip-flop family associated to $\varphi$ and $f$ with sojourns in a compact set $Y$.  

Then for every $\alpha>0$ there exists $t = t(\alpha) \in \NN$ with the following property:
Consider any $T\ge t$ and  any  family of plaques $\{D_i\}_{0\leq i\leq T}$ of $\fF$ such that
for every $i=0, \dots, T-1$.
$(D_i, D_{i+1})$ is a plaque-segment of length $L_i$.
Then the plaque-segment 
$$
(D_0,D_T)= (D_0,D_1)\ast(D_1,D_2)\ast \cdots \ast (D_{T-1},D_T)
$$
 satisfies
$$
 \left| \varphi_L (x) - \varphi_L(y) \right| < \alpha,  \quad \mbox{where} \quad L\eqdef \sum_{i=0}^{T-1} L_i,
$$
for every pair of points $x,y\in D_0$
such that 
$$
f^{L_{i-1}'}(x), f^{L_{i-1}'}(y) \in D_i, \quad \mbox{where} \quad L_{i-1}^\prime \eqdef\sum_{j=0}^{i-1} L_j
$$
for every $i=0,\dots, T-1$.
\end{lemm}

The proof is the same as the one of \cite[Lemma 2.4]{BoBoDi2} and the key ingredient is the 
expansion properties  in item \eqref{i.flipflop33}  in Definitions~\ref{d.flipflop} and \ref{d.flipfloptail}.
We omit this proof and refer to  \cite{BoBoDi2}.

%
%
%
%

\subsection{Flip-flop families and concatenations}
\label{ss.flip-flopfamiliesofplaques}

The aim of this section  is to prove the following theorem.

\begin{theor}\label{t.p.flipfloppattern} 
Let $(X,d)$ be a compact metric space,
$Y$ a compact subset of $X$,
$U$ an open subset of $X$,
$f\colon X \to X$ a homeomorphism, 
$\varphi\colon U \to \mathbb{R}$ a continuous function,
and
$\fF$ a  flip-flop family with sojourns along  $Y$ associated to $\varphi$ and $f$ 
whose plaques are contained in $U$.

Consider  sequences $(\delta_n)_{n\in \NN}, (\alpha_n)_{n\in \NN}$, and 
$(\beta_n)_{n\in \NN}$ of positive numbers such that:
$$
(\delta_n)_n \to 0
\quad
\mbox{and}
\quad
 \alpha_{n+1}< \dfrac{\alpha_n}{4}<\beta_n<\dfrac{\alpha_n}{2}.
 $$

Then there is  a scale $\cT=(T_n)_{n\in \NN}$ satisfying the following properties:
For every plaque $D\in \fF$, 
every  $T_n$-pattern $\mathfrak P= (\cP, \iota)$, 
and every $\omega \in \{+,- \}$
there is a family of plaques  $\cD_{\mathfrak{P} } = \{D_a\}_{a\in M(\mathfrak{P})}$
of $\fF$ such that :
\begin{enumerate}[{(I}1{)}]
\item \label{i.I1}
$D_0= D$;
\item \label{i.I2}
the family $\{D_a\}_{a\in M(\mathfrak{P})}$ is
$(\cJ_n,\mathfrak{P})$-controlled 
(relatively to $U$ and  $\mathfrak{P}$) 
where 
$\cJ_n=\{J_i\}_{i\in\{0,\dots,n\}}$ and
$$
J_i\eqdef \left[ -\alpha_i, -\frac{\alpha_i}2 \right] \cup \left[\frac{\alpha_i}2, \alpha_i\right] \mbox{ for } i<n
\mbox{ and }
J_n\eqdef \omega\, \left[ \frac{\alpha_n}2, \alpha_n\right];
$$
\item \label{i.I3}
 if $[a,a+T_{i}-1]$ is an interval of the partition 
 {$\cP$} of $w$-type then  for every point $x\in f^{-T_i}(D_{a+T_i})$ 
 the segment of orbit $\{x,\dots, f^{T_i}(x)\}$ is $\delta_i$-dense in $Y$.
 \end{enumerate}
\end{theor}

We say that the family of plaques $\cD_{\mathfrak{P} } = \{D_a\}_{a\in M(\mathfrak{P})}$
in  Theorem~\ref{t.p.flipfloppattern} is 
\emph{$((\mathcal{J})_n,\mathfrak P)$-controlled and starts at
$D_0$}.

\begin{proof}
The construction of the scales $\cT$  in the theorem is done by induction on $n$
(assuming that $T_i$ is defined  for $i\leq n$ we will
define $T_{n+1}$).
The proof considers two cases: either the pattern is trivial or it is a concatenations of 
$T_n$-patterns. Proposition~\ref{p.l.induction1} deals with trivial patterns while
Proposition \ref{p.l.induction2} deals with concatenations.

Note that the  
Theorem~\ref{t.p.flipfloppattern} claims the existence of a  scale  $\cT=(T_n)_{n\in \NN}$ that holds
for every disk $D\in \fF$.
In the proofs of Propositions~\ref{p.l.induction1}
and \ref{p.l.induction2} we get such a number depending on the 
plaque $D\in \fF$ and uniformly bounded. To uniformize this number for
every plaque in $\fF$ we will use Lemma~\ref{l.necessario} below. Recall 
the definition of the distortion time number $t(\alpha)$
in Lemma~\ref{l.distortion} associated to $\alpha>0$.

\begin{lemm}[Uniformization]
\label{l.necessario}
Take a scale $(T_n)_{n\in \NN}$ and a sequence $(\alpha_n)_{n\in \NN}$ as in 
Theorem~\ref{t.p.flipfloppattern}. Then
for every plaque $D_0\in \fF$,  
every $\omega_0 \in\{-,+\}$,
and  every
$$
t\ge \tau_n \eqdef \max \left\{ t\left(\frac{\alpha_{n+1}}{6}\right), 6 \,  \frac{\alpha_n\, T_n}{\alpha_{n+1}}\right\}
$$
the following property holds:

Let $D_1\in\fF$ be a plaque  such that $(D_0,D_1)$ is
a
$(\omega_0 \, [\frac{\alpha_{n+1}}2,\alpha_{n+1}],t)$-con\-trolled. 
plaque-segment.

Then there is 
 $\omega \in \{-,+\}$ such that
if $(D_1, D_2)$
 is a plaque-segment that   is $(\omega \, [\frac{\alpha_{n}}2,\alpha_{n}],T_n)$-controlled
then the concatenation $(D_0,D_2)= (D_0,D_1)\ast (D_1,D_2)$  is
$(\omega_0 \, [\frac{\alpha_{n+1}}2,\alpha_{n+1}],t+T_n)$-controlled. 
\end{lemm}

\begin{proof} 
We prove the lemma for $\omega_0=+$, the case $\omega_0=-$ is analogous. 
Note first that
the choice of $t$ implies that
the distortion of $\varphi_t$ 
in $f^{-t}(D_1)$
is bounded by $\frac{\alpha_{n+1}}{6}$.
Moreover,  as the orbit segment
$(D_0,D_1)$ is 
$([\frac{\alpha_{n+1}}2,\alpha_{n+1}],t)$-controlled,
we have
$$
\varphi_{t}(f^{-t}(D_1))\subset\left[ \frac{\alpha_{n+1}}2,\alpha_{n+1} \right].
$$
Now there are two cases:
\begin{enumerate}
 \item 
 if 
 $\displaystyle{
 \max_{x\in f^{-t}(D_1)}\varphi_t(x)}\leq \frac{5\, \alpha_{n+1}}{6},
$ 
we choose $\omega=+$,
 \item
 otherwise 
 $\displaystyle{
 \min_{x\in f^{-t}(D_1)}\varphi_t(x)\geq \frac{4\, \alpha_{n+1}}{6}}$
 and we choose $\omega=-$.
\end{enumerate}
In the first case, as  $\omega=+$ one can easily check that 
\begin{equation}
\label{e.easytocheck}
\varphi_{t+T_n}(x)>\varphi_t(x) \quad \mbox{for every $x\in f^{-(t+T_n)}(D_2)$.}
\end{equation}
Moreover, as $t>6 \, \frac{\alpha_n\, T_n}{\alpha_{n+1}}$ 
and $(D_1, D_2)$ is
 $([\frac{\alpha_{n}}2,\alpha_{n}],T_n)$-controlled
we have
$$
\varphi_{t+T_n}(x)<\dfrac{\frac 56\, \alpha_{n+1}t +T_n\alpha_n}{t+T_n}< 
\frac{t \, \alpha_{n+1}}{t+T_n}<
\alpha_{n+1} 
\quad \mbox{for every $x\in f^{-(t+T_n)}(D_2)$.}
$$ 
Finally, recalling that $(D_0,D_1)$ is  
$([\frac{\alpha_{n+1}}2,\alpha_{n+1}],t)$-controlled and \eqref{e.easytocheck}
$$
\frac{\alpha_{n+1}}{2}<
\varphi_t(x)<
\varphi_{t+T_n}(x)< \alpha_{n+1}
$$
ending the proof in the first case. 

Case (b) is analogous and hence omitted.
The proof of the lemma is now complete.
\end{proof}

\begin{prop}[Trivial patterns]
\label{p.l.induction1} 
Under the assumptions of Theorem~\ref{t.p.flipfloppattern},
assume that for every $i=0, \dots, n$  there are defined natural numbers $T_i$ such that  
the conclusions in the
theorem hold for $T_i$-patterns. 

Then there is $\widetilde T_{n+1}$ such that 
for every  $T> \widetilde{T}_{n+1}$,
every plaque $D\in \fF$, and every $\omega\in \{+,-\}$
there is a plaque $D_0\in \fF$ such that
\begin{itemize}
\item
$f^{-T}(D_0)\subset D$,
\item
for every $x\in f^{-T}(D_0)$ the set
$\{x,f(x),\dots, f^{T}(x)\}$ is $\delta_{n+1}$-dense  in $Y$, and
\item
 $(D_0,D)$ is 
   $(\omega \, [\frac{\alpha_{n+1}}2, \alpha_{n+1}],T)$-controlled.
\end{itemize}
\end{prop}

\begin{proof}
 We only present the proof for the case $\omega=+$, the case $\omega=-$ is analogous and thus omitted. 
Let $N=N_{\delta_{n+1}}$ as in
Definition~\ref{d.flipfloptail}. 
Then for every plaque $D\in \fF$ 
there is a $\widetilde D_0\in \fF$ such that
\begin{itemize}
\item
{$f^{-i}(\widetilde D_0)\subset U$ for all $i=0,\dots, N$,}
\item
$f^{-N} (\widetilde D_0)\subset D$, and
\item
for every $x\in f^{-N} (\widetilde D_0)$ the set
$\{x,f(x),\dots, f^N(x)\}$ is $\delta_{n+1}$-dense in $Y$.
\end{itemize}

In what follows we will focus on the control of Birkhoff averages.
Note that the average of $\varphi_N$ in 
$f^{-N} (\widetilde D_0)$ is uniformly upper bounded by the maximum of 
$|\varphi |$ in $X$ denoted by
$ \max |\varphi |$.

Recall the definition of  $t(\alpha)$
in Lemma~\ref{l.distortion} and 
fix $k_0$ large enough such that
\begin{equation}
\label{e.k0first}
k_0 > t \left( \frac{\alpha_{n+1}}{6}\right)\, \frac{1}{T_n}.
\end{equation}
Since $\beta_n<\frac{\alpha_n}{2}$ we have that if $k_0$ is large enough then
for every $k\ge k_0$ it holds
\begin{equation}
\label{e.k0second}
\beta_n<
\frac{- N \max |\varphi| + k \,  T_n \dfrac{\alpha_n}{2}}{N+ k\, T_n}<
\frac{N \max |\varphi| + k   \, T_n\,  \alpha_n}{N+ k \, T_n}<
3\, \frac{\alpha_n}{2}.
\end{equation}


\begin{clai}\label{c.icerm}
Consider  $k_0$ satisfying equations \eqref{e.k0first} and \eqref{e.k0second}.
Then for every  $k\ge k_0$ 
 there is plaque $\widetilde D_1\in \fF$
such that
\begin{enumerate}
\item
$f^{-kT_n} (\widetilde D_1)\subset \widetilde D_0$;
\item
for every $x\in 
f^{-k\,T_n-N} (\widetilde D_1)$ it holds
$$
\frac{1}{N+k\, T_n}\, \sum_{j=0}^{N+k\, T_n-1}
\varphi (f^j(x))\in \, \left[ \beta_{n}, \frac{3\,\alpha_{n}}2 \right];
$$
\item 
for every $x,y\in 
f^{-kT_n-N} (\widetilde D_1)$.
$$
\frac{1}{N+k \, T_n}\, \sum_{j=0}^{N+k\, T_n-1}
|\varphi (f^j(x))-\varphi (f^j(y))|
< \frac{\alpha_{n+1}}{6}.
$$
\end{enumerate}
\end{clai}

\begin{proof}
To prove the first item we 
consider  the concatenation of $k$ 
orbit-segments 
$(\widehat{D}_i, \widehat D_{i+1})$, $i=0,\dots, k-1$, 
of size $T_n$  associated to $\omega_i=+$ given by
 the induction hypothesis and such that
 $\widehat D_0=\widetilde D_0$.  
 These pairs  are obtained inductively: 
 assumed defined the pair  $(\widehat D_i,\widehat D_{i+1})$ we  apply 
 induction hypothesis to the final plaque $\widehat D_{i+1}$. 
 To conclude it is enough to take $\widetilde D_1=\widehat D_{k}$.
 See Figure~\ref{f.claimicerm}.

To prove item (b) note that
each plaque-segment  $(\widehat D_i, \widehat D_{i+1})$ 
is $([\frac{\alpha_{n}}2,\alpha_{n}],T_n)$-controlled. Hence
the averages in these segments belong to 
$[\frac{\alpha_n}{2},\alpha_n]$.  Now the control of averages follows  from  the choice of $k_0$ in
 equation~\eqref{e.k0second}.

Item (c) follows from the distortion Lemma~\ref{l.distortion}
 and the choice of $k_0$ in
equation~\eqref{e.k0first}. This ends the proof of the claim.
\end{proof}
%
%

\begin{figure}
\begin{minipage}[h]{\linewidth}
\centering
 \begin{overpic}[scale=.35, 
 ]{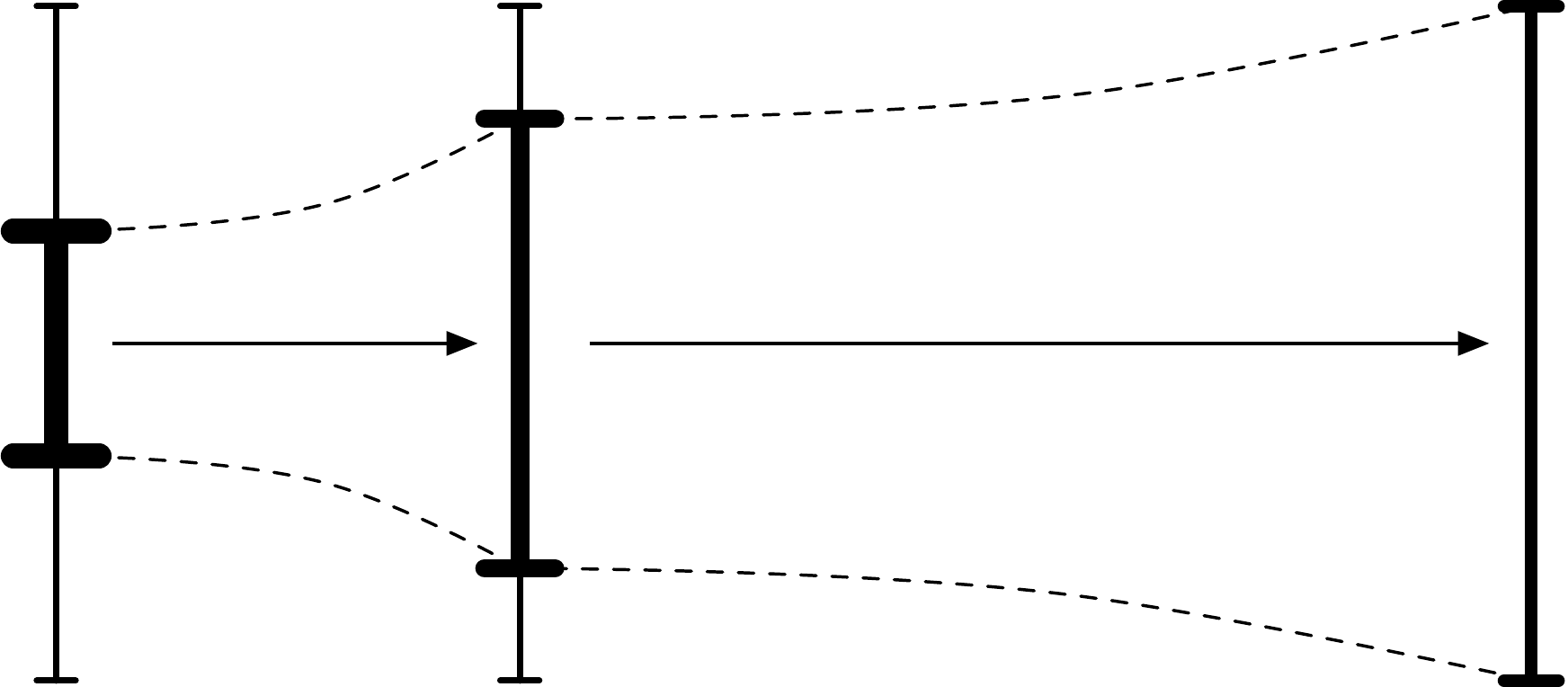}
			\put(-3,43){\small$D$}
			\put(88,43){\small$\widetilde{D}_1$}
			\put(24,43){\small$\widetilde{D}_0$}
			\put(18,23.5){\small$N$}
			\put(60,23.5){\small$k \, T_n$}			
  \end{overpic}
\caption{Proof of Claim \ref{c.icerm}}
\label{f.claimicerm}
\end{minipage}
\end{figure}

The proof of the  proposition  now follows arguing exactly as in 
the concatenation result
\cite[Lemma~2.12]{BoBoDi2}.
For completeness, we recall the main arguments involved in this proof.

Note that in Claim~\ref{c.icerm} we can assume that the constant $k_0$
is such that 
\begin{equation}\label{e.choiceofk}
k_0> \max\left\{ 
{\dfrac{6\,\alpha_n}{\alpha_{n+1}}}, t \left( \frac{\alpha_{n+1}}{6}\right)\,\frac{1}{T_n}
\right\}.
\end{equation}

\begin{lemm}
\label{l.inextremis}
For every $k\ge k_0$
there are $i_0$ and a plaque
$D_1\in \fF$ 
such that
$(D,D_1)$ is
$(\big[ \frac{1}{2} \alpha_{n+1},\alpha_{n+1} \big], \widetilde T_{n+1} )$-controlled, 
where
$$
\widetilde T_{n+1}=N+(k+i_0)T_n.
$$
\end{lemm}

\begin{proof}
We just describe the main steps of the proof.
Take $\widehat D_0\eqdef \widetilde D_1$, where $\widetilde D_1$ is the plaque given by 
Claim~\ref{c.icerm}.
By the  induction hypothesis, given any $i\in \{0, \dots, n\}$ 
there is a family of 
 $([-\alpha_n,-\frac{\alpha_n}{2}], T_n)$-controlled
 orbit-segments $(\widehat D_{j},\widehat D_{j+1})$ for $j=0,\dots, k+i-1$. 
 This implies that  
$$
\varphi_{T_n} \big( f^{-T_n}(\widehat D_{j+1}) \big) \subset
\left[-\alpha_n,-\frac{\alpha_n}2\right].
$$ 

 Consider the orbit-segment $(D,\widehat D_{k+i})$ 
 of length $N+(k+i)\,T_n$ obtained concatenating 
 $(D,\widehat D_0=\widetilde D_1)$ 
 and 
 $(\widehat D_j,\widehat D_{j+1})$, $j=0,\dots, k+i-1$, that is,
 $$
 (D,\widehat D_{k+i})=(D,\widehat D_0)
 \ast
 (\widehat D_0,\widehat D_{1})\ast \cdots \ast (\widehat D_{k+i-1},\widehat D_{k+i}).
 $$
 
 The choice of $k_0$ in \eqref{e.choiceofk} 
 ($k_0\,\alpha_{n+1}> 6 \, \alpha_n$)
  immediately implies that  every
  $x\in f^{-N-(k+i)T_n}(\widehat D_{k+i})$ satisfies the  following implication, 
 \begin{equation}\label{e.thefollowingholds}
 \varphi_{N+(k+i)T_n}(x)\geq \frac 56 \alpha_{n+1}
 \quad 
 \implies
 \quad
 \varphi_{N+(k+i+1)T_n}(x)\geq \frac 46\alpha_{n+1}.
 \end{equation}
 Let 
 \begin{equation*}
 \ell_0\eqdef \dfrac{3\, (N+k\, T_n) \, \alpha_n}{T_n\, (\alpha_n+\alpha_{n+1})}.
 \end{equation*}
 By item (b) in Claim~\ref{c.icerm} we have that 
 $$
 \beta_n \le 
 \varphi_{N+k\, T_n}(x)\le \frac{3\,\alpha_n}{2}
 \quad \mbox{for every $x\in f^{-N-k\,T_n}(\widehat D_0)$.}
 $$
A simple calculation gives that 
for every $\ell > \ell_0$ 
 it holds
 \begin{equation}\label{e.needtoexplain}
 \varphi_{N+(k+\ell)\,T_n}(x)<\frac 12\, \alpha_{n+1}
 \quad \mbox{for every} \quad
 x\in f^{-N-(k+\ell)\, T_n}(\widehat D_{k+\ell}).
\end{equation}
  Equation
 \eqref{e.needtoexplain} implies that there is a first $i_0$ with 
 $\varphi_{N+(k+i_0)\, T_n} (\bar x) \le \frac{5}{6} \, \alpha_{n+1}$ for some $\bar x
 \in f^{-N-(k+i_0)\,T_n}(\widehat D_{k+i_0})$.
 From Equation \eqref{e.thefollowingholds} we get
 \begin{equation}
 \label{e.previousterca}
 \varphi_{N+(k+i_0)\,T_n}(\bar x)\in
 \left[ \frac 46\alpha_{n+1},\frac 56\alpha_{n+1} \right].
 \end{equation}
 By the choice of $k_0$ in ~\eqref{e.k0first}, 
 the distortion of  
 $\varphi_{N+(k+i_0)T_n}$ in  $f^{-N-(k+i_0)\, T_n}(\widehat D_{k+i_0})$ 
 is strictly less than
 $\frac 16 \alpha_{n+1}$. Equation~\eqref{e.previousterca} now implies that 
 $$
 \varphi_{N+(k+i_0)\,T_n}(x)\in\left[ \frac 12\alpha_{n+1},\alpha_{n+1}\right]
 \quad
\mbox{for every $x\in f^{-N-(k+i_0)\,T_n}(\widehat D_{k+i_0})$.} 
 $$
 The lemma follows taking $D_1=\widehat D_{k+i_0}$.
\end{proof}

Take $i_0$ as in Lemma~\ref{l.inextremis}, 
$k_0$ as in \eqref{e.choiceofk}, $k \ge k_0$ sufficiently large, and define
$$
\widetilde T_{n+1}\eqdef N+(k+i_0)\, T_n> \tau_n,
$$
where $\tau_n$ is as in 
Lemma~\ref{l.necessario}. 
Consider the plaque $D_1$ given by Lemma~\ref{l.inextremis} such that
$(D,D_1)$ is
$(\big[ \frac{1}{2} \alpha_{n+1},\alpha_{n+1} \big], \widetilde T_n )$-controlled.
Using the induction hypothesis, consider a plaque $D_2$ such that
$(D_1,D_2)$ is $(\big[ \frac{1}{2} \alpha_{n},\alpha_{n} \big], T_n )$-controlled.
As $\widetilde T_{n+1}>\tau_n$, Lemma~\ref{l.necessario} 
implies that $(D_0,D_2)$ is 
$(\big[\frac{1}{2} \alpha_{n+1},\alpha_{n+1} \big], \widetilde T_{n+1}+T_n)$-controlled.
Repeating this last argument $j$ times (any $j$)
we get
a plaque $D_2(j)$ 
such that $(D,D_2(j))$ is 
$(\big[\frac{1}{2} \alpha_{n+1},\alpha_{n+1} \big], \widetilde T_{n+1}+j\,T_n)$-controlled.
%
This completes the average control in the proposition and ends 
its proof. 
\end{proof}


Given a pattern  $\fP$ and its set of marked points $M(\fP)$ we let
$e_{ M(\fP)}$ the right extreme of $M(\fP)$.

\begin{prop}[Concatenation of patterns] 
\label{p.l.induction2} 
Under the assumptions of Theorem~\ref{t.p.flipfloppattern},
assume that  for every $i=0,\dots, n$ there is defined  $T_i\in \NN$ satisfying
the conclusions  in the theorem.

Then there is a constant  $k_0$ such that for every $k\ge k_0$,
every family $\{ \mathfrak{P}_i\}_{i=1}^k$ of $T_n$-patterns,
every $\omega \in \{-,+\}$, 
and every plaque $D\in \mathfrak{F}$ there is a sequence
of symbols $(\omega_i)_{i=1,\dots,k}$, $\omega_i\in \{-,+\}$, 
satisfying the following property:

Consider  the family of sets $\cJ_i=\{J_{i,j}\}$,
$i=1,\dots,k$  and $j=1,\dots,n$,  defined
by
\begin{itemize}
\item
$J_{i,j}\eqdef \big[-\alpha_j, \frac{-1}{2}\,\alpha_j\big] \cup
\big[\frac{1}{2}\, \alpha_j, \alpha_j\big]$  if $j<n$ and
\item
$J_{i,n}\eqdef \omega_i\, \big[ \frac{1}{2}\, \alpha_n, \alpha_n \big]$.
\end{itemize}

Let  $\cD_{\mathfrak{P}_1}=\{D_{1,j}\}_{j\in M(\mathfrak{P}_1)}$ be 
a family  of $(\cJ_1,\mathfrak{P}_1)$-controlled plaques given by the 
induction hypothesis associated to 
the plaque $D_{1,0}=D$.

Let $\cD_{\mathfrak{P}_i}=\{D_{i,j}\}_{j\in M(\mathfrak{P}_i)}$ be the family of 
$(\cJ_i,\mathfrak{P}_i)$-controlled plaques associated to 
the final plaque 
$D_{i-1,e_{ M(\fP_{i-1})}}=D_{i,0}$ 
of the family $\cD_{\mathfrak{P}_{i-1}}=\{D_{i-1,j}\}_{j\in M(\mathfrak{P}_{i-1})}$ given by the induction hypothesis.

Then  the plaque-segment
$(D, D_{k,e_{ M(\fP_k)}})$ is  $(\omega\, [ \frac{1}{2}\,  \alpha_{n+1}, 
\alpha_{n+1} ],k\,T_n)$-controlled. 
\end{prop}

\begin{proof}
The proof of follows arguing  as in 
\cite[Lemma~2.12]{BoBoDi2}. We now recall the strategy for $\omega=+$ (the case 
$\omega = -$ is analogous and thus  omitted). 
We start by taking a sequence of  signals $\omega_i=+$ for every $i$. This sequence is  sufficiently large 
to guarantee distortion smaller than $\frac{1}{6}\alpha_{n+1}$
in the pre-images of the plaques. We stop for some large $i_0$.
Note that the averages are
in $\big[ \frac{1}{2}\alpha_n, \alpha_n \big]$. 
Thereafter for $i>i_0$ 
we consider a sequence of $\omega_i=-$ as in the proof of the previous lemma. 
In this way the averages of $\varphi$ start to decrease.  We stop 
when this average at some point of the plaque belongs to 
$\big[ \frac{4}{6}\alpha_{n+1},\frac{5}{6}\alpha_{n+1} \big]$
(the key point is that the averages can not jump from above $\frac{5}{6}\alpha_{n+1}$
to below $\frac{4}{6}\alpha_{n+1}$, this is because for large $i$ the cone
of size $\big[ \frac{4}{6}\alpha_{n+1},\frac{5}{6}\alpha_{n+1} \big]$ is very width, this is exactly the argument in \cite{BoBoDi2}).
The distortion control implies that the average of $\varphi$ in the 
whole pre-image of the plaque is contained in 
$\big[\frac{1}{2} \alpha_{n+1},\alpha_{n+1}\big]$. 

We conclude the proof using Lemma~\ref{l.necessario}: 
we can continue concatenating plaque-segments keeping the averages in 
$\big[ \frac{1}{2}\alpha_{n+1},\alpha_{n+1} \big]$.   This completes  the sketch of the proof of the 
proposition.
\end{proof}

\smallskip

\noindent{\emph{End of the proof of Theorem~\ref{t.p.flipfloppattern}.}}
The definition of  the scale $(T_n)_n$ is done inductively on $n$. 
Assuming that $T_i$ is defined  for $i\leq n$, we
define $T_{n+1}$  as follows.
Take  $\widetilde T_{n+1}$  as in Lemma~\ref{l.inextremis} and
$k_0$ as in Claim~\ref{c.icerm}.
Then define $T_{n+1}$ as  a multiple of $T_n$ such that
$$
T_{n+1}\ge \max\{ \widetilde T_{n+1},k_0 \, T_n\}
\quad
\mbox{and}
\quad 
\frac{T_{n+1}}{T_n}\ge
(n+1)\, \frac{T_n}{T_{n-1}}.
$$

Take now  a  
$T_{n+1}$-pattern $\mathfrak P= (\cP, \iota)$,
$\omega \in \{+,- \}$, and a plaque $D\in \mathfrak{F}$.
There are two cases to consider. 
If $\cP$ is the trivial $T_{n+1}$-pattern we just need to apply Proposition~\ref{p.l.induction1}.
Otherwise, the $T_{n+1}$-pattern  $\mathfrak{P}$ is a concatenation 
of a sequence of at least $k_0$ $T_{n}$-patterns. The theorem follows by applying 
Proposition~\ref{p.l.induction2} to this family of $T_n$-patterns.
\end{proof}

Bearing in mind Remark~\ref{r.initialcompatibility}, 
we are interested 
to get
plaque-segments associated  to $T_n$-patterns 
respecting  the plaque-segments associated to its initial subpatterns. 
A slight variation of the proof of 
Proposition~\ref{p.l.induction2} implies the following addendum:

\begin{adde}[Extension of initial subpatterns]
\label{p.extension} With the hypotheses of Theorem~\ref{t.p.flipfloppattern}, 
the scale  $\cT=(T_n)_{n\in \NN}$ can be chosen satisfying the 
following additional property: 

Let $\mathfrak P_1$ be  a $T_{n}$-pattern   and 
$\mathfrak P_2$ be a $T_{n+1}$-pattern such 
that $\mathfrak P_1$ is the initial 
$T_n$-subpattern of $\mathfrak P_2$.  

Consider a flip-flop family $\mathfrak{F}$ and 
$D_0$ a plaque $\mathfrak{F}$.
Let $\{D_a\}_{a\in M(\mathfrak P_1)}$ be a family of plaques 
associated to the pattern $\mathfrak P_1$ starting at $D_0$  given by 
Theorem~\ref{t.p.flipfloppattern}.

Then for every $\omega\in \{+,-\}$, there is a family $\{\widetilde D_a\}_{a\in M(\mathfrak P_2)}$ of plaques 
associated to the pattern $\mathfrak P_2$ satisfying the conclusion of 
Theorem~\ref{t.p.flipfloppattern}, starting at $D_0$ and  such that  
$$
\widetilde D_a= D_a, \quad \mbox{for every $a\in M(\mathfrak P_1)$}.
$$
\end{adde}

This result allows us  to 
choose the family of plaques associated to a pattern extending
the ones associated to its initial subpatterns.

\section{Flip-flop families with sojourns: 
proof of Theorem~\ref{t.flipfloptail} }\label{s.flipfloptail}

In this section we prove Theorem~\ref{t.flipfloptail}, Corollary~\ref{c.flipfloptail}, and
Proposition~\ref{p.r.hyperboliclike}.

\subsection{Proof of Theorem~\ref{t.flipfloptail}}
Consider a homeomorphism $f\colon X \to X$  defined on a compact metric space $(X,d)$, 
a continuous function $\varphi\colon X \to \RR$, 
and a flip-flop family $\fF=\fF^+\bigsqcup\fF^-$  
with sojourns along a compact subset $Y$ of $X$ associated to $\varphi$ and $f$.
We need to see that
every plaque $D\in \fF$ contains a point $x_D$ that is 
controlled at any scale 
with a long sparse tail with respect to $\varphi$ and $Y$.

Fix a sequences of strictly positive numbers $(\alpha_n)_n$ 
and $(\beta_n)_n$ such that
$$
0<\alpha_{n+1}< \frac{\alpha_n}{4}<
\beta_n<\frac{\alpha_n}2.
$$
Associated to these sequences we consider the family of control 
intervals $\cJ_n=\{J_i\}_{i\in\{0,\dots,n\}}$ defined as in Theorem~\ref{t.p.flipfloppattern}.
We also take 
an arbitrary sequence of positive numbers $(\delta_n)_n$ converging to $0$.
Denote by $\cT=(T_n)_{n\in \NN}$ the scale associated to these sequences given  by Theorem~\ref{t.p.flipfloppattern}.

By Lemma~\ref{l.tailexistence} there are a sequence $\bar \epsilon=(\varepsilon_n)_n$ and  
a $\cT$-long  $\bar \varepsilon$-sparse tail  $R_\infty$. 
Let $\mathfrak{P}_n$ be the sequence of initial patterns associated to the tail $R_\infty$ given
by Remark~\ref{r.initial}.
Let
$M (R_\infty)$ be the set of marked points of the components of $\fP_n$,  %
$$
M(R_\infty) \eqdef 
\bigcup_{0}^{\infty} M (\fP_n).
 $$
 
\begin{lemm}\label{l.infinitecontrol}  
For every plaque $D_0\in\fF$ 
there is a sequence $(D_a)_{a\in M(R_\infty)}$ of plaques of $\fF$
such that for every $n$ the subfamily 
$(D_a)_{a\in  M( \mathfrak{P}_n)}$
is $(\cJ_n, \fP_n)$-controlled.
\end{lemm}
\begin{proof}
Apply first Theorem~\ref{t.p.flipfloppattern} to construct the  family associated to 
the pattern $\mathfrak{P}_0$.
Thereafter 
 inductively apply Addendum~\ref{p.extension}
to construct 
the family of sets associated to $\mathfrak{P}_{n+1}$ extending the family constructed 
for $\mathfrak{P}_n$.
\end{proof}

By the  expansion property \eqref{i.defff2}
in Definition~\ref{d.flipfloptail} of  a flip-flop family with sojourns  
we have that
\begin{equation}\label{e.justapoint}
\bigcap_{a\in M (R_\infty)} f^{-a}(D_a)=\{x_D\}\subset D_0.
\end{equation}
By construction, the point $x_D$
is controlled at any scale with long sparse tail $R_\infty$ with respect to $\varphi$ and $Y$,
proving Theorem~\ref{t.flipfloptail}.  \qed

 \subsection{Proof of Corollary~\ref{c.flipfloptail}}
 Let   $\mu$ be  any weak$\ast$-accumulation  point of
 the family of empirical measures
$(\mu_n(x_D))_n$.
By Theorem~\ref{t.accumulation},  for $\mu$-almost every point $x$ its  Birkhoff average 
$\varphi_\infty(x)$ is zero 
its orbit   is dense in $X$. This immediately implies that almost every component $\nu$
of the ergodic decomposition of $\mu$ has full support and $\int \varphi \,d\nu=0$.
 \qed

  \subsection{Proof Proposition~\ref{p.r.hyperboliclike}}
Given $t\in (\alpha, \beta)$ consider $\alpha< \alpha_t <t< \beta_t <\beta$.  Consider
 small cylinders $C(\alpha_t)$ and $C(\beta_t)$ where the map $\varphi$ is less than $\alpha_t$ and bigger than $\beta_t$, respectively.  Consider now unstable subsets of these cylinders (i.e., the intersection of 
 the cylinders with unstable sets). For sufficiently large
$m$ we have that these sets are a flip-flop family relative to $f^m$.  Now it is enough to apply either 
the criterion in \cite{BoBoDi2} (to get item (a)) or to apply Corollary~\ref{c.flipfloptail} (to get item (b)).

\section{Proof of Theorem~\ref{t.homoclinic}:
Flip-flop families and homoclinic relations}
\label{s.flip-flophomoclinic}

The goal of this section is to prove
 Theorem~\ref{t.homoclinic}.
Consider $f\in \Diff^1(M)$ and a pair of hyperbolic periodic points $p$ and $q$ 
of $f$ that are homoclinically related and a 
continuous function  $\varphi\colon M\to \RR$ such that  
$\int \varphi\, d\mu_{\cO(p)}  <t<
\int \varphi \, d\mu_{\cO(q)} $
(recall that $\mu_{\cO(p)}$ and $\mu_{\cO(q)}$ are the 
unique $f$-invariant probability measures supported on the orbits $p$ and $q$, respectively). 
 For notational simplicity, let us 
assume that the periodic points $p$ and $q$ are fixed points. In this case the assumption
above just means $\varphi(p)<t<\varphi(q)$. After replacing $\varphi$ by the map $\varphi_t=\varphi-t$, 
to prove the theorem it is enough to get an ergodic measure $\mu_t$ whose support is 
$H(p,f)$ such that $\int \varphi_t\, d\mu_t=0$. Thus, in what follows, we can assume that $t=0$
and hence $\varphi(p)<0<\varphi(q)$.

\subsection{Flip-flop families obtained from homoclinic relations}
\label{ss.flip-flopfromhomoclinic}
To prove the theorem we construct a flip-flop family associated to $\varphi$ and $f^n$ for some $n>0$.
We begin by recalling some constructions from \cite{BoBoDi2}.

\subsubsection{The space of discs $\cD^i(M)$}
\label{sss.spaceofdiscs}
Recall that $M$ is a closed and compact Riemannian manifold, let $\dim (M)=d$.
For each $i\in\{1,\dots,d-1\}$ denote by $\cD^i(M)$  the set of 
$i$-dimensional (closed) discs $C^1$-embedded in $M$.
In the space $\cD^i(M)$ the standard  $C^1$-topology is defined as follows,
given a disc $D\in \cD^i(M)$ its neighbourhoods  are of the form 
$\{f(D) \colon  f \in \cW\}$, 
where $\cW$ is a neighbourhood of the identity map in $\Diff^1(M)$.
In  \cite{BoBoDi2}  it is introduced the following  metric $\fd$ on the space $\cD^i(M)$, 
$$ 
(D_1,D_2) \mapsto
\fd(D_1,D_2)\eqdef  d_\mathrm{Haus}(TD_1,TD_2)+d_\mathrm{Haus}(T\partial D_1,T\partial D_2),
$$
where
$D_1,D_2\in \cD^i(M)$,  
the tangent bundles $TD_i$ and $T\partial D_i$ are considered as compact subsets of the corresponding 
Grassmannian bundles,  and $d_\mathrm{Haus}$ denotes the corresponding 
Hausdorff distances. The distance $\fd$ behaves nicely for the composition of diffeomorphisms:
if $D$ and $D^\prime$ are close the same holds for $f(D)$ and $f(D^\prime)$, see
\cite[Proposition 3.1]{BoBoDi2}.

Given a family of discs $\fD\subset \cD^i(M)$ and $\eta>0$,  we denote by $\cV^\fd_\eta(\fD)$ the open
$\eta$-neighbourhood of $\fD$ with respect to the distance $\fd$, 
$$
\cV^\fd_\eta(\fD)\eqdef
\{D\in \cD^i(M)\colon \fd(D,\fD)<\eta\}.
$$

\subsubsection{Proof of Theorem~\ref{t.homoclinic}}
Since $\varphi(p)<0<\varphi(q)$, there are
small local unstable manifolds $W^\mru_{loc}(p,f)$ and $W^\mru_{loc}(q,f)$ 
of $p$ and $q$ such that $\varphi$ is strictly negative in $W^\mru_{loc}(p,f)$ and strictly positive in $W^\mru_{loc}(q,f)$.
Similarly for the local stable manifolds $W^\mrs_{loc}(p,f)$ and $W^\mrs_{loc}(q,f)$ of $p$ and $q$.

Since $p$ and $q$ are homoclinically related there $\ell_0\ge 0$
and small discs $\De^\mrs_p \subset W^\mrs(p)$ and $\De^\mrs_q\subset 
W^\mrs(q,f)$ such that 
the intersections
$\De^\mrs_p\cap f^{\ell_0}(W^\mru_{loc}(q,f))$ and $\De^\mrs_q\cap f^{\ell_0}(W^\mru_{loc}(p,f))$ 
are both transverse and consist of just a point.

For  $\varrho>0$ consider  the
$\varrho$-neighbourhoods 
 $\cV^{\fd}_\varrho(p)\eqdef \cV^{\fd}_\varrho(W^\mru_{loc}(p,f))$ and $\cV^{\fd}_\varrho(q)\eqdef
 \cV^{\fd}_\varrho(W^\mru_{loc}(q,f))$
of the local unstable manifolds of $p$ and $q$ for the distance $\fd$. 
For $\varrho$ small enough, every disc in  
$\cV^{\fd}_\varrho(p)$ intersects transversely 
$\De^\mrs_p$.  We also have that $\varphi$ is uniformly negative (say less than 
$-\alpha<0$) in every disc in this neighbourhood.  
Finally, the derivative of $Df$ is uniformly expanding in restriction to this family of discs. 
There are similar assertions for the discs in $\cV^{\fd}_\varrho(q)$:
every disc of this neighbourhood meets transversely
$\De^\mrs_q$, $\varphi$ is larger than $\alpha>0$ in the discs, and $Df$ is a uniform expansion. 

\begin{rema}\label{r.contraction} 
If 
$\varrho>0$ is small enough then
 there is $\ell_1$  such 
 that  for every $\ell>\ell_1$ the image $f^{\ell}(D)$ of any disc  $D\in  \cV^\fd_\varrho (p)$  contains discs 
 in $ \cV^\fd_\varrho(p)$  and in $\cV^\fd_\varrho(q)$.  
 The same holds  (with the same constant) for discs in $\cV^\fd_\varrho(q)$.  
 This is a well known fact and is the ground of the proof of the existence of 
  unstable manifolds using a graph transformation. In what follows we assume that $\varrho$ satisfies 
  this property. 
  \end{rema}

We consider the following family $\fF=\fF^+\bigsqcup \fF^-$ of discs:
\begin{itemize}
 \item $\fF^-$ is the family of discs in $\cV^\fd_\varrho(p)$ contained in $W^\mru(p,f)\cup W^\mru(q,f)$;
 \item $\fF^+$ is the family of discs in $\cV^\fd_\varrho(q)$ contained in $W^\mru(p,f)\cup W^\mru(q,f)$.
\end{itemize}
Note that as $q$ are homoclinically related
these two families are both infinite.

\begin{prop}\label{p.flipflophomoclinic}
There is $n$ such that
the family $\fF$  is a flip-flop family associated to $\varphi$  and $f^n$ and
 has sojourns (for $f$) along the homoclinic class $H(p,f)$. 
\end{prop}

We postpone the proof of Proposition~\ref{p.flipflophomoclinic}
and prove the theorem.

\begin{proof}[Proof of Theorem~\ref{t.homoclinic}] 
Consider the flip-flop family $\fF$ with sojourns along $H(p,f)$
given by Proposition~\ref{p.flipflophomoclinic}. 
Exactly as in the proof of Theorem~\ref{t.flipfloptail} in Section~\ref{s.flipfloptail} we use 
Theorem~\ref{t.p.flipfloppattern}, 
Addendum~\ref{p.extension}, and Lemma~\ref{l.tailexistence} to construct  
a scale $\cT$, a tail $R_\infty$, 
a sequence of increasing patterns $\fP_n$,
 and a family of discs $D_a$, $a\in M(R_{\infty})$ such 
that  the restriction of this 
family to the marked sets $M(\fP_n)$ is controlled by the pattern $\fP_n$ for every $n$ (exactly as in Lemma~\ref{l.infinitecontrol}).
The expansion property in the flip-flop family implies that (recall equation~\eqref{e.justapoint})
$$
\bigcap_{a\in M(R_\infty)} f^{-a}(D_a)=
x_\infty.
$$ 

\begin{clai}$x_\infty \in H(p,f)$. 
\end{clai}
\begin{proof} Every disc $D_a$ belongs to $\fF$,  hence it is contained in 
$W^\mru (p,f)\cup W^\mru (q,f)$ and intersects  transversely $W^\mrs(p,f)\cup W^\mrs(q,f)$. 
Thus $D_a$ 
contains a point of $H(p,f)$. The $f$-invariance of $H(p,f)$ implies that the same holds for $f^{-a}(D_a)$ 
The  compactness of $H(p,f)$ implies that $x_\infty\in H(p,f)$. 
\end{proof}

By construction, the point $x_\infty$ is controlled at any scale  with a long sparse tail 
for $\varphi$ and $f$ (the ambient space here is
$H(p,f)$). The theorem now follows from
Theorem~\ref{t.accumulation}. 
\end{proof}

\subsection{Proof of Proposition~\ref{p.flipflophomoclinic}}
\label{ss.postponed}
We split the proof of the proposition into two parts:

\subsubsection{$\fF=\fF^+\cup \fF^-$ is a  flip-flop family}\label{ss.flipflop}
By construction, the map $\varphi$ is less than $-\alpha<0$ in the discs
of  $\cV^{\fd}_\varrho ( p)$ and bigger than $\alpha>0$ in the discs
of $\cV^{\fd}_\varrho(q)$. The definition of $\fF^\pm$ implies 
\eqref{i.flipflop1} in Definition~\ref{d.flipflop}.

Recall the choice of $\ell_0$ above and that,
by construction, the image $f^{\ell_0}(D)$ of  any disc $D\in \fF$ intersects transversely the compact parts $\De^\mrs_p$ of $W^\mrs(p,f)$ and 
$\De^\mrs_q$ of  $W^\mrs(q,f)$.
 Thus the $\lambda$-lemma (inclination lemma) and the invariance of $W^\mru(p,f)\cup W^\mru(q,f)$ imply 
the existence of $n_0>0$ such that for every $n>n_0$ and every disc $D\in\fF$
the set $f^n(D)$ contains a disc $D^+\in\fF^+$ and a disc $D^-\in \fF^-$. 
This proves item \eqref{i.flipflop2} in Definition~\ref{d.flipflop}.

It remains to get the expansion property in item \eqref{i.flipflop33} of Definition~\ref{d.flipflop}. We need to get $n$ such that
for every $D\in \fF$ the disc $f^n(D)$ contains a disc $D'$ such 
that $f^n\colon f^{-n}(D')\to D'$ is a uniform expansion. For that it is enough to take sufficiently large $n$ (independent of $D$). 
To see why this is so recall first that $f^{n_0}(D)$ contains a disc $D_{n_0}\in \fF^+$. 
Now Remark~\ref{r.contraction} 
 provides a sequence of discs $D_{n_0+i\ell_0}$ in $\fF^\pm$ such 
 that $D_{n_0+(i+1)\ell_0}\subset f^{\ell_0}(D_{n_0+i})$ and 
 $f^{\ell_0}\colon f^{-\ell_0}(D_{n_0+(i+1)\ell_0})\to D_{n_0+(i+1)\ell_0}$  is a uniform expansion. 
 This implies that for $i$ large enough 
(independent of $D$) we get the 
announced expansion  for $f^{n_0+i\ell_0}\colon f^{-(n_0+i\ell_0)}(D_{n_0+i\ell_0})\to D_{n_0+i\ell_0}$,
just note that the $i\,\ell_0$ additional iterates in the ``expanding part" compensate any contraction introduced by
the first $n_0$ iterates.

\subsubsection{The family $\fF$ sojourns along the homoclinic class $H(p,f)$}\label{ss.sojournsalong}
Consider any $\delta>0$. We need prove that there is $N>0$ such  that 
every disc $D\in\fF$ contains a pair of discs 
$\widehat D^+$, $\widehat D^-$ such  that
for every $x\in \widehat D^\pm$ the segment of orbit $\{x,\dots, f^N(x)\}$ is $\delta$-dense
in $H(p,f)$ (item \eqref{i.defff0}),
$f^N(\widehat D^+)\in\fF^+$
and $f^N(\widehat D^-)\in\fF^-$ (item \eqref{i.defff1}), and
$f^N\colon \widehat D^\pm\to f^N(\widehat D^\pm)$ is expanding
(item \eqref{i.defff2}).

We need the following property of  $H(p,f)$ 
that is a direct consequence of the density
of transverse homoclinic intersection points of $W^\mru(p,f)\cap W^\mrs(p,f)$ in $H(p,f)$
and the existence of (hyperbolic) horseshoes associated to these points. 

\begin{rema}\label{r.existenceofperiodicorbits}
For every $\epsilon>0$ there is a hyperbolic periodic point $r_\epsilon\in H(p,f)$
 that is homoclinically related to $p$ and $q$ whose orbit  is $\epsilon/2$-dense in $H(p,f)=
 H(q,f)$.
\end{rema}

To prove item \eqref{i.defff0} consider the point $r=r_{\frac{\delta}{2}}\in H(p,f)$ given by 
Remark~\ref{r.existenceofperiodicorbits}.
As the points $r, p$, and $q$ are pairwise homoclinically related,  the stable manifold of
the orbit of $r$,  $W^\mrs(\cO(r),f)$ 
 accumulates the ones of $p$ and $q$.
 Hence there are compact discs $\De^\mrs_{r,p}, \De^\mrs_{r,q}\subset W^\mrs(\cO(r),f)$ such 
 that any disc in 
 $\cV^{\fd}_\varrho(p)\cup \cV^{\fd}_\varrho (q)$ 
 meets transversely $\De^\mrs_{r,p}\cup \De^\mrs_{r,q}$.

Let $\pi$ be the period of $r$. As in Remark~\ref{r.contraction}, for each $i=0,\dots, \pi-1$,
we fix a  small local unstable manifold $W^\mru_{loc}(f^i(r),f)$ and a small $C^1$-neighbourhood 
$\cV_\eta^\fd (f^i(r))\eqdef  \cV_\eta^\fd (W^\mru_{loc}(f^i(r),f))$ 
 such that the image $f(D)$ of
 any disc $D\in \cV^\fd_\eta(f^i(r))$ contains a disc in 
$\cV_\eta^\fd (f^{i+1}(r))$ (for $\pi-1$ we take ``$\pi=0$").   

Take now $D=D_0$ any disc in $\cV^\fd_\eta(r)$, let $D_1$ be a  sub-disc of
$f(D_0)$ in $\cV^\fd_\eta (f(r))$, and inductively define $D_{i+1}$ as a disc
in $\cV^\fd_\eta (f^{i+1}(r))$ contained in $f(D_i)$. 
 Assuming thta the local unstable manifolds and their neighbourhoods are small enough we have that
 every point in a  disc of $\cV_\eta^\fd (f^i(r))$, $i=0,\dots, \pi-1$, is at
 distance less that $\frac{\delta}2$ from the orbit of $r$. 
 Since the orbit of $r$ is $\frac{\delta}2$-dense in $H(p,f)$ for every $x\in f^{-\pi }(D_\pi)\subset D$,
 we have that the segment of  orbit $\{ x,\dots , f^{\pi}(x)\}$ 
 is $\delta$-dense in $H(p,f)$. 

%
%

Consider now a disc $D\in\fF$. By construction,  this disc intersects transversely $W^\mrs(\cO(r),f)$ in some point of $\De^\mrs_{r,p}\cup \De^\mrs_{r,q}$. By 
the $\lambda$-lemma there  is $j_0$ (independent of $D$) such 
that $f^{j_0}(D)$ contains a  disc $D_{0}$ in $\cV_\eta^\fd(r)$. 
The argument above  provides a 
sequence of discs $D_{i}\in \cV_\eta^\fd (f^{i}(r))$, $j\in\{0,\dots,\pi-1\}$, 
with $D_{i+1}\subset f(D_{i})$ and such that 
for every $x\in f^{-\pi}(D_{\pi})\subset D_0\subset f^{j_0}(D)$ 
its  orbit segment $\{x,\dots, f^{\pi}(x)\}$ 
is $\delta$-dense in $H(p,f)$.
A new application of  the $\lambda$-lemma provides a uniform $j_1>0$ such that 
$f^{j_1}(D_\pi)$ contains  discs $\widetilde D^\pm \in \fF^\pm$ (recall that the initial $D\in \fF$ and therefore it
is contained
in $W^\mru(p,f) \cup W^\mru(q,f)$).

Now it is enough to take 
$$
N\eqdef j_0+\pi+j_1
\quad
\mbox{and}
\quad
\widehat D^\pm\eqdef f^{-j_0-\pi-j_1}(\widetilde D^\pm) \subset D.
$$
By construction the  orbit segment  
$\{y,\dots, f^{N}(y)\}$ of any point 
$y\in \widehat D^\pm$
is $\delta$-dense in $H(p,f)$,  proving item
\eqref{i.defff0} in Definition~\ref{d.flipfloptail}. By construction, $f^N(\widehat D^\pm)=\widetilde D^\pm \in \fF^\pm$ proving  item 
\eqref{i.defff1} in Definition~\ref{d.flipfloptail}

Note that the discs $\widehat D^\pm\subset D$ satisfy
the density in $H(p,f)$  and return to $\fF^\pm$  properties, 
however they can fail to satisfy the expansion property in \eqref{i.defff2} 
in Definition~\ref{d.flipfloptail}.
To get  additionally such an expansion 
one considers further iterates of the disc in a ``expanding" region nearby $p$ or $q$. The expansion is obtained
 using Remark~\ref{r.contraction} 
and arguing exactly as in Section~\ref{ss.flipflop}.
The proof of Proposition~\ref{p.flipflophomoclinic} is now complete. \qed

\subsection{Proof of Corollary~\ref{c.function}}
\label{ss.proofofcorollarycfunction}
By hypothesis, the saddles  $p_f$ and $q_f$  have  different $\mru$-indices (say $i$ and $j$, $i<j$) that depend continuously on $f$ and whose chain recurrence classes coincide for every  diffeomorphism $f$ in a $C^1$-open
set $\cU$. 
As in the proof of Theorem~\ref{t.homoclinic}, let us assume that $t=0$ and hence 
the Birkhoff average of 
 $\varphi$ is negative  in $\cO(p_f)$  
 and positive  in $\cO(q_f)$.  

According to \cite{ABCDW}, 
up to restrict to a $C^1$-open and dense subset of $\cU$, 
we can assume that for every $k\in [i,j]$
every diffeomorphism $f\in \cU$ has a periodic point $r_f$ of $\mru$-index $k$ 
that is $C^1$-robustly in $C(p_f,f)$.
Therefore, after replacing  $p_f$ and $q_f$ by other periodic points, we can assume that 
the $\mru$-indices of $p_f$ and $q_f$ are consecutive.

Following Propositions 3.7 and 3.10 in \cite{ABCDW}, 
an arbitrarily small  $C^1$-per\-tur\-ba\-tion of $f$ gives a diffeomorphism $h$ with a periodic point $r_h$ 
having a (unique) center eigenvalue equal to $1$ that is robustly
in $C(p_h,h)$. This means that this (non-hyperbolic) periodic point $r_h$ admits a continuation $r_g\in C(p_g,g)=C(q_g,g)$ for some $g$ arbitrarily close to $f$. 

Consider the average of $\varphi$ along the orbit of 
$r_h$ and assume first that  it is different from zero, for example negative. 
Then, after an arbitrarily small perturbation, we can transform $r_h$ in a hyperbolic point $r_g$ of $g$ of the same 
index as $q_g$ and homoclinically related to $q_g$ (for this last step we  use the version of Hayashi's connecting lemma~\cite{Ha}
for chain recurrence classes in \cite{BC}\footnote{This result guarantees that 
given two saddles in
  the same chain recurrence there is an arbitrarily small  $C^1$-perturbation of the diffeomorphism 
 that gives an intersection between 
the invariant manifolds of these saddles. If the saddles belong $C^1$-robustly to the same class then one can repeat the 
previous argument, interchanging the roles of the saddles, to get that the invariant manifolds of these saddles meet cyclically. Finally, if the saddles have the same index one can turn these intersections into transverse ones, thus the two saddles are homoclinically related and hence they are $C^1$-robustly  in the same homoclinic class.}).  
The diffeomorphism $g$ belongs to $\cU$, the saddles $r_g$ and $q_g$ are homoclinically related, and the averages of $\varphi$ in these orbits have different signals. The corollary now follows from
Theorem~\ref{t.homoclinic}.

In the case when the average of $\varphi$ throughout the orbit of $r_g$ is zero one needs a slight modification of the previous argument. 
 Let us sketch this construction, arguing as above,  we
can assume that,  after an arbitrarily small perturbation, the point $r_g$ is hyperbolic of the same index as $p_g$
(with center derivative arbitrarily close to one) and that
 $r_g$ and $p_g$ are homoclinically related.
Using the homoclinic relation between $r_g$ and $p_g$ we get a point $\bar r_g$ with some center eigenvalue arbitrarily close to one and with negative average for $\varphi$. Next, arguing as above and after a small perturbation, we change the index of the point
$\bar r_g$ and  
generate transverse cyclic intersections between
$\bar r_g$ and $q_g$ (i.e., we put the saddle $\bar r_g$  in the homoclinic class of $q_g$). We are now in the previous case and  prove the corollary using the saddles $\bar r_g$ and $q_g$.
\qed

\section{Flip-flop families in partially hyperbolic dynamics}
\label{s.flipflopph}
In  this section we prove Theorems~\ref{t.cycle} and \ref{t.ctail}.
For that we borrow and adapt some constructions in \cite{BoBoDi2}.
In Section~\ref{ss.dynamicalblenders} we recall the definition of a dynamical blender and its main properties. Section~\ref{ss.ffconf} is dedicated to the study of flip-flop configurations.
In Section~\ref{ss.flipflopfamilieswithtails} we see how  flip-flop configurations yield 
flip-flop families with sojourns.
In Section~\ref{ss.controlofaveragesflipflop} we analyse the 
control of averages in flip-flop configurations.
Finally, in Section~\ref{ss.proofoftheoremtcycleandother} 
we conclude the proofs of Theorems~\ref{t.cycle} and \ref{t.ctail}.

%


%


%

\subsection{Dynamical blenders}\label{ss.dynamicalblenders}
The definition
of a \emph{dynamical blender}  in \cite{BoBoDi2} 
involves three main ingredients: 
the  distance on the space of $C^1$-discs  (Section~\ref{sss.spaceofdiscs}),
strictly invariant families of  discs (Section~\ref{ss.invariantfamilies}), and
invariant cone fields (Section~\ref{ss.invariantconefields}). We now describe succinctly these ingredients.

%
%

%

\subsubsection{Strictly invariant families of  discs}
\label{ss.invariantfamilies}
Recall the notation  $\cD^i(M)$ 
for  the set of 
$i$-dimensional (closed) discs $C^1$-embedded in $M$ and the definitions 
of the distance $\fd$ and 
the open neighbourhood  $\cV^\fd_\eta(\fD)$ 
of a family of discs $\fD$ with respect to $\fd$
in  Section~\ref{sss.spaceofdiscs}.

\begin{defi}[Strictly $f$-invariant families of discs]
\label{d.strict} 
Let $f\in \Diff^1(M)$. A family of discs $\fD\subset \cD^i(M)$  is \emph{strictly $f$-invariant} if there is $\varepsilon>0$ such that for every disc
$D_0\in\cV^\fd_\varepsilon(\fD)$ there is a disc $D_1\in \fD$ with 
$D_1\subset f(D_0)$. 
\end{defi}

The existence of a strictly invariant family of discs is a $C^1$-robust property:
If the family $\fD$ is strictly $f$-invariant  then there are $\mu, \eta>0$ such that
the family $\fD_\mu= \cV^\fd_\mu(\fD)$
is strictly $g$-invariant 
for every  $g\in \Diff^1(M)$ that is $\eta$-$C^1$-close to $f$,
 see \cite[Lemma 3.8]{BoBoDi2}.

\subsubsection{Invariant cone fields}
\label{ss.invariantconefields}
Given a vector space of finite dimension $E$, we say that
a subset $C$ of $E$  is a \emph{cone of index $i$} 
if there are a splitting   $E = E_1 \oplus E_2$ with $\dim (E_1) = i$
and a norm $\| \mathord{\cdot} \|$ defined on $E$ such that
$$
C = \{v=v_1 + v_2  \colon v_i \in E_i, \ \|v_2\| \le \|v_1\|\}.
$$
A cone $C'$ is \emph{strictly contained} in the cone $C$  above if there is  $\alpha>1$
such that 
$$
C' \subset C_\alpha=\{v_1 + v_2 \colon 
v_i \in E_i, \ \|v_2\| \le \alpha^{-1} \|v_1\|\} \subset C.
$$

A \emph{cone field of index $i$} defined on a subset $V$ of a compact 
manifold $M$ 
is a continuous map $x\mapsto \cC(x)\subset T_x M$  that associates to each point 
$x\in V$ a cone $\cC(x)$ of index $i$ . We denote this cone field by $\cC=\{\cC(x)\}_{x\in V}$.

Given a diffeomorphism $f\in \Diff^1(M)$ and a cone field 
$\cC=\{\cC(x)\}_{x\in V}$ we say that $\cC$  is \emph{strictly $Df$-invariant} 
if $Df(x)(\cC(x))$ is strictly contained in $\cC(f(x))$
for every $x \in V \cap f^{-1}(V)$.

The following result is a standard lemma about persistence of invariant cone fields
(see for instance \cite[Lemma 3.9]{BoBoDi2}).

\begin{lemm}\label{l.cones} 
Let  $f\in \Diff^1(M)$, $V$ a compact subset of $M$, and $\cC$ a strictly $Df$-invariant cone field 
defined on $V$. Then there is a $C^1$-neighbourhood
$\cU$ of $f$ such that $\cC$ is strictly $Dg$-invariant for every $g\in \cU$. 
\end{lemm}

\subsubsection{Dynamical blenders}
\label{ss.manyblenders}

We are now ready  to define a {\emph{dynamical blender}} and recall its main properties.

\begin{defi}[Dynamical blender, \cite{BoBoDi2}] 
\label{d.dynamicalblender}
Let $f\in \Diff^1(M)$. A compact $f$-invariant set $\Ga \subset M$  
is a \emph{dynamical  $\mathrm{cu}$-blender}
\emph{of $\mathrm{uu}$-index $i$} if the following properties hold:
\begin{enumerate}
 \item  
 there is an open neighbourhood $V$ of $\Gamma$ such that 
 $$
 \Gamma=\bigcap_{n\in\ZZ} f^n(\overline V);
 $$
  \item 
 the set $\Gamma$ is transitive;
 \item 
 the set $\Gamma$ is  (uniformly) hyperbolic with $\mathrm{u}$-index strictly larger than $i$;
 \item 
 there is a strictly $Df$-invariant cone field $\cC^{\mathrm{uu}}$ of index $i$ defined on $\overline V$;
 and
 \item 
 there are a strictly $f$-invariant family of discs  $\fD\subset \cD^i(M)$ and  $\varepsilon>0$ such  
 that every disc in
 $\cV^\fd_\varepsilon(\fD)$ is contained in $V$ and tangent to $\cC^{\mathrm{uu}}$.
\end{enumerate}
We say that $V$ is the \emph{domain of the blender}, $\cC^{\mathrm{uu}}$ is its \emph{strong unstable cone field}, 
and $\fD$ is its 
\emph{strictly invariant family of discs}.
To emphasise the role of these objects  we write
$(\Gamma, V,\cC^{\mathrm{uu}}, \fD)$. 
\end{defi}

\begin{rema}\label{r.blendersph}
Let $\Gamma$ be a 
hyperbolic set of  $\mathrm{u}$-index $j$ that is also
a $\mathrm{cu}$-blender of $\mathrm{uu}$-index $i$. 
By definition, the set 
$\Gamma$ has a partially hyperbolic splitting  (recall \eqref{e.ph}) of the form 
$$
T_\Gamma M= E^\mathrm{uu} \oplus E^\mathrm{cu} \oplus E^\mathrm{s},
$$
where $\dim(E^\mathrm{uu})=i$,
$\dim (E^\mathrm{cu})=j-i\ge 1$,
and $E^{\mathrm{u}}=E^\mathrm{uu}\oplus E^\mathrm{cu}$. Here
$E^\mathrm{s}$  and $E^\mathrm{u}$ are  the stable and unstable bundles
of $\Gamma$.  We also define the bundle
$E^{\mathrm{cs}}\eqdef E^\mathrm{cu}\oplus E^\mathrm{ss}$.
\end{rema}

Next lemma claims that blenders have well defined continuations.

\begin{lemm}
[Lemma 3.8 and Scholium 3.14 in \cite{BoBoDi2}]
\label{l.robust} 
Let $(\Gamma,V,\cC^{\mathrm{uu}},\fD)$ be a dynamical 
blender of $f\in \Diff^1(M)$. 
Then there are a $C^1$-neighbourhood $\cU$ of $f$ and $\varepsilon>0$  such  that 
for every diffeomorphism $g\in\cU$ the $4$-tuple $(\Gamma_g,V,\cC^{\mathrm{uu}},\cV^\fd_\epsilon(\fD))$ is a dynamical blender, where $\Gamma_g$ is the hyperbolic continuation 
of $\Gamma$ for $g$. 

Moreover, 
every disc 
$D\in \cV^\fd_\varepsilon(\fD)$  meets the local stable manifold of $\Gamma_g$
defined by 
$$
W^\mrs_{loc}(\Ga_g)\eqdef \{x\in V\colon f^i(x)\in V\, \mbox{for every $i\ge 0$}\}.
$$
\end{lemm}
%
%


\subsection{Flip-flop configurations and partial hyperbolicity}
\label{ss.ffconf}
We now recall the definition of a flip-flop configuration and borrow some results from  \cite{BoBoDi2}.

\begin{defi}[Flip-flop configuration]
\label{d.flip} 
 Consider  $f\in \Diff^1(M)$ having
a dynamical $\mathrm{cu}$-blen\-der 
$(\Ga, V, \cC^\mathrm{uu}, \fD)$
of $\mathrm{uu}$-index $i$  
%
%
 and a hyperbolic periodic point $q$ of $\mathrm{u}$-index $i$.
We say that $(\Ga, V, \cC^\mathrm{uu}, \fD)$   and $q$ form a \emph{flip-flop configuration} if
there are:
\begin{itemize}
\item a disc $\Delta^\mathrm{u}$  contained in the unstable manifold $W^\mathrm{u}(q,f)$ and
\item a compact submanifold with boundary $\De^\mathrm{s} \subset V \cap W^\mathrm{s}(q,f)$
\end{itemize}
such that:
\begin{enumerate}
\item\label{i.FFC0} 
The disc $\Delta^\mathrm{u}$ belongs to the interior of the family $\fD$.
\item\label{i.FFC1}
$f^{-n}(\Delta^\mathrm{u}) \cap \overline V = \emptyset$ for all $n>0$.
\item\label{i.FFC2}
There is $N>0$ such that $f^{n}(\Delta^\mathrm{s}) \cap \overline V = \emptyset$ for every 
$n>N$. Moreover, if 
$x\in \De^\mathrm{s}$ and $j>0$ are such  that 
$f^j(x)\notin V$ then 
$f^i(x)\notin \overline V$ for every $i\ge j$.
\item\label{i.FFC3} 
$T_y W^\mathrm{s}(q,f) \cap \cC^\mathrm{uu}(y) = \{0\}$ for every $y \in \Delta^\mathrm{s}$.
\item\label{i.FFC4} 
There are a compact set $K$ in the relative interior of $\Delta^\mathrm{s}$ and  $\epsilon>0$ such that 
for every $D\in\fD$ there exists $x$ such that $K\cap D=\{x\}$ and $d(x,\partial D)>\epsilon$.
\end{enumerate}

The sets $\De^{\mathrm{u}}$ and $\De^{\mathrm{s}}$ are called the \emph{ unstable and stable connecting sets} 
of the flip-flop configuration, respectively. 
\end{defi}

 \cite[Proposition 4.2]{BoBoDi2} asserts that flip-flop configuration are $C^1$-robust. 
Next lemma claims that flip-flop configurations yield partially hyperbolic dynamics.
Recall Remark~\ref{r.blendersph} and the definition of the center unstable bundle $E^{\mathrm{cu}}$
of  a blender. 

\begin{lemm}[Lemma~4.6 in \cite{BoBoDi2}]\label{l.Laranjeiras}  
Consider $f\in \Diff^1(M)$ having a hyperbolic  periodic point $q$  
and a dynamical blender  
$(\Ga,V,\cC^\mathrm{uu},\fD)$ in a flip-flop configuration with connecting sets 
$\De^\mathrm{u}\subset W^\mathrm{u}(q,f)$  and 
$\De^\mathrm{s}\subset W^\mathrm{s}(q,f)$.
Consider the closed set
$$
 \Delta\eqdef
\mathcal{O}(q) \, \cup \,
\overline{V}
\, \cup \,  
\bigcup_{k\ge 0} f^k(\Delta^\mathrm{s}) 
\, \cup \,
\bigcup_{k\le 0} f^k(\Delta^\mathrm{u}).
$$
Then  there is a compact neighbourhood $U$ of  $\Delta$,
called a \emph{partially hyperbolic neighbourhood of the flip-flop configuration},
such that
the maximal invariant set $\Gamma(U)$ of $f$ in $U$
$$
\Gamma(U)\eqdef \bigcap_{i\in \ZZ} f^i(U)
$$
has a partially hyperbolic 
splitting
$$
T_{\Gamma(U)} M =
\widetilde E^\mathrm{uu} \oplus \widetilde E^\mathrm{cs},
$$
where  $\widetilde E^\mathrm{uu}$  is 
uniformly expanding and  $\widetilde E^\mathrm{uu}$  and  $\widetilde E^\mathrm{cu}$ 
 extend the bundles $E^\mathrm{uu}$ and
$E^\mathrm{cs}$, respectively,   defined over~$\Gamma$.

Moreover, there is a strictly $Df$-invariant cone field 
over $U$ that extends the cone field $\cC^\mathrm{uu}$  defined on $\overline V$
(we also denote this cone field  by $\cC^\mathrm{uu}$) 
 whose vectors are uniformly expanded by $Df$.
\end{lemm}

\subsection{Flip-flop families  with sojourns in homoclinic classes}
\label{ss.flipflopfamilieswithtails}
\cite[Proposition 4.9]{BoBoDi2} claims that  flip-flop configurations yield 
flip-flop families. These configurations are enough to construct measures with controlled averages. 
However they 
do not provide control of the support of the obtained measure. 
In this paper, we want to get measures with ``full support". Bearing this in mind we defined flip-flop families
with sojourns (Definition~\ref{d.flipfloptail}). These ``sojourns" guarantee ``density" of orbits in the ambient space.

%
%

\begin{theor}\label{t.p.flipfloptail}
Consider $f\in \Diff^1(M)$ with a hyperbolic periodic point $q$ and a 
dynamical blender $\Ga$ in a flip-flop configuration.
Let 
$\varphi\colon M\to\RR$ be a continuous function that is positive on the blender
$\Ga$  and negative on the orbit of $q$. 

Then there are $N\geq1$ and a flip-flop family $\mathfrak{F}$ 
 with respect to $\varphi_N$ and 
 $f^N$  which  sojourns along the homoclinic class 
 $H(q,f)$ (for $f$).
  
  Moreover, given any $\delta>0$ the flip-flop family $\fF$ can be chosen such that:
 \begin{itemize}
 \item
 every $D\in \fF$  is
 contained in a $\delta$-neighbourhood of $\Ga\cup \{\cO(q)\}$, 
 \item
every  $D\in \fF$  transversely intersects 
 $W^\mrs(q,f^N)$, and
 \item
there is  $D\in \fF$ contained in $W^\mru_{loc}(q,f^N)$. 
\end{itemize}
\end{theor}

To prove this theorem we need to recall 
the construction of flip-flop families in \cite{BoBoDi2}. As the families in \cite{BoBoDi2}
do not have sojourns we need to adapt this construction to our context bearing in mind this fact.

\subsubsection{Flip-flop families associated to flip-flop configurations}

We now borrow the following result from \cite{BoBoDi2} and recall some steps of its proof.

\begin{prop}[Proposition 4.9 in \cite{BoBoDi2}]
\label{p.flipflopconf}
Let $f\in \Diff^1(M)$  be a diffeomorphism with a hyperbolic periodic point $q$ and a 
dynamical blender $\Ga$ in a flip-flop configuration.
 Let $U$ be a partially 
 hyperbolic neighbourhood of this configuration
 and  $\varphi\colon U\to\RR$ a continuous function that  is positive on the blender and negative on the orbit of $q$. 
 
 Then there are $N\geq1$ and a flip-flop family $\mathfrak{F}=\mathfrak{F}^+\bigsqcup \mathfrak{F}^-$ with respect to  $\varphi_N$ and
 $f^N$. 
 
 Moreover, given any $\delta>0$ the flip-flop family can be chosen such that the 
 plaques in $\mathfrak{F}^+$  are contained in a $\delta$-neighbourhood of $\Ga$
 and the plaques in $\mathfrak{F}^-$ are contained in a $\delta$-neigbourhood of $q$. 
\end{prop}

Note that Theorem~\ref{t.p.flipfloptail} is just the proposition above with additional sojourns.
Observe also that  the map $\varphi_N$ is only defined on $\bigcap_0^{N-1} f^{-i}(U)$ and 
that
the plaques of
 $\mathfrak{F}$ are contained in that set.  

We now review the construction in \cite{BoBoDi2}. For simplicity let us suppose that $q$ is a fixed point.
The definition of  the family $\mathfrak{F}$ in Proposition~\ref{p.flipflopconf}
 involves a preliminary family  of discs $\mathfrak{D}_q$ satisfying the following properties
(see \cite[Lemma 4.11]{BoBoDi2}):
\begin{enumerate}
 \item[(p1)]
 \label{i.1} The family of discs $\mathfrak{D}_q$  form  a small $C^1$-neighbourhood 
 (in the metric $\fd$) of  the local unstable manifold $W^\mru_{loc}(q,f)$. This neighbourhood can be taken arbitrarily small. 
 \item[(p2)]
 \label{i.2} The sets of the family $\mathfrak{F}^-$ are contained in discs in $\mathfrak{D}_q$.
 \item[(p3)]
 \label{i.3} Each disc in $\mathfrak{D}_q$ contains a plaque of $\mathfrak{F}^-$.
 \item[(p4)]
 \label{i.4} The image by $f^N$ of any plaque of $\mathfrak{F}$ contains a disc of $\mathfrak{D}_q$. 
\end{enumerate}

We have the following direct consequences of the properties above:
\begin{enumerate}
 \item[(p5)] 
 As $\fD_q$ can be taken contained in an arbitrarily small neighbourhood of $W^\mru_{loc}(q,f)$,
 we can  assume that $W^\mrs_{loc}(q,f)$ meets transversely every disc in $\fD_q$.
 \item[(p6)] 
 As $W^\mru_{loc}(q,f)$ can be chosen arbitrarily small, we can assume that $f^N$ 
 expands uniformly the vectors tangent to 
 the discs in $\fD_q$ (see also Remark~\ref{r.contraction}).
 \item[(p7)] 
 As a consequence of items~(p2),(p3), and (p4), the image by $f^N$ of
 any disc in $\fD_q$ contains a disc in $\fD_q$. 
 \end{enumerate}

We say that the flip-flop family $\fF$ is \emph{prepared with and adapted family $\fD_q$} 
if $\fF$ and $\fD_q$ satisfy  properties (p1)--(p7) above.

\subsubsection{Proof of Theorem~\ref{t.p.flipfloptail}}
Since a flip-flop family yields a prepared flip-flop family, Theorem~\ref{t.p.flipfloptail} is a consequence of 
the following proposition.
 
\begin{prop}
\label{p.l.flipfloptail} 
Let $f\in \Diff^1(M)$  be a diffeomorphism with a hyperbolic periodic point $q$ and a 
dynamical blender $\Ga$ in a flip-flop configuration,
 $U$ be a partially 
 hyperbolic neighbourhood of this configuration, and 
 $\varphi\colon U\to\RR$ a continuous function that  is positive on the blender and negative on the orbit of $q$. 

Let $N\ge 1$ and $\mathfrak{F}$ be  a  prepared flip-flop family  with an adapted family of discs $\fD_q$  with respect to $\varphi_N$ and $f^N$.

Then the flip-flop family $\fF$ sojourns along the homoclinic class $H(q, f^N)$.
\end{prop}

\begin{proof}
We need the following lemma whose 
proof is similar to the one of Proposition~\ref{p.flipflophomoclinic}
and follows using the partially hyperbolicity in the set $U$.

\begin{lemm}
\label{l.preparar}
For every $\delta>0$ there is 
$L\in \NN$ such that every disc $D\in\fD_q$ contains a sub-disc $\widehat D$ such that
\begin{itemize}
\item
for every $x\in \widehat D$ the segment of orbit $\{x,\dots, f^L(x)\}$ is $\delta$-dense in 
$H(q,f^N)$,
\item
$f^{L}(\widehat D)$ contains a disc of $\fD_q$.
\item
for every $i\in\{0,\dots, L\}$ and  every pair of points $x,y\in \widehat D$  
 it holds 
 $$ 
 d(f^{L-i}(x),f^{L-i}(y))\leq\zeta \, \alpha^i\,  d(f^L(x) f^L(y)),
 $$
 for some constants $\zeta>0$ and $0<\alpha<1$ (independent of the points and the discs). 
\end{itemize}
\end{lemm}

\begin{proof}
Consider a hyperbolic periodic point $r_\delta\in H(q,f^N)$ 
homoclinically related to
$q$ and whose orbit is $\delta$-dense in $H(q,f^N)$ (recall Remark~\ref{r.existenceofperiodicorbits}). 
The $\lambda$-lemma implies that a compact part of $W^\mrs(r_\delta,f)$ intersects  transversely 
every disc in $\fD_q$.
Thus, again by the $\lambda$-lemma,  iterations of any disc  $D\in \fD_q$ accumulate to $W^\mru_{loc}(\cO(r_\delta),f)$.
 Again the $\lambda$-lemma and (p7) in the definition of  a prepared family imply that further iterations of $D$ 
contains a disc in $\fD_q$.  Since the number of iterates involved can be taken uniform, 
considering the corresponding pre-image 
one gets the disc $\widehat D$ satisfying the first two items of the lemma.
Finally, 
exactly as in the end of the proof of Proposition~\ref{p.flipflophomoclinic}
further iterations provides the uniform expansion property.
This ends the proof of  the lemma.
\end{proof}

To end the proof of the proposition recall that,
by condition (p4), the image of any plaque $D\in \fF$ contains a disc
in $\fD_q$. This provides the ``sojourns property" for $\fF$ (may be one needs to add some extra additional
``final" iterates for recovering the expansion).  
\end{proof}

\vskip .cm

\subsection{Control of averages in flip-flop configurations}
\label{ss.controlofaveragesflipflop}
As a first consequence of Theorem~\ref{t.p.flipfloptail} 
we get measures with controlled averages and full support in a homoclinic class. 

\begin{theor}\label{t.p.tail} 
Let $f\in \Diff^1(M)$ be a diffeomorphism and 
$\varphi\colon M\to \RR$ be a continuous map.  
Suppose that $f$ has a dynamical blender  $\Gamma$ 
and hyperbolic periodic point $q$  
that  are  in flip-flop configuration with respect to $\varphi$ and $f$.
Then there is an ergodic measure $\nu$ whose support is the whole homoclinic class $H(q,f)$ such that 
 $$
 \int \varphi d\, \nu=0.
 $$
\end{theor}

\begin{proof}Under the hypotheses of the theorem, 
Theorem~\ref{t.p.flipfloptail} provides a flip-flop family $\fF$ associated 
to $f^N$ and $\varphi_N$ 
which sojourns in $H(q,f^N)$. 
By Theorem~\ref{t.flipfloptail} every plaque $D\in \fF$ contains a point $x_D$ that 
is controlled at any scale
with a  long sparse tail with respect $\varphi_N$ and  $H(q, f^N)$. 
Note that $W^\mru_{loc}(q,f)$ contains a plaque $\De\in \fF$.
Let  $x_\De\in \Delta$ be the controlled point given by Theorem~\ref{t.flipfloptail}.

\begin{clai}
$x_\Delta\in H(q,f^N)$.
\end{clai}

\begin{proof} 
 By construction, there is a sequence of discs $D_k$ and numbers $n_k\to \infty$ such
 that $D_{k+1}\subset f^{n_k}(D_k)$ and $D_0=\Delta$. Thus $f^{-n_k}(D_k)\subset \Delta$
 is a decreasing nested sequence of compact sets that 
 satisfies
$$
x_\Delta=
\bigcap_{k\in \NN} f^{-n_k}(D_k).
$$
The fact that this intersection is just a point follows from the expansion property in the definition of a flip-flop family.
Since $D_k\subset W^\mru(q,f^N)$ and   intersects transversely  $W^\mrs(q,f)$ we have that
$D_k$ contains a point  $x_k\in H(q,f^N)$. 
Hence $y_k=f^{-n_k}(x_k)\in H(q,f^N)\cap f^{-n_K}(D_k)$. 
Thus $y_k\to x_\De$ and 
$x_\Delta\in H(q, f^N)$. 
\end{proof}

Take any measure $\mu$ that is an
accumulation point of the measures 
$$
\mu_n\eqdef \frac1n\sum_0^{n-1}\delta(f^{iN}(x_\Delta)).
$$ 
As the point  $x_\Delta$ is controlled at any scale  with  a long sparse tail
with respect to  $\varphi$ and $H(q,f^N)$, 
Theorem~\ref{t.accumulation} implies that  $\mu$-almost every point $y$ has a dense orbit in  $H(q,f^N)$ and the average of $\varphi_N$ along the $f^N$-orbit of $y$  is zero.

To conclude the proof of the proposition  consider the measure 
$$
\eta \eqdef 
\frac 1N\sum_0^{N-1}  f^i_*(\mu).
$$ 
Now the $f$-orbit of $\eta$-almost every point $y$ is dense in $H(q,f)$ and satisfies
$\varphi_\infty(y)=0$.
By construction,  any ergodic component 
$\nu$ of $\eta$ satisfies the conclusion in the theorem.  
\end{proof}

\begin{rema}
Let us compare Theorem~\ref{t.p.tail} with Corollary~\ref{c.function}. In both cases there is a
continuous map $\varphi$ with a ``positive" and  a ``negative region".
Corollary~\ref{c.function} is a perturbation result while Theorem~\ref{t.p.tail} does not involve perturbations.
 The corollary  provides a (locally)  open and dense subset of diffeomorphisms $f$ 
 having an ergodic measure
 $\mu_f$
 with $\int \varphi \, d\mu_f=0$. In the theorem the mere existence of the
  flip-flop configuration for  $f$ and $\varphi$  gives 
 an ergodic measure
 $\mu_f$ of $f$
 with $\int \varphi\,  d\mu_f=0$.  

In the corollary  the support of the ergodic measure is not completely determined
(either $H(p,f)$ or $H(q,f)$) while in
the theorem the support is $H(q,f)$ ($q$ is the saddle in the flip-flop configuration).

\end{rema}

\subsection{Ergodic non-hyperbolic measures with full support}
\label{ss.proofoftheoremtcycleandother}
In this section we conclude the proofs of Theorems~\ref{t.cycle} and \ref{t.ctail}.

\subsubsection{Proof of Theorem~\ref{t.cycle}}
\label{sss.proofoftheoremtcycle}
Consider   a $C^1$-open set $\cU\subset \Diff^1(M)$ consisting of diffeomorphisms $f$ 
 having hyperbolic periodic points $p_f$ and $q_f$, depending continuously on $f$, 
  of different indices
whose  chain recurrence classes $C(p_f,f)$ and $C(q_f,f)$ coincide and have  a partially hyperbolic splitting with one-dimensional center. 
This implies that the $\mru$-indices of $p_f$ and $q_f$ are
$j+1$ and $j$ for some $j$.

Consider  $f\in \cU$. We prove that there are a neighbourhood $\cV_f$ of $f$ and an open and dense
subset $\cZ_f$ of $\cV_f$ where the conclusion of the theorem holds
(every $g\in\cZ_f$ has a nonhyperbolic ergodic measure
$\mu_g$ 
 whose support
is $H(p_g,g)=H(q_g,g)$).  
 The theorem follows
considering the set $\cV=\bigcup_{f\in \cU} \cZ_f$ that is, by construction, open and dense in $\cU$.
Thus, in what follows, we fix $f\in \cU$ and study a local problem in a neighbourhood of $f$.

The partial hyperbolicity of $C(p_f,f)$ gives neighbourhoods $V$ of $C(p_f,f)$ 
and $\cV_f$ of $f$ such  that
for every $g\in \cV_g$  the  maximal invariant set of $g$ in $V$ 
\begin{equation}
\label{e.upsilon}
\Upsilon_g \eqdef \bigcap_{i\in \ZZ} g^i(V)
\end{equation}
has a partially hyperbolic splitting  with one-dimensional center.
Since chain recurrence classes depend upper semi-continuously, 
after shrinking $\cV_f$, if necessary, we can assume that  $C(p_g,g)\subset V$ for every $g\in \cV_f$. Thus $H(p_g,g)=H(q_g,g)\subset
C(p_g,g)\subset \Upsilon_g$ and these sets are partially hyperbolic with one-dimensional center.
Hence \cite[Theorem E]{BDPR} (see Remark~\ref{r.bdpr}) gives an open and dense subset
$\cW_f$ 
 of $\cV_f$ such that for every $f\in \cW_f$ the homoclinic classes of $p_f$ and $q_f$ are equal
 (here the hypothesis on the one-dimensional center is essential).
 To prove the 
theorem it is enough to get an open an dense subset $\cZ_f$ of $\cW_f$ (thus of $\cV_f$) 
consisting of diffeomorphisms $g$  having a nonhyperbolic ergodic  measure $\mu_{g}$ whose support is $H(p_g,g)=H(q_g,g)$.
 
By \cite{BD-cycles} there is an open and dense subset $\cC_f$ of $\cW_f$ 
such that every $g\in \cC_f$ 
has transitive hyperbolic sets $\Lambda^p_g$ and $\Lambda^q_g$, with 
 $p_g\in \Lambda^p_g$
and $q_g\in \Lambda^q_g$,
having a robust cycle (i.e.,  there is a neighbourhood of $g$ consisting of diffeomorphisms 
$h$ such that the invariant sets of $\Lambda^p_h\ni p_h$
and $\Lambda^q_h\ni q_h$ meet cyclically). 
Now \cite[Proposition 5.2]{BoBoDi2}
(Robust cycles yield spawners) and
\cite[Proposition 5.3]{BoBoDi2} 
(Spawners yield split flip-flop configurations)
gives an open and dense subset $\cZ_f$ of $\cC_f$  such that every $g\in\cZ_f$ has  a dynamical blender 
$\Ga_g\subset C(p_g,g)$
that is in a flip-flop configuration with $q_g$. Moreover, the  
blender $\Ga_g$ and the saddle $p_g$ are
\emph{homoclinically related} (their invariant manifolds intersect cyclically and transversely).
As in the case of a homoclinic relation between periodic points this implies that $\Ga_g \subset C(p_g,g)$.
 
For $g\in \cZ_f$ consider a continuous function $\mathrm{J}^c_g$ defined on 
the partially hyperbolic set
$\Upsilon_g$ 
as the
logarithm of the center derivative of $g$, recall \eqref{e.logmap}. Note that $\Gamma_g\cup  \cO(q_g)\subset \Upsilon_g$.
Up considering an adapted metric,   we can assume the map $\mathrm{J}^c_g$ is positive on the blender $\Ga_g$ and negative  
on the orbit of $q_g$.  We extend the map $\mathrm{J}^c_g$ to  a continuous map defined on the whole manifold
(with a slight abuse of notation, we denote this new map also by $\mathrm{J}^c_g$).
For diffeomorphisms $g\in \cZ_g$
Theorem~\ref{t.p.tail} gives an ergodic measure $\mu_g$ whose support is  $H(q_g,g)=H(p_g,g)$  and such 
that $\int \mathrm{J}^c_g\, d\mu_g=0$. As $\mu_g$ is supported on 
$ H(q_g,g)\subset \Upsilon_g$,
the function $\mathrm{J}^c_g$ coincides with the logarithm of the center
derivative of $g$ in the support of $\mu_g$. Thus  $\int \mathrm{J}^c_g\, d\mu_g=0$ is the 
center Lyapunov exponent of $\mu_g$. This concludes the proof of the theorem.  \qed

\subsubsection{Proof of Theorem~\ref{t.ctail}}
\label{sss.proofoftctail}
In the previous sections we dealt  with averages of continuous functions. For the analysis of Lyapunov
exponents let us recall that in 
\cite{BoBoDi2} the partial hyperbolicity of the set guarantees the continuity
of central (one-dimensional) derivatives in a locally maximal invariant set (these maps are continuously extended
to a neighbourhood of the set). Here we
we argue as in previous sections keeping in mind the following three facts: (1) The existence of a flip-flop
configuration is a hypothesis.  (2) The filtrating neighbourhood implies that it contains the homoclinic classes. (3) The existence of invariant  cone fields in the filtrating neighbourhood
gives the partial hyperbolicity with one-dimensional center of the maximal invariant set in $U$ and
thus of the homoclinic classes.

\section{Applications to robust transitive  diffeomorphisms}
\label{s.applications}

In this section we prove Theorem~\ref{t.c.openanddense}
and  Corollary~\ref{c.averages}.
Recall that $\cR\cT(M)$ is the  (open) subset of $\diff^1(M)$ 
of diffeomorphisms that are
 robustly transitive and
 have a pair  of hyperbolic periodic points of different indices
 and 
a partially hyperbolic  splitting with one-dimensional center.
We prove the following proposition:

\begin{prop}\label{p.l.h(q)} 
There is a $C^1$-open and dense subset $\cI(M)$ of $\cR\cT(M)$ such that for 
every $f\in \cI(M)$ there are hyperbolic periodic points 
$p_f$ and $q_f$ 
of different indices such that
$$
H(p_f,f)=H(q_f,f)=M.
$$ 
\end{prop}

In view of this proposition, Theorem~\ref{t.c.openanddense} is a direct consequence of Theorem~\ref{t.cycle} 
(in $\cR\cT(M)$ the unique chain recurrence class is the whole manifold $M$) and Corollary~\ref{c.averages} is a direct  consequence of  Corollary~\ref{c.function}.

\begin{proof}[Proof of Proposition~\ref{p.l.h(q)}]
The diffeomorphisms $f\in \cR\cT(M)$ have a  partially hyperbolic splitting with one-dimensional center
$TM=E^{\mathrm{uu}}\oplus E^{\mathrm{c}}\oplus E^{\mathrm{ss}}$,
where 
$E^{\mathrm{uu}}$ is uniformly expanding,
$E^{\mathrm{ss}}$ is uniformly contracting,
 and $\dim (E^{\mathrm{c}})=1$.
This implies that  there is $j\in \{1,\dots,\dim(M)-2\}$ such that
  every hyperbolic periodic point $p$  of $f$ has $\mathrm{s}$-index either $j$ or $j+1$, where $j=\dim (E^{\mathrm{ss}})$.

  We define 
  $\cR\cT_{j}(M)$ as the open subset of $\cR\cT(M)$ consisting of diffeomorphism whose saddles
  have $\mathrm{s}$-indices either $j$ or $j+1$. 
  The next lemma is a consequence of the ergodic closing lemma in \cite{Ma}, for an explicit formulation of this
  result see \cite[Theorem in page 4]{DPU}.
   
  \begin{lemm}
The set   $\bigcup_{j=1}^{\dim (M)-2}  \cR\cT_{j}(M)$ is open and dense in $\cR\cT(M)$.
  \end{lemm}

In view of this lemma the proposition is a consequence of the following result:

\begin{lemm}\label{l.l.h(q)} 
Let $j\in \{1,\dots,\dim(M)-2\}$.
There is an open an dense subset $\cI_j(M)$ of $\cR\cT_j(M)$ such that  
every $f\in \cI_j(M)$ has hyperbolic periodic points  $p_f$ and $q_f$ of different indices such that
$$
H(p_f,f)=H(q_f,f)=M.
$$ 
\end{lemm}

\begin{proof}
For the diffeomorphisms $f\in \cR\cT_j(M)$ there are defined the strong stable foliation
$\cF^{\mathrm{ss}}_f$  of dimension $j$ and the strong unstable foliation  $\cF^{\mathrm{uu}}_f$ 
of dimension $\dim(M)-j-1$.
Recall that $\cF^{\mathrm{ii}}_f$, $\mathrm{i}=\mathrm{s}, \,\mathrm{u}$, is the only $Df$-invariant 
foliation of dimension $\dim (E^{\mathrm{ii}})$ 
 tangent to $E^{\mathrm{ii}}$.

 The foliation $\cF^{\mathrm{ii}}_f$ is \emph{minimal} if every leaf $F^{\mathrm{ii}}_f(x)$ of
 $\cF^{\mathrm{ii}}_f$ is dense in $M$. The  foliation $\cF^{\mathrm{ii}}_f$ is \emph{$C^1$-robustly minimal}
 if there is a $C^1$-neighbourhood $\cV_f$ of $f$ such that for every $g\in \cV_f$ the foliation
 $\cF^{\mathrm{ii}}_g$ is minimal.
 We denote by $\cM^{\mathrm{i}}_j(M)$, $i=\mathrm{s,u}$, the open subset of $\cR\cT_j(M)$ of diffeomorphisms such that
 $\cF^{\mathrm{ii}}_f$ is robustly minimal. 
 Let
  $$ 
  \cM_j(M) \eqdef \cM^{\mathrm{s}}_j(M) \cup  \cM^{\mathrm{u}}_j(M).
  $$

\begin{lemm}[\cite{BDU,RHU}]
\label{l.minimal}
 The set $\cM_j(M)$ is open and dense in $\cR\cT_j(M)$.
\end{lemm}

We need the following property.

\begin{clai}$\,$
\label{cl.sacocheio}
\begin{itemize}
\item
Let  $f\in \cM^{\mathrm{u}}_{j}(M)$. Then $H(q,f)=M$ for every saddle 
$q$ of $\mathrm{s}$-index $j+1$.
\item 
Let  $f\in \cM^{\mathrm{s}}_{j}(M)$. Then $H(q,f)=M$ for every saddle 
$q$ of $\mathrm{s}$-index $j$. 
\end{itemize}
\end{clai}

\begin{proof}
We prove the first item, the second one is analogous and thus omitted.
Fix any hyperbolic periodic point  $q$ of $\mathrm{s}$-index $j+1$. Then the unstable manifold of $q$ is a leaf
of $\cF_f^{\mathrm{uu}}$, hence  it is dense in $M$.

The minimality of $\cF_f^{\mathrm{uu}}$ and the fact that $W^{\mathrm{s}}(q,f)$ contains a disc of dimension $j+1$ transverse
to $\cF_f^{\mathrm{uu}}$  imply that there is $K>0$ such that 
$W^{\mathrm{s}}(q,f)$ intersects transversely every strong unstable disc of radius larger than $K$.

Take now any point $x\in M$ and any $\epsilon>0$. 
We see that the ball $B_\epsilon (x)$ intersects $H(q,f)$. Since this holds for any $x\in M$ and $\epsilon>0$
and $H(q,f)$ is closed this implies $H(q,f)=M$.

The density of $W^{\mathrm{u}}(q,f)$ implies that there is
a disc  $\Delta\subset W^{\mathrm{u}}(q,f)$ of dimension $\dim (E^{\mathrm{uu}})$ contained in $B_\epsilon(x)$. Since $Df$ expands the vectors tangent to $E^{\mathrm{uu}}$
there is $n>0$ such that $f^n(\Delta)$ has radius at least $K$. Thus $f^n(\Delta)$ meets transversely $W^{\mathrm{s}}(q,f)$.
Thus $\Delta$ contains a point of the homoclinic class of $q$. This implies the claim.
\end{proof}

By \cite[Theorem E]{BDPR} (see also Remark~\ref{r.bdpr}) in this partially hyperbolic setting with one-dimensional center, 
there is an open and dense subset $\cP_j(M)$ of $\cR\cT_j(M)$ such that 
for every pair of saddles $p_f$ and $q_f$ of $f$ it holds $H(p_f,f)=H(q_f,f)$. Note that this claim is
only relevant when the saddles have different indices.

In view of Claim~\ref{cl.sacocheio} to prove Lemma~\ref{l.l.h(q)} it is enough to take 
$$
\cI_j(M)=\cP_j(M)\cap \cM_j(M)
$$
that is open and dense in $\cR\cT_j(M)$ (recall Lemma~\ref{l.minimal}).
\end{proof}
The proof of  Proposition~\ref{p.l.h(q)}  is now complete.
\end{proof}



\end{document}